\documentclass[11pt]{article}
\usepackage{amssymb}
\usepackage{amsmath}
\usepackage{amsthm}
\usepackage[english]{babel}
\usepackage{MnSymbol}
\usepackage{graphicx}
\newtheorem{theorem}{Theorem}[section]
\newtheorem{lemma}[theorem]{Lemma}
\newtheorem{proposition}[theorem]{Proposition}

\newtheorem{example}[theorem]{Example}
\newtheorem{question}[theorem]{Question}

{\theoremstyle{definition}}
{\theoremstyle{definition}}
{\theoremstyle{definition}\newtheorem{remark}[theorem]{Remark}}

\newtheorem*{que1}{Question~IS}

\numberwithin{equation}{section}

\def\C{{\mathbb C}}
\def\N{{\mathbb N}}
\def\Z{{\mathbb Z}}

\def\K{{\mathbb K}}

\def\kappa{\varkappa}
\def\phi{\varphi}
\def\leq{\leqslant}
\def\geq{\geqslant}

\def\dim{{\rm dim}\,}

\def\spann{\hbox{\tt span}\,}

\def\deg{\hbox{\tt deg}\,}
\def\pd#1#2{\frac{{\partial\,} #1}{{\partial\,} #2}}
\def\cial#1{{\smash{#1}}^\rcirclearrowleft\!\!\!\!\!^{\scriptscriptstyle\alpha}}
\def\tabb#1#2#3#4#5#6#7{\noindent{\scalebox{0.7}{\hbox{\vrule\,{\makebox[0.75cm][l]{\,#1}\,\vrule\,
\makebox[5.5cm][l]{#2}\,\vrule\,\makebox[4.3cm][l]{#3}\,\vrule\,\makebox[2.7cm][l]{#4}\,\vrule\,
\makebox[5cm][l]{#5}\,\vrule\,\makebox[3cm][l]{#6}\,\vrule\,\makebox[1.7cm][l]{#7}\,\vrule}}}}\hrule}
\def\hhhhh{\vrule height12pt depth5pt width0pt}
\def\tabbb#1#2#3#4#5#6#7{\noindent{\scalebox{0.7}{\hbox{\vrule\,{\makebox[0.75cm][l]{\,#1}\,\vrule\,
\makebox[4.7cm][l]{#2}\,\vrule\,\makebox[5.3cm][l]{#3}\,\vrule\,\makebox[3cm][l]{#4}\,\vrule\,
\makebox[5cm][l]{#5}\,\vrule\,\makebox[3cm][l]{#6}\,\vrule\,\makebox[1.2cm][l]{#7}\,\vrule}}}}\hrule}
\def\tabbbb#1#2#3#4#5#6{\noindent{\scalebox{0.7}{\hbox{\vrule\,{\makebox[0.75cm][l]{\,#1}\,\vrule\,
\makebox[8.4cm][l]{#2}\,\vrule\,\makebox[8cm][l]{#3}\,\vrule\,\makebox[2cm][l]{#4}\,\vrule\,
\makebox[3cm][l]{#5}\,\vrule\,\makebox[1cm][l]{#6}\,\vrule}}}}\hrule}
\def\hhh{\vrule height10pt depth5pt width0pt}
\def\hhhh{\vrule height12pt depth7pt width0pt}
\def\taBB#1#2#3#4#5#6#7{\noindent{\scalebox{0.62}{\hbox{\vrule\,{\makebox[0.75cm][l]{\,#1}\,\vrule\,
\makebox[6cm][l]{#2}\,\vrule\,\makebox[6.95cm][l]{#3}\,\vrule\,\makebox[2.7cm][l]{#4}\,\vrule\,
\makebox[5cm][l]{#5}\,\vrule\,\makebox[3cm][l]{#6}\,\vrule\,\makebox[1.7cm][l]{#7}\,\vrule}}}}\hrule}

\title{Potential algebras with few generators}

\author{Natalia Iyudu and Stanislav Shkarin}

\date{}

\begin{document}

\maketitle

\begin{abstract}We give a complete description of quadratic potential and twisted potential algebras on 3 generators as well as cubic potential and twisted potential algebras on 2 generators up to graded algebra isomorphisms under the assumption that the ground field is algebraically closed and has characteristic different from $2$ or $3$.

 We also prove that for two generated potential algebra necessary condition of finite-dimensionality is that potential contains terms of degree three, this
answers a question of Agata Smoktunowicz and the first named author, formulated in \cite{AN}. We clarify situation in case of arbitrary number of generators as well.
\end{abstract}



\section{Introduction \label{s1}}\rm

Throughout this paper $\K$ is an algebraically closed field of characteristic different from $2$ or $3$.
 If $B$ is a $\Z_+$-graded vector space, $B_m$ always stands for the $m^{\rm th}$ component of $B$. We only deal with the situation when each $B_m$ is finite dimensional, which allows to consider
 the polynomial generating function of the sequence of dimensions of graded components, called

$$
\textstyle \text{the {\it Hilbert series} of $B$:}\qquad H_B(t)=\sum\limits_{j=0}^\infty \dim B_m\,\,t^m.
$$

The classical potential algebras are defined as $\K[x_1,\dots,x_n]/I_L$, where $I_L$ is the ideal generated by all first order partial derivatives $\frac{\partial L}{\partial x_j}$ of $L\in \K[x_1,\dots,x_n]$, called the potential. Potential algebras have been defined in the non-commutative setting by Kontsevich \cite{kon}, see also \cite{BW} (an alternative equivalent definition was suggested by Ginsburg \cite{xx}). An element $F\in \K\langle x_1,\dots,x_n\rangle$ is called {\it cyclicly invariant} if it is invariant for the linear map $C:\K\langle x_1,\dots,x_n\rangle\to \K\langle x_1,\dots,x_n\rangle$ defined on monomials by $C(1)=1$ and $C(x_ju)=ux_j$ for all $j$ and all monomials $u$. For example, $x^2y+xyx+yx^2$ and $x^3$ are cyclicly invariant, while $xy-yx$ is not. The symbol $\K^{\rm cyc}\langle x_1,\dots,x_n\rangle$ stands for the vector space of all cyclicly invariant elements of $\K\langle x_1,\dots,x_n\rangle$. We define noncommutative left or right (respectively) derivatives as linear maps $\delta_{x_j}:\K\langle x_1,\dots,x_n\rangle\to \K\langle x_1,\dots,x_n\rangle$ and $\delta^R_{x_j}:\K\langle x_1,\dots,x_n\rangle\to \K\langle x_1,\dots,x_n\rangle$ by their action on monomials:
$$
\delta_{x_j}\,u=\left\{\begin{array}{ll}v&\text{if}\ \ u=x_jv;\\ 0&\text{otherwise},\end{array}\right.\qquad
\delta^R_{x_j}\,u=\left\{\begin{array}{ll}v&\text{if}\ \ u=vx_j;\\ 0&\text{otherwise}.\end{array}\right.
$$
Note that
$$
\text{$F\in\K\langle x_1,\dots,x_n\rangle$ is cyclicly invariant if and only if $\delta_{x_j}\,F=\delta^R_{x_j}\,F$ for $1\leq j\leq n$}.
$$
For $F\in \K^{\rm cyc}\langle x_1,\dots,x_n\rangle$, the {\it potential algebra} $A_F$ is defined as $\K\langle x_1,\dots,x_n\rangle/I$, where $I$ is the ideal generated by $\delta_{x_j}F$ for $1\leq j\leq n$. We shall call $F$ the {\it potential} for $A_F$.

Note that if the characteristic of $\K$ is either 0 or is greater than the top degree of non-zero homogeneous components of $F\in\K^{\rm cyc}\langle x_1,\dots,x_n\rangle$, then  $F=G^\rcirclearrowleft$ for some (non-unique) $G\in\K\langle x_1,\dots,x_n\rangle$, where the linear map $G\mapsto G^\rcirclearrowleft$ from $\K\langle x_1,\dots,x_n\rangle$ to $\K^{\rm cyc}\langle x_1,\dots,x_n\rangle$ is defined by its action on homogeneous elements by
$$
\text{$u^\rcirclearrowleft=u+Cu+{\dots}+C^{d-1}u$, where $d$ is the degree of $u$.}
$$
For example, ${x^4}^\rcirclearrowleft=4x^4$ and ${x^2y}^\rcirclearrowleft=x^2y+xyx+yx^2$. One easily sees that the usual partial derivatives $\pd{G^{\rm ab}}{x_j}$ of the abelianization $G^{\rm ab}$ of $G\in \K\langle x_1,\dots,x_n\rangle$ ($G^{\rm ab}$ is the image of $G$ under the canonical map from $\K\langle x_1,\dots,x_n\rangle$ onto $\K[x_1,\dots,x_n]$) are the abelianizations of $\delta_{x_j}(G^\rcirclearrowleft)$. Thus commutative potential algebras are exactly the abelianizations of the non-commutative ones. The above definitions and observations immediately yield the following lemma.

\begin{lemma}\label{gen1} For every $F\in\K\langle x_1,\dots,x_n\rangle$ with trivial zero degree component $(F_0=0),$
\begin{equation}\notag
\smash{F=\sum_{j=1}^n x_j(\delta_{x_j}F)=\sum_{j=1}^n (\delta^R_{x_j}F)x_j.}
\end{equation}
Thus $F$ is cyclicly invariant if and only if $\smash{F=\sum\limits_{j=1}^n (\delta_{x_j}F)x_j}$. In particular,
\begin{align}\notag
&F=\sum_{j=1}^n x_j(\delta_{x_j}F)=\sum_{j=1}^n(\delta_{x_j}F)x_j\ \ \text{for every $F\in\K^{\rm cyc}\langle x_1,\dots,x_n\rangle$ with $F_0=0,$}
\\
&\smash{\sum_{j=1}^n \bigl[x_j,\delta_{x_j}F\bigr]=0\ \ \text{for every $F\in\K^{\rm cyc}\langle x_1,\dots,x_n\rangle$}}. \label{syz}
\end{align}
\end{lemma}

We consider a larger class of algebras. We call $F\in\K\langle x_1,\dots,x_n\rangle$ a {\it twisted potential} if  the linear span of $\delta_{x_1}F,\dots,\delta_{x_n}F$ coincides with the linear span of $\delta^R_{x_1}F,\dots,\delta^R_{x_n}F$. Just as for potentials,  if $F$ is a twisted potential, the corresponding {\it twisted potential algebra} $A_F$ is given by generators $x_1,\dots,x_n$ and relations $\delta_{x_1}F,\dots,\delta_{x_n}F$. Clearly, the same algebra is presented by the relations $\delta^R_{x_1}F,\dots,\delta^R_{x_n}F$.
Note that there is a number of other generalizations of the concept of a potential algebra. For instance, one can replace the free algebra in the above definition by a (directed) graph algebra \cite{BW}. Our definition then corresponds to the case of the $n$-petal rose (one vertex with $n$ loops) graph.

There is a complex of right $A$-modules associated to each twisted potential algebra $A=A_F$ with $F_0=F_1=0$ ($F$ starts in degree $\geq2$). Namely, we consider the sequence of right $A$-modules:
\begin{equation}\label{comq}
\begin{array}{c}
0\to A\mathop{\longrightarrow}\limits^{d_3}A^n\mathop{\longrightarrow}\limits^{d_2}A^n
\mathop{\longrightarrow}\limits^{d_1}A\mathop{\longrightarrow}\limits^{d_0}\K\to 0,\qquad\text{where}\ \ d_2(u_1,\dots,u_n)_j=\sum\limits_{k=1}^n (\delta_{x_j}\delta^R_{x_k} F)u_k,
\\
\text{$d_0$ is the augmentation map, $d_1(u_1,\dots,u_n)=x_1u_1+{\dots}+x_nu_n$ and $d_3(u)=(x_1u,\dots,x_nu)$.}
\end{array}
\end{equation}
We say that the twisted potential algebra $A=A$ is {\it exact} if (\ref{comq}) is an exact complex. For the sake of completeness, we shall verify in Section~3 that (\ref{comq}) is indeed a complex and that it is always exact at its three rightmost terms. Obviously, exactness of (\ref{comq}) is preserved under linear substitutions and therefore is an isomorphism invariant as long as degree-graded twisted potential algebras are concerned. Note also that the {\it superpotential} algebras \cite{BW} are also particular cases of twisted potential algebras: for them $\delta^R_{x_j}F=\pm \delta_{x_j}F$.

\begin{remark}\label{rem0}
It is an elementary linear algebra exercise to verify that if $F\in\K\langle x_1,\dots,x_n\rangle$ is a twisted potential for which the dimension of the linear span of $\delta_{x_1}F,\dots,\delta_{x_n}F$ is $m<n$, then one can choose an $m$-dimensional subspace $M$ in $V=\spann\{x_1,\dots,x_n\}$ such that $F$ belongs to the tensor algebra of $M$. In other words, there is a basis $y_1,\dots,y_n$ in $V$ such that only $y_1,\dots,y_m$ feature in $F$ when written in terms of $y_1,\dots,y_n$. Thus we have the twisted potential algebra $B$ with generators $y_1,\dots,y_m$ and twisted potential $F$, while the original $A_F$ is the free product of $B$ and the free $\K$-algebra on $n-m$ generators. One easily sees that such an $A_F$ is never exact. Moreover $A_F$ is Koszul or PBW or a domain \cite{popo} if and only if $B$ is of the same type. Finally, if $F$ is homogeneous, the Hilbert series of $A_F$ and $B$ are related by $H_{A_F}(t)=(H_B(t)^{-1}-(n-m)t)^{-1}$.
\end{remark}

We say that a twisted potential $F\in\K\langle x_1,\dots,x_n\rangle$ is {\it non-degenerate} if $\delta_{x_1}F,\dots,\delta_{x_n}F$ are linearly independent. According to the above remark, in order to describe all twisted potential algebras with $n$ generators, it is enough to describe non-degenerate twisted potential algebras with $\leq n$ generators.

Note that if $F\in\K\langle x_1,\dots,x_n\rangle$ is a non-degenerate twisted potential, then there is a unique matrix $M\in GL_n(\K)$ such that
\begin{equation}\label{qua}
\left(\begin{array}{c}\delta_{x_1}^R\,F\\ \vdots\\ \delta_{x_n}^R\,F\end{array}\right)
=M\left(\begin{array}{c}\delta_{x_1}\,F\\ \vdots\\ \delta_{x_n}\,F\end{array}\right).
\end{equation}
We say that $M$ {\it provides the twist} or {\it is the twist}. By Lemma~\ref{gen1}, cyclic invariance happens precisely when (\ref{qua}) is satisfied with $M$ being the identity matrix. That is,  every non-degenerate potential $F\in\K^{\rm cyc}\langle x_1,\dots,x_n\rangle$ is a non-degenerate twisted potential with trivial twist. Note that the definition of non-degenerate twisted potential algebras is very similar to that of algebras defined by multilinear forms of Dubois-Violette \cite{DV1,DV2}. In fact, our definition generalizes the latter.

\begin{remark}\label{rer}
Assume that $F$ is a non-degenerate twisted potential with the twist $M\in GL_n(\K)$. If we perform a non-degenerate linear substitution $x_j=\sum_{k} c_{j,k}y_k$, then in the new variables $y_j$, $F$ remains a non-degenerate twisted potential. Furthermore, the corresponding twist changes in a very specific way: the new twist is the conjugate of $M$ by the transpose of the substitution matrix $C$. We leave this elementary calculation for the reader to verify. One useful consequence of this observation is that by means of a linear substitution, $M$ can be replaced by a convenient conjugate matrix. For instance, $M$ can be transformed into its Jordan normal form.
\end{remark}

We say that a twisted potential $F\in\K\langle x_1,\dots,x_n\rangle$ is {\it proper} if the equality
\begin{equation}\label{syz1}
\sum_{j=1}^n x_j(\delta_{x_j}F)=\sum_{j=1}^n (\delta^R_{x_j}F)x_j
\end{equation}
of Lemma~\ref{gen1} provides the only linear dependence of the $2n^2$  elements $x_k(\delta_{x_j}F)$ and $(\delta^R_{x_j}F)x_k$ with $1\leq j,k\leq n$ of $\K\langle x_1,\dots,x_n\rangle$ up to a scalar multiple. Note that in this case $\delta_{x_j}F$ are automatically linearly independent and therefore $F$ is non-degenerate.

\begin{lemma}\label{gen2} Let $F\in\K\langle x_1,\dots,x_n\rangle$ be a homogeneous twisted potential of degree $k\geq 3$ and $A=A_F$ be the corresponding twisted potential algebra. Then  $\dim A_{k}\geq n^k-2n^2+1$. Moreover, $F$ is non-degenerate if an only if $\dim A_{k-1}=n^{k-1}-n$ and $F$ is proper if and only if $\dim A_{k}= n^k-2n^2+1$. Furthermore, if $F$ is proper, then $F$ is uniquely determined by $A_F$ up to a scalar multiple and any linear substitution providing a graded algebra isomorhpism between $A_F$ and another twisted potential algebra $A_G$ must  transform $F$ to $G$ up to a scalar multiple.
\end{lemma}

\begin{proof}Let $V$ be the linear span of $x_j$ for $1\leq j\leq n$, $R_F$ be the linear span of $\delta_{x_j}F$ for $1\leq j\leq n$ and $I$ be the ideal of relations for $A$: $I$ is the ideal in $\K\langle x_1,\dots,x_n\rangle$ generated by $R_F$. Obviously, $F$ is non-degenerate if and only if $\dim R_F=n$ if and only if $\dim A_{k-1}=n^{k-1}-n$. Clearly $I_k$ is spanned by $2n^2$ elements $x_j\delta_{x_m}F$ and $\delta_{x_m}Fx_j$ for $1\leq j,m\leq n$. The equation (\ref{syz1}) provides a non-trivial linear dependence of these elements. Hence $\dim I_k\leq 2n^2-1$ and therefore $\dim A_{k}\geq n^k-2n^2+1$. Clearly, the equality $\dim A_{k}= n^k-2n^2+1$ holds if and only if there is no linear dependence of $x_j\delta_{x_m}F$ and $\delta_{x_m}Fx_j$ other than (\ref{syz1}) (up to a scalar multiple). That is, $F$ is proper if and only if $\dim A_{k}= n^k-2n^2+1$.

Now let $F$ be proper. Then  $\dim I_k=\dim(VR_F+R_FV)=2n^2-1$.  By Lemma~\ref{gen1}, $F\in VR_F\cap R_FV$. Since $\delta_{x_j}F$ are linearly independent,  $\dim VR_F=\dim R_FV=n^2$. Hence $\dim (VR_F\cap R_FV)=1$. Thus $VR_F\cap R_FV$ is the one-dimensional space spanned by $F$. It follows that $F$ is uniquely determined by $A$ up to a scalar multiple. If, additionally, a linear substitution provides an isomorphism between $A$ and another twisted potential algebra $A_G$, then the said substitution must transform $VR_F\cap R_FV$ to $VR_G\cap R_GV$. Since the first of these spaces is the one-dimensional space spanned by $F$ and the second contains $G$, it must also be one-dimensional and must be spanned by $G$. Hence our substitution transforms $F$ into $G$ up to a scalar multiple.
\end{proof}

\begin{remark}\label{proper}
We stress that any homogeneous twisted potential, when proper, is uniquely (up to a scalar multiple) determined by the corresponding twisted potential algebra. We dub such algebras {\it proper twisted potential algebras}. By the above lemma, proper and non-proper degree-graded twisted potential algebras can not be isomorphic. Similarly, we say that a twisted potential algebra is {\it degenerate} if it is given by a degenerate twisted potential. Again, the concept is well-defined and a non-degenerate degree-graded twisted potential algebra can not be isomorphic to a degenerate one. Note that we are talking of isomorphisms in the category of graded algebras (=isomorphisms provided by linear substitutions). As we have already mentioned, a proper degree-graded twisted potential algebra is always non-degenerate. We shall see later that every exact degree-graded twisted potential algebra is proper.
\end{remark}

The main objective of this paper is to provide a complete classification up to graded algebra isomorphisms of twisted potential algebras in two cases: when the twisted potential is a homogeneous (non-commutative) polynomial of degree 3 on three variables and when it is a homogenenous (non-commutative) polynomial of degree 4 on two variables. This task resonates with the Artin--Schelter classification result \cite{AS}: many algebras we deal with are indeed Artin--Schelter regular. However there are two differences. For one, the classes are not exactly the same. The main difference though is that Artin and Schelter have never provided a classification up to an isomorphism.

For the sake of convenience, we introduce the following notation. For integers $n,k$ satisfying $n\geq 2$ and $k\geq 3$ and $M\in GL_n(\K)$,
\begin{equation}\label{pnkm}
\begin{array}{c}
\text{${\cal P}_{n,k}(M)$ is the set of homogeneous degree $k$ elements}
\\
\text{$F\in \K\langle x_1,\dots,x_n\rangle$ for which (\ref{qua}) is satisfied.}
\end{array}
\end{equation}
Obviously, ${\cal P}_{n,k}(M)$ is a vector space. However, ${\cal P}_{n,k}(M)$ is often trivial. For instance, it is trivial if the eigenvalues of $M$ are algebraically independent over the subfield of $\K$ generated by $1$. We denote
\begin{equation}\label{pnk}
{\cal P}_{n,k}={\cal P}_{n,k}({\rm Id}).
\end{equation}
In other words, ${\cal P}_{n,k}$ consists of homogeneous degree $k$ elements of $\K^{\rm cyc}\langle x_1,\dots,x_n\rangle$. Finally,
\begin{equation}\label{pnkmst}
\text{${\cal P}^*_{n,k}$ is the set of all homogeneous degree $k$ twisted potentials in $\K\langle x_1,\dots,x_n\rangle$.}
\end{equation}

In case of cubic twisted potentials, we deal with quadratic algebras and together with classification we provide the information whether algebras in question are Koszul and/or PBW. The latter, as defined in \cite{popo}, is the property to have a Gr\"obner basis in the ideal of relations (with respect to some compatible ordering and some choice of degree 1 generators) consisting exclusively of quadratic elements. The results are presented in tables. The first column provides a label for further references. The letter $P$ in the label indicates that we have a potential algebra, while the letter $T$ indicates that the algebra is twisted potential and non-potential. The exceptions column says which values of the parameters are excluded. The isomorphism column provides generators of a group action on the space of parameters such that corresponding algebras are isomorphic precisely when the parameters are in the same orbit. The Koszul/PBW/Exact column says whether algebras in question are Koszul or PBW or exact. For instance, the Y/N/Y entry means that the algebra is Koszul, exact but not PBW. We introduce some notation for the rest of the paper. Let
$$
\xi_8\ \ \text{and}\ \ \xi_9\ \ \text{be fixed elements of $\K^*$ of multiplicative orders $8$ and $9$ respectively.}
$$
Note that such elements exist since $\K$ is algebraically closed and has characteristic different from $2$ or $3$. We also denote
$$
\theta=\xi_9^3\ \ \text{and}\ \ i=\xi_8^2.
$$
Obviously,
$$
\theta^3=1\neq\theta\ \ \ \text{and}\ \ \ i^2=-1.
$$

\begin{theorem}\label{main33-p} $A$ is a potential algebra on three generators given by a homogeneous degree $3$ potential if and only if $A$ is isomorphic $($as a graded algebra$)$ to an algebra from the following table. The algebras from different rows of the table are non-isomorphic. Algebras from {\rm (P1--P9)} are proper, algebras from {\rm (P10--P14)} are non-proper and non-degenerate, while algebras from {\rm (P15--P18)} are degenerate.
\bigskip
\hrule
\tabb{\hhh}{\rm The potential $F$}{$\!\!\!\begin{array}l \text{\rm Defining}\\ \text{\rm Relations of $A_F$}\end{array}$}{\rm Exceptions}{\rm Isomorphisms}{\rm Hilbert series}{$\!\!\!\begin{array}l \text{\rm Koszul/}\\ \text{\rm PBW/}\\ \text{\rm Exact}\end{array}$}
\tabb{\rm P1}{$x^3+y^3+z^3+axyz^\rcirclearrowleft+bxzy^\rcirclearrowleft$}{$\begin{array}l xx+ayz+bzy;\\ yy+azx+bxz;\\ zz+axy+byx\end{array}$}
{$\begin{array}l (a,b)\neq(0,0)\\ (a^3,b^3)\neq(1,1)\\ (a+b)^3+1\neq 0\end{array}$}{$\begin{array}l (a,b)\mapsto (\theta a,\theta b)\\ (a,b)\mapsto \bigl(\frac{\theta a+\theta^2b+1}{a+b+1},\frac{\theta^2 a+\theta b+1}{a+b+1}\bigr)\end{array}$}{$(1-t)^{-3}$}{\rm Y/N/Y}
\tabb{\rm P2}{$xyz^\rcirclearrowleft+axzy^\rcirclearrowleft$}{$\begin{array}l yz+azy;\\ zx+axz;\\ xy+ayx\end{array}$}
{$a\neq0$}{$a\mapsto a^{-1}$}{$(1-t)^{-3}$}{\rm Y/Y/Y}
\tabb{\rm P3}{$(y+z)^3+xyz^\rcirclearrowleft+axzy^\rcirclearrowleft$}{$\begin{array}l yz+azy;\\ axz+zx+(y+z)^2;\\ xy+ayx+(y+z)^2\end{array}$}
{$a\neq 0$, $a\neq -1$}{$a\mapsto a^{-1}$}{$(1-t)^{-3}$}{\rm Y/Y/Y}
\tabb{\rm P4}{$z^3+xyz^\rcirclearrowleft+axzy^\rcirclearrowleft$}{$\begin{array}l yz+azy;\\ axz+zx;\\ xy+ayx+zz\end{array}$}
{$a\neq 0$}{$a\mapsto a^{-1}$}{$(1-t)^{-3}$}{\rm Y/Y/Y}
\tabb{\rm P5}{$y^3+{xz^2}^\rcirclearrowleft+xyz^\rcirclearrowleft-xzy^\rcirclearrowleft$}{$\begin{array}l yz-zy+zz;\\ -xz+zx+yy;\\ xy-yx+xz+zx\end{array}$}{\rm none}{\rm trivial}{$(1-t)^{-3}$}{\rm Y/Y/Y}
\tabb{\rm P6}{${xz^2}^\rcirclearrowleft+{y^2z}^\rcirclearrowleft+xyz^\rcirclearrowleft-xzy^\rcirclearrowleft$}{$\begin{array}l yz-zy+zz;\\ -xz+zx+yz+zy;\\ xy-yx+xz+zx+yy\end{array}$}{\rm none}{\rm trivial}{$(1-t)^{-3}$}{\rm Y/Y/Y}
\tabb{\rm P7}{$y^3+z^3+xyz^\rcirclearrowleft-xzy^\rcirclearrowleft$}{$\begin{array}l yz-zy;\\ -xz+zx+yy;\\ xy-yx+zz\end{array}$}{\rm none}{\rm trivial}{$(1-t)^{-3}$}{\rm Y/Y/Y}
\tabb{\rm P8}{${yz^2}^\rcirclearrowleft+xyz^\rcirclearrowleft-xzy^\rcirclearrowleft$}{$\begin{array}l yz-zy;\\ -xz+zx+zz;\\ xy-yx+yz+zy\end{array}$}{\rm none}{\rm trivial}{$(1-t)^{-3}$}{\rm Y/Y/Y}
\tabb{\rm P9}{$(y+z)^3+xyz^\rcirclearrowleft$}{$\begin{array}l yz;\\ zx+(y+z)^2;\\ xy+(y+z)^2\end{array}$}
{\rm none}{\rm trivial}{$\frac{(1+t)(1+t^2)(1+t+t^2)}{1-t-t^3-2t^4}$}{\rm N/N/N}
\tabb{\rm P10}{${xz^2}^\rcirclearrowleft+y^3$}{$\begin{array}l zz;\\ yy;\\ xz+zx\end{array}$}{\rm none}{\rm trivial}{$\frac{1+t}{1-2t}$}{\rm Y/Y/N}
\tabb{\rm P11}{$x^3+y^3+z^3$}{$\begin{array}l xx;\\ yy;\\ zz\end{array}$}{\rm none}{\rm trivial}{$\frac{1+t}{1-2t}$}{\rm Y/Y/N}
\tabb{\rm P12}{$xyz^\rcirclearrowleft$}{$\begin{array}l yz;\\ zx;\\ xy\end{array}$}{\rm none}{\rm trivial}{$\frac{1+t}{1-2t}$}{\rm Y/Y/N}
\tabb{\rm P13}{${xz^2}^\rcirclearrowleft+{y^2z}^\rcirclearrowleft$}{$\begin{array}l zz;\\ yz+zy;\\ xz+zx+yy\end{array}$}{\rm none}{\rm trivial}{$\frac{1+t}{1-2t}$}{\rm Y/Y/N}
\tabb{\rm P14}{$z^3+xyz^\rcirclearrowleft$}{$\begin{array}l yz;\\ zx;\\ xy+zz\end{array}$}
{\rm none}{\rm trivial}{$\frac{1+t+t^2+t^3+t^4}{1-2t+t^2-t^3-t^4}$}{\rm N/N/N}
\tabb{\rm P15}{$y^3+z^3$}{$\begin{array}l yy;\\ zz\end{array}$}{\rm none}{\rm trivial}{$\frac{1+t}{1-2t-t^2}$}{\rm Y/Y/N}
\tabb{\rm P16}{${yz^2}^\rcirclearrowleft$}{$\begin{array}l zz;\\ yz+zy\end{array}$}{\rm none}{\rm trivial}{$\frac{1+t}{1-2t-t^2}$}{\rm Y/Y/N}
\tabb{\rm P17\hhhh}{$z^3$}{$\,z^2$}{\rm none}{\rm trivial}{$\frac{1+t}{1-2t-2t^2}$}{\rm Y/Y/N}
\tabb{\rm P18\hhhh}{$0$}{\rm \,none}{\rm none}{\rm trivial}{$(1-3t)^{-1}$}{\rm Y/Y/N}
\end{theorem}

\vfill\break

\begin{theorem}\label{main33-t} $A$ is a non-potential twisted potential algebra on three generators given by a homogeneous degree $3$ twisted potential if and only if  $A$ is isomorphic $($as a graded algebra$)$ to an algebra from the following table. The algebras from different rows of the table are non-isomorphic.
\bigskip
\hrule
\taBB{\hhh}{\rm Twisted potential $F$}{\rm Defining Relations of $A_F$}{\rm Exceptions}{\rm Isomorphisms}{\rm Hilbert series}{$\!\!\!\begin{array}l \text{\rm Koszul/}\\ \text{\rm PBW/}\\ \text{\rm Exact}\end{array}$}
\taBB{\rm T1}{$\begin{array}{l}bxyz +ayzx+czxy\\ -abyxz-bcxzy-aczyx\end{array}$}{$\begin{array}l xy-ayx;\\ zx-bxz;\\ yz-czy\end{array}$}{$\begin{array}{l}abc\neq0\\
\left(\!\!\begin{array}{l} a{-}b\\ a{-}c\end{array}\!\!\right)\neq \left(\!\!\begin{array}{l} 0\\ 0\end{array}\!\!\right)\end{array}$}{$\begin{array}l (a,b,c)\mapsto (b,c,a)\\ (a,b,c)\mapsto(a^{-1},c^{-1},b^{-1})\end{array}$}{$(1-t)^{-3}$}{\rm Y/Y/Y}
\taBB{\rm T2}{$\begin{array}{l}axyz +byzx+azxy\\ -abyxz-a^2xzy-abzyx-az^3\end{array}$}{$\begin{array}l xy-byx-zz;\\ zx-axz;\\ yz-azy\end{array}$}{$\begin{array}{l}ab\neq0\\ a\neq b\end{array}$}{$(a,b)\mapsto(a^{-1},b^{-1})$}{$(1-t)^{-3}$}{\rm Y/Y/Y}
\taBB{\rm T3}{$\!\!\begin{array}{l} xzy^\rcirclearrowleft{-}xyz^\rcirclearrowleft{+}a(xz^2{+}z^2x{+}z^2y)
\\+\frac{1-a}{2}(y^2z{+}zy^2{-}2zxz{-}zyz){-}\frac{1{+}a}{2}yzy\end{array}$}{$\!\!\begin{array}l yz-zy-azz;\\ xz-zx-azy+\frac{a(1-a)}{2}zz;\\
xy{-}yx{+}(1{-}2a)zx{+}\frac{a{-}1}2 yy\\ \qquad\qquad\quad{+}\frac{(1{+}a)(1{-}2a)}{4}zy{+}\frac{a^2(1{-}a)}{2}zz {\vrule height0pt depth7pt width0pt} \end{array}$}{$a\neq\frac13$}{\rm trivial}{$(1-t)^{-3}$}{\rm Y/Y/Y}
\taBB{\rm T4}{$\begin{array}{l} {\vrule height12pt depth8pt width0pt} xzy^\rcirclearrowleft{-}xyz^\rcirclearrowleft{+}\frac13xz^2{+}\frac13z^2x
\\{-}\frac23zxz+\frac13y^2z{+}\frac13zy^2{-}\frac23yzy\\ {+}\frac13z^2z{-}\frac{1}{3}zyz{+}\frac{a}{27}z^3 {\vrule height12pt depth7pt width0pt}\end{array}$}{$\begin{array}l {\vrule height12pt depth8pt width0pt} yz-zy-\frac13zz;\\ xz-zx-\frac13zy-\frac19zz;\\ xy-yx-\frac13yy+\frac13zx+\frac29zy+\frac{1-a}{27}zz {\vrule height12pt depth7pt width0pt} \end{array}$}{\rm none}{\rm trivial}{$(1-t)^{-3}$}{\rm Y/Y/Y}
\taBB{\rm T5}{$\begin{array}{l} {\vrule height11pt depth0pt width0pt} zyx+byxz+b^2xzy\\-bzxy-yzx-b^2xyz\\+(ab-1)zxz+azzx+ab^2xzz\end{array}$}{$\begin{array}{l} bxy{+}(1{-}ab)xz{-}yx{-}azx;\\ bxz-zx;\\
yz-zy-azz\end{array}$}{$b\neq 0$}{\rm trivial}{$(1-t)^{-3}$}{\rm Y/Y/Y}
\taBB{\rm T6}{$\begin{array}{l} yxz-xzy+zyx+yzx\\-xyz-zxy+(a-1)yzy\\+ayyz+azyy+zzz\end{array}$}{$\begin{array}{l} -xy+yx+ayy+zz;\\ xz+zx+(a{-}1)zy+ayz;\\
yz+zy\end{array}$}{\rm none}{\rm trivial}{$(1-t)^{-3}$}{\rm Y/Y/Y}
\taBB{\rm T7}{$\begin{array}{l} xzy^\rcirclearrowleft-xyz^\rcirclearrowleft-yzy\\+ayyz^\rcirclearrowleft+by^3+z^3\end{array}$}{\!\!$\begin{array}{l} -xy+yx+ayy+zz;\\ xz{+}byy{+}ayz{-}zx{+}(a{-}1)zy;\\ yz-zy\end{array}$}{\rm none}{$(a,b)\mapsto (a,-b)$}{$(1-t)^{-3}$}{\rm Y/Y/Y}
\taBB{\rm T8}{$xzy^\rcirclearrowleft-xyz^\rcirclearrowleft-yzy+yzz^\rcirclearrowleft+ay^3$}{$\begin{array}{l} -xy+yx+yz+zy;\\ xz+ayy-zx-zy+zz;\\ yz-zy\end{array}$}{$a\neq 0$}{\rm trivial}{$(1-t)^{-3}$}{\rm Y/Y/Y}
\taBB{\rm T9}{$\begin{array}{l} {\vrule height11pt depth0pt width0pt}a^2xyz+yzx+azxy\\-a^2xzy-zyx-ayxz\\+a^2xzz+zyz+azxz\end{array}$}{$\begin{array}{l} axy-yx+2zx;\\ axz-zx;\\ yz-zy+zz\end{array}$}{$a\neq 0$}{\rm trivial}{$(1-t)^{-3}$}{\rm Y/Y/Y}
\taBB{\rm T10}{$\begin{array}{l} xyz-yzx+zxy\\-yxz+xzy-zyx+yyz\\-yzy+zyy+zzz\end{array}$}{$\begin{array}{l} xy-yx+yy+zz;\\ xz+zx+2zy;\\ yz+zy\end{array}$}{\rm none}{\rm trivial}{$(1-t)^{-3}$}{\rm Y/Y/Y}
\taBB{\rm T11}{$\begin{array}{l} xxz+axzx+a^2 zxx\\ +yyz-ayzy+a^2zyy\end{array}$}{$\begin{array}{l} xz+azx;\\ yz-azy;\\ xx+yy\end{array}$}{$a\neq 0$}{$a\mapsto -a$}{$(1-t)^{-3}$}{\rm Y/Y/Y}
\taBB{\rm T12}{$\begin{array}{l} zzy+izyz-yzz+yyx\\-yxy+xyy+x^3\end{array}$}{$\begin{array}{l} xx+yy;\\ xy-yx+zz;\\ zy+iyz\end{array}$}{\rm none}{\rm trivial}{$(1-t)^{-3}$}{\rm Y/N/Y}
\taBB{\rm T13}{$\begin{array}{l} zzy-izyz-yzz+yyx\\-yxy+xyy+x^3\end{array}$}{$\begin{array}{l} xx+yy;\\ xy-yx+zz;\\ zy-iyz\end{array}$}{\rm none}{\rm trivial}{$(1-t)^{-3}$}{\rm Y/N/Y}
\taBB{\rm T14}{$\begin{array}{l} xyx+yxy+zyx+yzy+zyz\\+\theta xzy+\theta zxz+\theta^2xzx+\theta^2yxz \end{array}$}{$\begin{array}{l} {\vrule height11pt depth0pt width0pt}yx+\theta zy+\theta^2zx;\\ xy+zy+\theta^2xz;\\  yx+yz+\theta xz\end{array}$}{\rm none}{\rm trivial}{$(1-t)^{-3}$}{\rm Y/N/Y}
\taBB{\rm T15}{$\begin{array}{l} xyx+yxy+zyx+yzy+zyz\\+\theta^2 xzy+\theta^2 zxz+\theta xzx+\theta yxz \end{array}$}{$\begin{array}{l} {\vrule height11pt depth0pt width0pt}yx+\theta^2 zy+\theta zx;\\ xy+zy+\theta xz;\\  yx+yz+\theta^2 xz\end{array}$}{\rm none}{\rm trivial}{$(1-t)^{-3}$}{\rm Y/N/Y}
\taBB{\rm T16}{${y^2z}^\rcirclearrowleft+z^3+x^2z-xzx+zx^2$}{$\begin{array}{l} xx+yy+zz;\\  xz-zx;\\ yz+zy\end{array}$}{\rm none}{\rm trivial}{$(1-t)^{-3}$}{\rm Y/Y/Y}
\taBB{\rm T17}{${xy^2}^\rcirclearrowleft+y^3+xz^2-zxz+z^2x$}{$\begin{array}{l} xz-zx;\\ xy+yx+yy;\\  yy+zz\end{array}$}{\rm none}{\rm trivial}{$(1-t)^{-3}$}{\rm Y/Y/Y}
\taBB{\rm T18}{$y^3+{yz^2}^\rcirclearrowleft+az^3+x^2z-xzx+zx^2$}{$\begin{array}{l} xz-zx;\\ yz+zy+xx+azz;\\  yy+zz\end{array}$}{$a^2+4\neq 0$}{$a\mapsto -a$}{$(1-t)^{-3}$}{\rm Y/N/Y}
\taBB{\rm T19}{$x^2y+axyx+a^2yx^2+z^3$}{$\begin{array}l xx;\\ xy+ayx;\\ zz\end{array}$}{$\begin{array}l a\neq 0\\ a\neq 1\end{array}$}{\rm trivial}{$\frac{1+t}{1-2t}$}{\rm Y/Y/N}
\taBB{\rm T20}{$\begin{array}{l}xy^2+ayxy+a^2y^2x\\ +x^2z+a^2xzx+a^4zx^2\end{array}$}{$\begin{array}l xx;\\ xy+ayx;\\ xz+a^2zx+yy\end{array}$}{$\begin{array}l a\neq 0\\ a\neq 1\end{array}$}{\rm trivial}{$\frac{1+t}{1-2t}$}{\rm Y/Y/N}
\taBB{\rm T21}{$y^3+z^3+x^2z-xzx+zx^2$}{$\begin{array}l yy;\\ xx+zz;\\ xz-zx\end{array}$}{$\begin{array}l a\neq 0\\ a\neq 1\end{array}$}{\rm trivial}{$\frac{1+t}{1-2t}$}{\rm Y/Y/N}
\taBB{\rm T22}{$x^2y+axyx+a^2yx^2$}{$\begin{array}l xy+ayx;\\ xx\end{array}$}{$\begin{array}l a\neq 0\\ a\neq 1\end{array}$}{\rm trivial}{$\frac{1+t}{1-2t-t^2}$}{\rm Y/Y/N}
\taBB{\rm T23}{$x^2y-xyx+yx^2+y^3$}{$\begin{array}l xy-yx;\\ xx+yy\end{array}$}{\rm none}{\rm trivial}{$\frac{1+t}{1-2t-t^2}$}{\rm Y/Y/N}
\end{theorem}

\vfill\break

\begin{theorem}\label{main24-p} $A$ is a potential algebra on two generators given by a homogeneous degree $4$ potential if and only if $A$ is isomorphic $($as a graded algebra$)$ to an algebra from the following table. The algebras from different rows of the table are non-isomorphic. Algebras from {\rm (P19--P23)} are proper, algebras from {\rm (P24--P26)} are non-proper and non-degenerate, while algebras from {\rm (P27--P28)} are degenerate.
\bigskip
\hrule
\tabbb{\hhh}{\rm Potential $F$}{\rm Defining relations of $A_F$}{\rm Exceptions}{\rm Isomorphisms}{\rm Hilbert series}{\rm Exact}
\tabbb{\rm P19\hhh}{$x^4+a{x^2y^2}^\rcirclearrowleft+bxyxy^\rcirclearrowleft+y^4$}{$\begin{array}l x^3+axy^2+ay^2x+2byxy;\\ ax^2y+ayx^2+2bxyx+y^3\end{array}$}{$\begin{array}l 4(a+b)^2\neq1\\ (a,b)\neq(0,0)\\ (a,b)\neq\pm(1,1/2)\end{array}$}{$\begin{array}l (a,b)\mapsto (-a,-b)\\ (a,b)\mapsto \bigl(\frac{1-2b}{1+2a+2b},\frac{1-2a+2b}{2(1+2a+2b)}\bigr)\end{array}$}{$(1+t)^{-1}(1-t)^{-3}$}{\rm Y}
\tabbb{\rm P20\hhh}{${x^2y^2}^\rcirclearrowleft+\frac{a}{2}xyxy^\rcirclearrowleft$}{$\begin{array}l {\vrule height11pt depth0pt width0pt}xy^2+y^2x+ayxy;\\ x^2y+yx^2+axyx\end{array}$}
{\rm none}{\rm trivial}{$(1+t)^{-1}(1-t)^{-3}$}{\rm Y}
\tabbb{\rm P21\hhh}{$x^4+{x^2y^2}^\rcirclearrowleft+\frac{a}{2}xyxy^\rcirclearrowleft$}{$\begin{array}l {\vrule height11pt depth0pt width0pt}x^3+xy^2+y^2x+ayxy;\\ x^2y+yx^2+axyx\end{array}$}{\rm none}{\rm trivial}{$(1+t)^{-1}(1-t)^{-3}$}{\rm Y}
\tabbb{\rm P22\hhh}{$x^3y^\rcirclearrowleft+{x^2y^2}^\rcirclearrowleft-xyxy^\rcirclearrowleft$}{$\begin{array}l x^2y^\rcirclearrowleft+{xy^2}^\rcirclearrowleft-3yxy;\\ x^3+x^2y+yx^2-2xyx\end{array}$}{\rm none}{\rm trivial}{$(1+t)^{-1}(1-t)^{-3}$}{\rm Y}
\tabbb{\rm P23\hhh}{$x^4+\frac12xyxy$}{$\begin{array}l x^3+yxy;\\ xyx\end{array}$}{\rm none}{\rm trivial}{$\frac{(1+t^2)(1-t^5)}{(1-t-t^4-t^5)(1-t)}$}{\rm N}
\tabbb{\rm P24\hhh}{$x^4+y^4$}{$\begin{array}l {\vrule height11pt depth0pt width0pt}x^3;\\ y^3\end{array}$}{\rm none}{\rm trivial}{$\frac{1+t+t^2}{1-t-t^2}$}{\rm N}
\tabbb{\rm P25\hhh}{$x^3y^\rcirclearrowleft$}{$\begin{array}l {\vrule height11pt depth0pt width0pt}x^2y+yx^2+xyx;\\ x^3\end{array}$}{\rm none}{\rm trivial}{$\frac{1+t+t^2}{1-t-t^2}$}{\rm N}
\tabbb{\rm P26\hhh}{$xyxy^\rcirclearrowleft$}{$\begin{array}l yxy;\\ xyx\end{array}$}{\rm none}{\rm trivial}{$\frac{1+t+t^2}{1-t-t^2}$}{\rm N}
\tabbb{\rm P27\hhh}{$x^4$}{$x^3$}{\rm none}{\rm trivial}{$\frac{1+t+t^3}{1-t-t^2-t^3}$}{\rm N}
\tabbb{\rm P28\hhh}{$0$}{\rm none}{\rm none}{\rm trivial}{$(1-2t)^{-1}$}{\rm N}
\end{theorem}

\vskip.5cm

\begin{theorem}\label{main24-t} $A$ is a non-potential twisted potential algebra on two generators given by a homogeneous degree $4$ potential if and only if  $A$ is isomorphic $($as a graded algebra$)$ to an algebra from the following table. Distinct algebras anywhere in the table are non-isomorphic. Algebras from {\rm (T24--T33)} are proper, while the algebras in {\rm (T34)} are non-proper and non-degenerate.
\bigskip
\hrule
\tabbbb{\hhh}{\rm Twisted potential $F$}{\rm Defining relations of $A_F$}{\rm Exceptions}{\rm Hilbert series}{\rm Exact}
\tabbbb{\rm T24\hhh}{$x^2y^2+a^2y^2x^2+axy^2x+ayx^2y+bxyxy+abyxyx$}{$\begin{array}l {\vrule height11pt depth0pt width0pt}a^2yx^2+ax^2y+ab xyx;\\ xy^2+ay^2x+byxy\end{array}$}{$\begin{array}{l}a\neq 0\\ a\neq 1\end{array}$}{$(1+t)^{-1}(1-t)^{-3}$}{\rm Y}
\tabbbb{\rm T25\hhh}{$\begin{array}l {\vrule height11pt depth0pt width0pt}x^2y^2+y^2x^2-xy^2x-yx^2y+(a-1)x^2yx\\ +(1-a)xyx^2+ayx^3-ax^3y+\frac{a}2x^4\end{array}$}{$\begin{array}l {\vrule height11pt depth0pt width0pt}xy^2{-}y^2x{+}(a{-}1)xyx{+}(1{-}a)yx^2{-}ax^2y{+}\frac{a}2x^3;\\ yx^2-x^2y+ax^3\end{array}$}{\rm none}{$(1+t)^{-1}(1-t)^{-3}$}{\rm Y}
\tabbbb{\rm T26\hhh}{$\begin{array}{l}{x^2y^2}^\rcirclearrowleft-xyxy^\rcirclearrowleft+ayx^3\\ +ax^3y+(a-1)xyx^2+(a+1)x^2yx\end{array}$}{$\begin{array}l xy^2{+}y^2x{-}2yxy{+}ax^2y{+}(a{-}1)yx^2{+}(a{+}1)xyx;\\ ax^3+x^2y+yx^2-2xyx\end{array}$}{\rm none}{$(1+t)^{-1}(1-t)^{-3}$}{\rm Y}
\tabbbb{\rm T27\hhh}{${x^2y^2}^\rcirclearrowleft-xyxy^\rcirclearrowleft-xyx^2+x^2yx+ax^4$}{$\begin{array}l {\vrule height11pt depth0pt width0pt}x^2y+yx^2-2xyx;\\ xy^2+y^2x-2yxy-yx^2+xyx+ax^3\end{array}$}{\rm none}{$(1+t)^{-1}(1-t)^{-3}$}{\rm Y}
\tabbbb{\rm T28\hhh}{$x^2y^2+a^2y^2x^2+axy^2x-ayx^2y$}{$\begin{array}l {\vrule height11pt depth0pt width0pt}a^2yx^2-ax^2y;\\ xy^2+ay^2x\end{array}$}{\rm $a\neq 0$}{$(1+t)^{-1}(1-t)^{-3}$}{\rm Y}
\tabbbb{\rm T29\hhh}{$x^3y+yx^3+\theta xyx^2+\theta^2x^2yx+y^4$}{$\begin{array}l {\vrule height11pt depth0pt width0pt}x^2y+\theta yx^2+\theta^2xyx;\\ x^3+y^3\end{array}$}{\rm none}{$(1+t)^{-1}(1-t)^{-3}$}{\rm Y}
\tabbbb{\rm T30\hhh}{$x^3y+yx^3+\theta^2xyx^2+\theta x^2yx+y^4$}{$\begin{array}l {\vrule height11pt depth0pt width0pt}x^2y+\theta^2yx^2+\theta xyx;\\ x^3+y^3\end{array}$}{\rm none}{$(1+t)^{-1}(1-t)^{-3}$}{\rm Y}
\tabbbb{\rm T31\hhh}{$x^4-iyx^3-y^2x^2+iy^3x+y^4+xy^3+x^2y^2+x^3y$}{$\begin{array}l {\vrule height11pt depth0pt width0pt}x^3+x^2y+xy^2+y^3;\\ -ix^3-yx^2+iy^2x+y^3\end{array}$}{\rm none}{$(1+t)^{-1}(1-t)^{-3}$}{\rm Y}
\tabbbb{\rm T32\hhh}{$x^4+iyx^3-y^2x^2-iy^3x+y^4+xy^3+x^2y^2+x^3y$}{$\begin{array}l {\vrule height11pt depth0pt width0pt}x^3+x^2y+xy^2+y^3;\\ ix^3-yx^2-iy^2x+y^3\end{array}$}{\rm none}{$(1+t)^{-1}(1-t)^{-3}$}{\rm Y}
\tabbbb{\rm T33\hhh}{$\begin{array}{l}{\vrule height11pt depth0pt width0pt}x^2y^2-yx^2y+y^2x^2-xy^2x\\ +y^3x-xy^3+yxy^2-y^2xy\end{array}$}{$\begin{array}l {\vrule height11pt depth0pt width0pt}-x^2y+yx^2+y^2x+xy^2-yxy;\\ xy^2-y^2x-y^3\end{array}$}{\rm none}{$(1+t)^{-1}(1-t)^{-3}$}{\rm Y}
\tabbbb{\rm T34\hhh}{$x^3y+ax^2yx+a^2xyx^2+a^3yx^3$}{$\begin{array}l {\vrule height11pt depth0pt width0pt}x^2y+axyx+a^2yx^2;\\ x^3\end{array}$}{$\begin{array}{l}a\neq 0\\ a\neq 1\end{array}$}{$\!\!(1{+}t{+}t^2)(1{-}t{-}t^2)^{-1}$}{\rm N}
\end{theorem}

\vfill\break

\begin{remark}\label{CY0} Recall \cite{xx} that a $\K$-algebra $A$ is called $n$-Calabi--Yau if $A$ admits a projective $A$-bimodule resolution $0\to P_0\to{\dots}\to P_n\to A\to 0$ such that the dual sequence $0\to {\rm Hom}(P_n,A)\to {\dots}\to {\rm Hom}(P_0,A)\to 0$ is quasi-isomorphic to $0\to P_0\to{\dots}\to P_n\to 0$ (this effect is known as Poincar\'e duality). Now algebras from (P1--P8) and (P19--P27) are $3$-Calabi--Yau with the required resolution provided by tensoring the complex (\ref{comq}) by $A$ (over $\K$) on the left and interpreting the result as a bimodule complex. Actually, this captures (up to an isomorphism) all $3$-Calabi--Yau algebras which are also potential with the potential from ${\cal P}_{3,3}$ or ${\cal P}_{2,4}$. This augments the coarse description of Bockland \cite{yy} of graded $3$-Calabi--Yau algebras. What Bockland provides is a description of directed graphs and degrees such that there exists a homogeneous potential $F$ of given degree with the quotient of the graph path algebra by the relations $\delta_{x_j}F$ ($F$ is assumed to be written in terms of generators of the path algebra) being $3$-Calabi--Yau and proves that (in the category of degree graded algebras) every $3$-Calabi--Yau algebra emerges this way. On the other hand, we take two specific situations: $3$-petal rose and degree $3$ and $2$ petal rose and degree $4$ and describe the corresponding $3$-Calabi--Yau algebras themselves up to an isomorphism.
\end{remark}

\begin{remark}\label{wem} Potential algebras find applications in complex geometry as well. Namely, for an algebraic quasi-projective complex $3$-fold $X$ and a birational flop contraction $f:X\to Y$ contracting a single rational curve $C\subset X$ to a point $p$, Donovan and Wemyss \cite{WEM} associated an invariant, they named a contraction algebra and denoted $A_{\rm con}$. It turns out that $A_{\rm con}$ is finite dimensional and is either the quotient of $\C[x]$ by the ideal generated by $x^n$ for some $n\in\N$ or is a potential algebra on $2$ generators given by a potential $F$ satisfying $F_0=F_1=F_2=0$. This draws attention to finite dimensional potential algebras. Furthermore, Toda \cite{toda} demonstrated that the dimensions of $A_{\rm con}$ and its abelianization $A_{\rm con}^{\rm ab}$ are $\sum\limits_{j=1}^l j^2n_j$ and $n_1$ respectively, where the natural numbers $n_1,\dots,n_l$ are the so-called Gopakumar--Vafa invariants. Admittedly, it is not known whether every finite dimensional potential algebra $A_F$ with $F\in\K^{\rm cyc}\langle x,y\rangle$ satisfying $F_0=F_1=F_2=0$ features as a contraction algebra. As suggested by Wemyss, the first natural step to figuring this out is to determine whether for $F\in\K^{\rm cyc}\langle x,y\rangle$ with $F_0=F_1=F_2=0$, the number $\dim A_F-\dim A_F^{\rm ab}$ has the form $\sum\limits_{j=2}^l j^2n_j$ with $l,n_j\in\N$ (which it must if $A_F$ is a contraction algebra). The complete list of positive integers which fail to have this form is $1$, $2$, $3$, $5$, $6$, $7$, $9$, $10$, $11$, $14$, $15$, $18$, $19$, $23$ and $27$.
\end{remark}

Smoktunowitz and the first named author \cite{AN} proved that for every $F\in\K^{\rm cyc}\langle x,y\rangle$ with\break $F_0=F_1=F_2=F_3=F_4=0$, $A_F$ is infinite dimensional. This results prompted them to raise the following question.

\begin{que1} Does there exist $F\in\K^{\rm cyc}\langle x,y\rangle$ with $F_0=F_1=F_2=F_3=0$ such that $A_F$ is finite dimensional?
\end{que1}

We answer this question negatively.

\begin{theorem}\label{main-growth} Let $n,k\in\N$ be such that $n\geq 2$, $k\geq 3$ and $(n,k)\neq(2,3)$ and let $F\in\K^{\rm cyc}\langle x_1,\dots,x_n\rangle$ be such that $F_0={\dots}=F_{k-1}=0$. Then $A_F$ is infinite dimensional. Furthermore, $A_F$ has at least cubic growth if $(n,k)=(2,4)$ or $(n,k)=(3,3)$ with cubic growth being possible in both cases and $A_F$ has exponential growth otherwise.
\end{theorem}

Throughout the paper we perform linear substitutions. When describing a substitution, we keep the same letters for both old and new variables. We introduce a substitution by showing by which linear combination of (new) variables must the (old) variables be replaced. For example, if we write $x\to x+y+z$, $y\to z-y$ and $z\to 7z$, this means that all occurrences of $x$ (in the relations, potential etc.) are replaced by $x+y+z$, all occurrences of $y$ are replaced by $z-y$, while $z$ is swapped for $7z$. A {\bf scaling} is a linear substitution with a diagonal matrix. That is it swaps each variable with it own scalar multiple. For example, the substitution $x\to 2x$, $y\to -3y$ and $z\to iz$ is a scaling.

Section~2 is devoted to recalling relevant general information as well as to proving few auxiliary results of general nature. In Section~3 we prove a number of general results on potential and twisted potential algebras and provide examples. In particular, we prove Theorem~\ref{main-growth} in Section~3. In Sections~4--7 we prove Theorems~\ref{main24-p}, \ref{main33-p}, \ref{main24-t} and~\ref{main33-t} respectively. Section~8 is devoted to finite dimensional potential algebras. We make extra comments and discuss some open questions in the final Section~9.

\section{General background}

We shall always use the following partial order on power series with real coefficients. Namely, we write $\sum a_nt^n\geq \sum b_nt^n$ if $a_n\geq b_n$ for all $n\in\Z_+$.
If $V$ is an $n$-dimensional vector space over $\K$ and $R$ is a subspace of the $n^2$-dimensional space $V^2=V \otimes V$, then the quotient of the tensor algebra $T(V)$ by the ideal $I$ generated by $R$ is called a {\it quadratic algebra} and denoted $A(V,R)$. A quadratic algebra $A=A(V,R)$ is a {\it $PBW$-algebra} if there are bases $x_1,\dots,x_n$ and $g_1,\dots,g_m$ in $V$ and $R$ respectively such that with respect to some compatible with multiplication well-ordering on the monomials in $x_1,\dots,x_n$, $g_1,\dots,g_m$ is a Gr\"obner basis of the ideal $I$ of relations of $A$.

If we pick a basis $x_1,\dots,x_n$ in $V$, we get a bilinear form $b$ on the free algebra $\K\langle x_1,\dots,x_n\rangle$ (naturally identified with the tensor algebra $T(V)$) defined by $b(u,v)=\delta_{u,v}$ for every monomials $u$ and $v$ in the variables $x_1,\dots,x_n$. The algebra $A^!=A(V,R^\perp)$, where $R^\perp=\{u\in V^2:b(r,u)=0\ \text{for each}\ r\in R\}$, is known as the {\it dual algebra} of $A$. The algebra $A$ is called {\it Koszul} if $\K$ as a graded right $A$-module has a free resolution $\dots\to M_m\to\dots\to M_1\to A\to\K\to 0$, where the second last arrow is the augmentation map and the matrices of the maps $M_m\to M_{m-1}$ with respect to some free bases consist of homogeneous elements of degree $1$. Recall that there is a specific complex of free right $A$-modules, called the Koszul complex, whose exactness is equivalent to the Koszulity of $A$:
\begin{equation}\label{koco1}
\cdots\mathop{\longrightarrow}^{d_{k+1}} (A^!_k)^*\otimes A\mathop{\longrightarrow}^{d_k} (A^!_{k-1})^*\otimes A
 \mathop{\longrightarrow}^{d_{k-1}}\cdots \mathop{\longrightarrow}^{d_1} (A^!_{0})^*\otimes A=A\longrightarrow \K\to 0,
\end{equation}
where the tensor products are over $\K$, the second last arrow is the augmentation map and $d_k$ are given by $d_k(\phi \otimes u)=\sum\limits_{j=1}^n \phi_j\otimes x_ju$, where $\phi_j\in (A^!_{k-1})^*$, $\phi_j(v)=\phi(x_jv)$. Although $A^!$ and the Koszul complex seem to depend on the choice of a basis in $V$, it is not really the case up to the natural equivalence \cite{popo}. Recall that
$$
\begin{array}{l}
\text{every PBW-algebra is Koszul;}\\
\text{$A$ is Koszul $\iff$ $A^!$ is Koszul};\\
\text{if $A$ is Koszul, then $H_A(-t)H_{A^!}(t)=1$}.
\end{array}
$$

Note that if $F\in{\cal P}^*_{n,3}$, the corresponding twisted potential algebra $A_F$ is quadratic. One can easily verify that the complex (\ref{comq}) is always a subcomplex of the Koszul complex for $A_F$. Furthermore, the two complexes coincide precisely when $A_F$ is a proper twisted potential algebra. Thus we have the following curious fact:
\begin{equation}\label{koex}
\begin{array}{c}
\text{if $F\in{\cal P}^*_{n,3}$ is proper, then}\\ \text{$A_F$ is Koszul $\iff$ $A_F$ is exact.}\end{array}
\end{equation}

\subsection{Minimal series and maximal ranks}

Recall that a {\it finitely presented algebra} is an associative algebra $A$ given by  generators $x_1,\dots,x_n$ and relations $r_1,\dots,r_m\in \K\langle x_1,\dots,x_n\rangle$. That is, $A=\K\langle x_1,\dots,x_n\rangle/I$, where $I$, known as the ideal of relations for $A$, is the ideal in $\K\langle x_1,\dots,x_n\rangle$ generated by $r_1,\dots,r_m$. The {\it Poincar\'e series} of $A$ is $P^*_A(t)=\sum\limits_{n=0}^\infty a_kt^k$ with $a_k$ being the dimension of the subspace of $A$ spanned by non-commutative polynomials in $x_j$ of degree up to $k$. This series encodes the growth of $A$. $A$ is said to have {\it polynomial growth of degree} $m$ if $m=\lim\limits_{k\to \infty}\frac{\ln a_k}{\ln k}<\infty$ and $A$ is said to have {\it exponential growth} if $a_k\geq c^k$ for all $k$ for some $c>1$. However this is not the right series for our purposes. We shall define a different version of Poincar\'e series.

For each $k\in\Z_+$, let $J^{(k)}$ be the ideal in  $\K\langle x_1,\dots,x_n\rangle$ generated by all monomials of degree $k+1$. Clearly, $A^{(k)}=\K\langle x_1,\dots,x_n\rangle/(I+J^{(k)})$ is a finite dimensional algebra. Let
$$
d_k=d_k(A)=\dim A^{(k)}\ \ \text{for $k\in\Z_+$ with $d_{-1}=0$.}
$$
Obviously, $(d_k)$ is an increasing sequence of non-negative integers. We call $\smash{P_A=\sum\limits_{j=0}^\infty d_jt^j}$, the {\it P-series} of $A$. First, it is easy to see that if $r_j$ are homogeneous, the Hilbert series of $A$ is $H_A=\sum\limits_{j=0}^\infty (d_j-d_{j-1})t^j$. Note that $P_A^*\geq P_A$. As a result, $A$ is infinite dimensional if $(d_k)$ is unbounded and $A$ has exponential growth if the sequence $(d_k)$ grows exponentially. The reason for our choice is that unlike the classical Poincar\'e series $P_A^*$, the series $P_A$ enjoys certain stability under deformations. Note that the same series, although defined in a different manner, was introduced by Zelmanov \cite{Zelm}.

First, we describe what we mean by a variety of finitely presented algebras over a ground field $\K$. Let $m,n,d\in\N$ and  $P_d=\K[t_1,\dots,t_d]$. For $r\in P_d\langle x_1,\dots,x_n\rangle$ and $a\in\K^d$, $r(a)\in \K\langle x_1,\dots,x_n\rangle$ is obtained from $r$ by specifying $t_j=a_j$ for $1\leq j\leq d$.
\begin{equation}\label{VNG}
\begin{array}{l}
\text{Let $r_j\in P_d\langle x_1,\dots,x_n\rangle$ for $1\leq j\leq m$ and for $a\in\K^d$, let $A^a$ be the $\K$-algebra presented by}\\
\text{generators $x_1,\dots,x_n$ and  relations $r_1(a),\dots,r_m(a)$. We call $\{A^a\}_{a\in\K^d}$ a {\it variety of algebras}.}\\
\text{If, additionally, each $r_j$ is homogeneous, $\{A^a\}_{a\in\K^d}$ is called a {\it variety of graded algebras}}
\end{array}
\end{equation}

If $W=\{A^a\}_{a\in\K^d}$ is a variety of algebras, we denote
$$
\smash{d_k(W)=\min\{d_k(A):A\in W\}\ \ \text{for}\ \ k\in\Z_+\quad \text{and}\quad P_W=\sum_{k=0}^\infty d_k(W)t^k.}\hhhh
$$
If $W=\{A^a\}_{a\in\K^d}$ is a variety of graded algebras, we denote
$$
\smash{h_k(W)=\min\{\dim A_k:A\in W\}\ \ \text{for}\ \ k\in\Z_+\quad \text{and}\quad H_W=\sum_{k=0}^\infty h_k(W)t^k}.\hhhh
$$
If $\K$ is uncountable, we say that a property $P$ of points in $\K^d$ holds for {\it generic} $a\in\K^d$ if the set of points violating $P$ is contained in the union of countably many affine algebraic varieties in $\K^d$ (each different from the entire $\K^d$).

\begin{lemma}\label{miser1} Assume that $W=\{A^a\}_{a\in\K^d}$ is a variety of finitely presented algebras $($as in {\rm (\ref{VNG}))}. Then for every $k\in\Z_+$, the set $U_k=\{a\in\K^d:d_k(A^a)=d_k(W)\}$ is Zarisski open in $\K^d$. In particular, if $\K$ is uncountable, then $P_{A^a}=P_W$ for generic $a\in\K^d$.
\end{lemma}

\begin{proof} Let $a_0\in U_k$ and $\Phi_k$ be the linear span of monomials of degree $\leq k$ in $\K\langle x_1,\dots,x_n\rangle$. Then the dimension of  $\Phi_k\cap(I^{a_0}+J^{(k)})$ is $N=\dim\Phi_k-d_k(W)$. Thus we can pick a basis $g_1,\dots,g_N$ in $\Phi_k\cap(I^{a_0}+J^{(k)})$. Pick monomials $m_1,\dots,m_N$ of degree $\leq k$ such that the matrix of $m_j$-coefficeints  of $g_s$ is non-degenerate. Then modulo $J^{(k)}$ each $g_s$ can be written as $\sum_t u_tr_{j_t}(a_0)v_t$ for some $u_t,v_t\in \K\langle x_1,\dots,x_n\rangle$. For $a\in \K^d$ consider the matrix $M(a)$ of  $m_j$-coefficeints  of $\sum_t u_tr_{j_t}(a)v_t$. Then $M(a_0)=M$ and the entries of $M(a)$ depend on $a$ polynomially. Since $M$ is non-degenerate, the set $U$ of $a$ for which $M(a)$ is non-degenerate is Zarisski open in $\K^d$ and contains $a_0$. By definition of $M(a)$, $\dim(\Phi_k\cap(I^{a}+J^{(k)}))\geq N$ for $a\in U$, which in turn means that $d_k(A^a)\leq d_k(W)$. By minimality of $d_k(W)$, we have $d_k(A^a)=d_k(W)$ for $a\in U$ and therefore $U\subseteq U_k$. Since $a_0\in U$ and $a_0$ was an arbitrary element of $U_k$ to begin with, $U_k$ is Zarisski open.
\end{proof}

The following statement, apart from being well-known, follows easily (and can be proven in exactly the same way) from Lemma~\ref{miser1}. It was probably Ufnarovskii, who first made this observation \cite{Uf}.

\begin{lemma}\label{miser} Assume that $W=\{A^a\}_{a\in\K^d}$ is a variety of graded algebras $($as in {\rm (\ref{VNG}))}. Then for every $k\in\Z_+$, the non-empty set  $U_k=\{a\in\K^d:\dim A^a_k=h_k(W)\}$ is Zarisski open in $\K^d$. As a consequence, $H_{A^a}=H_W$ for generic $a\in \K^d$ provided $\K$ is uncountable.
\end{lemma}

\begin{lemma} \label{maxra}
Assume that $W=\{A^a\}_{a\in\K^d}$ is a variety of graded algebras $($as in {\rm (\ref{VNG}))}. Let also $\Lambda$ be a $p\times q$ matrix, whose entries are degree $r$ $($the degree is with respect to $x_j)$ homogeneous elements of $P_d\langle x_1,\dots,x_n\rangle$, where $P_d=\K[t_1,\dots,t_d]$. For every fixed $a\in \K^d$, we can interpret $\Lambda$ as a map from $(A^a)^q$ to $(A^a)^p$ $($treated as free right $A^a$-modules$)$ acting by multiplication of the matrix $\Lambda$ by a column vector from $(A^a)^q$.

Fix a non-negative integer $k$ and let $U$ be a non-empty Zarisski open subset of $\K^d$ such that $\dim A^a_k$ and
$\dim A^a_{k+r}$ do not depend on $a$ provided $a\in U$. For $a\in\K^d$ let $\rho(k,a)$ be the rank of $\Lambda$ as a
linear map from $(A_k^a)^q$ to $(A^a_{k+r})^p$ and $\rho_{\max}(k)=\max\{\rho(k,a):a\in U\}$.
Then the set $W_k=\{a\in U:\rho(k,a)=\rho_{\max}(k)\}$ is Zarisski open in $\K^d$.
\end{lemma}

\begin{proof} Let $a\in W_k$. Then $\rho(k,a)=g$, where $g=\rho_{\max}(k)$.
Pick linear bases of monomials $e_1,\dots,e_u$ and $f_1,\dots,f_v$ is $A^a_k$ and $A^a_{k+r}$ respectively.
Obviously, the same sets of monomials serve as linear bases for $A^s_k$ and $A^s_{k+d}$ respectively
for $s$ from a Zarisski open set $V\subseteq U$. Then $\Lambda$ as a linear map from
$(A_k^s)^q$ to $(A^s_{k+r})^p$ for $s\in V$ has an $u^q\times v^p$ matrix $M_s$ with respect to the said
bases. The entries of this matrix depend on the parameters polynomially. Since the rank of this matrix
for $s=a$ equals $g$, there is a square $g\times g$ submatrix whose determinant is non-zero when
$s=a$. The same determinant is non-zero for a Zarisski open subset of $V$. Thus for $s$ from the
last set the rank of $M_s$ is at least $g$. By maximality of $g$, the said rank equals $g$.
Thus $a$ is contained in a Zarisski open set, for all $s$ from which $\rho(k,s)=g$. That is,
$W_k$ is Zarisski open.
\end{proof}

\begin{lemma}\label{complex} Let $W=\{A^a\}_{a\in\K^d}$ be a variety of graded algebras $($as in {\rm (\ref{VNG}))}. Assume that $\K$ is uncountable and that we have a complex
\begin{equation}\label{coco}
0\to (A^a)^{k_1}\mathop{\longrightarrow}^{d_{1}}(A^a)^{k_2}\mathop{\longrightarrow}^{d_{2}}{\dots}\mathop{\longrightarrow}^{d_{{m-1}}}(A^a)^{k_m}\to\K\to0,
\end{equation}
of right $A^a$-modules, where the second last arrow vanishes on all homogeneous elements of degree $\geq1$ and the maps $d_j$ are given by matrices $\Lambda_j$ satisfying conditions of Lemma~$\ref{maxra}$ composed of homogeneous elements of degree $r_j$. Assume also that $U$ is a non-empty Zarisski open subset of $\K^d$ such that $\dim A^a_j$ does not depend on $a\in U$ for $0\leq j\leq r=r_1+{\dots}+r_{m-1}$. Then the following dichotomy holds$:$ either $(\ref{coco})$ is non-exact for all $a\in U$ or $(\ref{coco})$ is exact for generic $a\in\K^d$.
\end{lemma}

\begin{proof} Assume that there is $a_0\in U$ for which $(\ref{coco})$ is exact. The proof will be complete if we show that $(\ref{coco})$ is exact for generic $a\in\K^d$. By Lemma~\ref{miser}, $V=\{a\in\K^d:\dim A^a_j=h_j(W)\ \ \text{for}\ \ 0\leq j\leq r\}$ is non-empty and Zarisski open in $\K^d$. Obviously, $U\subseteq V$. Denote $B=A^{a_0}$. First, we shall verify that $H_B=H_W$. Assume the contrary. Since $H_W$ is the minimal Hilbert series for the variety $W$, there is $s\in\Z_+$ such that $\dim B_j=h_j(W)$ for $j<s+r$ and $\dim B_{s+r}>h_{s+r}(W)$. Consider the graded 'slice'
$$
0\to (A_s^a)^k_1\mathop{\longrightarrow}^{d_{1}}(A_{s+r_1}^a)^{k_2}\mathop{\longrightarrow}^{d_{2}}{\dots}\mathop{\longrightarrow}^{d_{{m-1}}}(A_{s+r}^a)^k_m\to 0
$$
of (\ref{coco}). For $a=a_0$ this complex of finite dimensional vector spaces is exact. By Lemma~\ref{miser}, $H_{A^a}=H_W$ for generic $a$. Applying Lemma~\ref{maxra} to each arrow in the above display (working from left to right), we then see that for generic $a\in U$, the dimension of $d_{m-1}((A_{s+r-r_{m-1}}^a)^k_{m-1})$ equals
$$
k_{m-1}h_{s+r-r_{m-1}}(W)-k_{m-2}h_{s+r-r_{m-1}-r_{m-2}}(W)+{\dots}+(-1)^mk_1h_{s}(W).
$$
The same is true for $a=a_0$ because of the exactness, which also implies that the same expression equals $k_m\dim B_{s+r}$. Thus $\dim B_{s+r}\leq \dim A^{a}_{s+r}$ for generic $a$. By minimality of $H_W$, $\dim B_{s+r}=h_{s+r}(W)$, which is a contradiction. Thus $H_B=H_W$.

Now we restrict ourselves to $a$ for which $H_{A^a}=H_W$ (happens for generic $a$), which we now know includes $a=a_0$ for which (\ref{coco}) is exact. We repeat the argument with applying Lemma~\ref{maxra} to the arrows in degree 'slices' of (\ref{coco}): for generic $a$ satisfying $H_{A^a}=H_W$ (this makes the dimensions of the spaces independent on $a$) the rank of each arrow is maximal. Comparing the ranks with that for $a=a_0$ and using the exactness for $a=a_0$, we see that this simultaneous maximality of the ranks is actually equivalent to the exactness of (\ref{coco}). Thus $(\ref{coco})$ is exact for generic $a$.
\end{proof}

\subsection{The module of syzigies}

Assume $A$ is a finitely presented algebra given by  generators $x_1,\dots,x_n$ and relations $r_1,\dots,r_m\in \K\langle x_1,\dots,x_n\rangle$ (not necessarily homogeneous) and $I$ is the ideal of relations of $A$.

The module of (two-sided) syzigies $S(A)$ for $A$ ($1$-syzigies to be precise) is defined in the following way. First, we consider the free $\K\langle x_1,\dots,x_n\rangle$-bimodule $M$ with generators $\widehat{r}_1,\dots,\widehat{r}_m$ being just $m$ symbols. The map $\widehat{r}_j\mapsto r_j$ naturally extends to a bimodule morphism from $M$ to $\K\langle x_1,\dots,x_n\rangle$. The module $S(A)$ is by definition the kernel of this morphism. That is, it consists of sums $\sum u_j\widehat{r}_{s_j}v_j$ with $u_j,v_j\in \K\langle x_1,\dots,x_n\rangle$, which vanish when symbols $\widehat{r}_s$ are replaces by $r_s$. The members of $S(A)$ are called syzigies.

Some syzigies are always there. For example, for $1\leq j,k\leq m$ and $u\in  \K\langle x_1,\dots,x_n\rangle$, $\widehat{r}_jur_k-r_ju\widehat{r}_k$ is a syzigy. We call a syzigie of this form, a {\it trivial syzigy}. Note that $S(A)$ depends not just on the algebra $A$ but on the choice of a presentation of $A$, so $S(A)$ constitutes an abuse of notation. If we consider a set $M$ of monomials in $x_1,\dots,x_n$, none of which contains another as a submonomial, an {\it overlap} of monomials in $M$ is a monomial $m$, which starts with $m_1\in M$, ends with $m_2\in M$ and has degree strictly less than ${\rm deg}\,m_1m_2$. Naturally, the degree of $m$ is called the degree of the overlap.

\begin{remark}\label{syzz} Note that syzigies are implicitly computed, when a reduced Gr\"obner basis of the ideal of relations is constructed. Namely, each time an ambiguity (=an overlap of leading monomials of members of the basis constructed so far) is resolved (=produces no new element of the basis) the computation leading to the resolution can be written as a syzigy. However, there is more. If we collect all syzigies obtained while constructing a reduced Gr\"obner basis (the process could be infinite resulting in infinitely many syzigies) and throw in the trivial syzigies, we obtain a generating set for the module $S(A)$. It is nothing new: this statement can be viewed as just another way of stating the diamond lemma.
\end{remark}

\subsection{A remark on PBW algebras}

Note that if $A=A(V,R)$ is a quadratic algebra, $x_1,\dots,x_n$ is a fixed basis in $V$ and the monomials in $x_j$ are equipped with an order compatible with multiplication, then we can choose a basis $g_1,\dots,g_m$ in $R$ such that the leading monomials $\overline{g}_j$ of $g_j$ are pairwise distinct. We then call $S=\{\overline{g}_1,\dots,\overline{g}_m\}$ the {\it set of leading monomials of $R$.} Note that although there are multiple bases in $R$ with pairwise distinct leading monomials of the members, the set $S$ is uniquely determined by $R$ (provided $x_j$ and the order are fixed). The following result is an improved version of a lemma from \cite{SKL1}.

\begin{lemma}\label{ome0} Let $A=A(V,R)$ be a quadratic algebra. Then the following statements are equivalent$:$
\begin{itemize}\itemsep=-2pt
\item[\rm (\ref{ome0}.1)]$A$ is PBW, $\dim A_1=3$, $\dim A_2=6$ and $\dim A_3=10;$
\item[\rm (\ref{ome0}.2)]$A$ is PBW and $H_A=(1-t)^{-3};$
\item[\rm (\ref{ome0}.3)]$\dim A_3=10$ and there is a basis $x,y,z$ in $V$ and a well-ordering on $x,y,z$ monomials compatible with multiplication, with respect to which the set of leading monomials of $R$ is $\{xy,xz,yz\}$.
\end{itemize}
\end{lemma}

\begin{proof} The implication (\ref{ome0}.2)$\Longrightarrow$(\ref{ome0}.1) is obvious. Next, assume that (\ref{ome0}.1) is satisfied. Then $\dim V=\dim R=3$ and $\dim A_3=10$. Let $a,b,c$ be a PBW-basis for $A$, while $f,g,h$ be corresponding PBW-generators. Since $f$, $g$ and $h$ form a Gr\"obner basis of the ideal of relations of $A$, it is easy to see that $\dim A_3$ is $9$ plus the number of overlaps of the leading monomials $\overline{f}$, $\overline{g}$ and $\overline{h}$ of $f$, $g$ and $h$. Since $\dim A_3=10$, the monomials $\overline{f}$, $\overline{g}$ and $\overline{h}$ must produce exactly one overlap. Now it is a straightforward routine check that if at least one of three degree 2 monomials in 3 variables is a square, these monomials overlap at least twice. The same happens, if the three monomials contain $uv$ and $vu$ for some distinct $u,v\in\{a,b,c\}$. Finally, the triples $(ab,bc,ca)$ and $(ba,cb,ac)$ produce 3 overlaps apiece. The only option left, is for $(\overline{f},\overline{g},\overline{h})$ to be $\{xy,xz,yz\}$, where $(x,y,z)$ is a permutation of $(a,b,c)$. This completes the proof of implication (\ref{ome0}.1)$\Longrightarrow$(\ref{ome0}.3).

Finally, assume that (\ref{ome0}.3) is satisfied. Then the leading monomials of defining relations have exactly one overlap. If this overlap produces a non-trivial degree $3$ element of the Gr\"obner basis of the ideal of relations of $A$, then $\dim A_3=9$, which contradicts the assumptions. Hence, the overlap resolves. That is, a linear basis in $R$ is actually a Gr\"obner basis of the ideal of relations of $A$. Then $A$ is PBW. Furthermore, the leading monomials of the defining relations are the same as for $\K[x,y,z]$ with respect to the left-to-right lexicographical ordering with $x>y>z$. Hence $A$ and $\K[x,y,z]$ have the same Hilbert series: $H_A=(1-t)^{-3}$. This completes the proof of implication (\ref{ome0}.3)$\Longrightarrow$(\ref{ome0}.2).
\end{proof}

\subsection{Some canonical forms}

The following lemma is a well-known fact. We provide a proof for the sake of completeness.

\begin{lemma}\label{1-dim} Let $\K$ be an arbitrary algebraically closed field $($characteristics $2$ and $3$ are allowed here$)$, $M$ be a $2$-dimensional vector space over $\K$ and $S$ be a $1$-dimensional subspace of $M^2=M\otimes M$. Then $S$ satisfies exactly one of the following conditions$:$
\begin{itemize}\itemsep=-3pt
\item[\rm (I1)]there is a basis $x,y$ in $M$ such that $S=\spann\{yy\};$
\item[\rm (I2)]there is a basis $x,y$ in $M$ such that $S=\spann\{yx\};$
\item[\rm (I3)]there is a basis $x,y$ in $M$ such that $S=\spann\{xy-\alpha yx\}$ with $\alpha\in\K^*;$
\item[\rm (I4)]there is a basis $x,y$ in $M$ such that $S=\spann\{xy-yx-yy\}$.
\end{itemize}
Furthermore, if $S=\spann\{xy-\alpha yx\}=\spann\{x'y'-\beta y'x'\}$ with $\alpha\beta\neq 0$ for two different bases $x,y$ and $x',y'$ in $M$, then either $\alpha=\beta$ or $\alpha\beta=1$.
\end{lemma}

\begin{proof} If $M$ is spanned by a rank one element, then $S=\spann\{uv\}$, where $u,v$ are non-zero elements of $M$ uniquely determined by $S$ up to non-zero scalar multiples. If $u$ and $v$ are linearly independent, we set $y=u$ and $x=v$ to see that (I2) is satisfied. If $u$ and $v$ are linearly dependent, we set $y=u$ and pick an arbitrary $x\in M$ such that $y$ and $x$ are linearly independent. In this case (I1) is satisfied. Obviously, (I1) and (I2) can not happen simultaneously. Since $S$ in (I3) and (I4) are spanned by rank 2 elements, neither of them can happen together with either (I1) or (I2).

Now let $u,v$ be an arbitrary basis in $M$ and $S$ be spanned by a rank $2$ element $f=auu+buv+cvu+dvv$ with $a,b,c,d\in\K$. A linear substitution $u\to u$, $v\to v+su$ with an appropriate $s\in\K$ turns $a$ into $0$ (one must use the fact that $f$ has rank $2$ and that $\K$ is algebraically closed: $s$ is a solution of a quadratic equation). Thus we can assume that $a=0$. Since $f$ has rank $2$, it follows that $bc\neq 0$. If $b+c\neq 0$, we set $x=u+\frac{dv}{b+c}$ and $y=bv$ to see that (I3) is satisfied with $\alpha=\frac{c}{b}\neq 1$. Note also that the only linear substitutions which send $xy-\alpha yx$  to $xy-\beta yx$ (up to a scalar multiple) with $\alpha\beta\in\K^*$, $\alpha\neq 1$ are scalings and scalings composed with swapping $x$ and $y$. In the first case $\alpha=\beta$. In the second case $\alpha\beta=1$. Finally, if $b+c=0$, then we have two options. If, additionally, $d=0$, $S$ is spanned by $xy-yx$ with $x=u$ and $y=v$, which falls into (I3) with $\alpha=1$.  Note that any linear substitution keeps the shape of $xy-yx$ up to a scalar multiple. If $d\neq 0$, we set $x=u$ and $y=\frac{dv}{b}$ to see that $S$ is spanned by $xy-yx-yy$ yielding (I4). The remarks on linear substitutions, we have thrown on the way complete the proof.
\end{proof}

One of the instruments we use is the following canonical form result on ternary cubics, which goes all the way back to Weierstrass. Note that if $\K$ is not algebraically closed or if the characteristic of $\K$ is $2$ or $3$, the result does not hold.

\begin{lemma}\label{co3-3} Every homogeneous degree $3$ polynomial $L\in \K[x,y,z]$ by means of a non-degenerate linear change of variables can be brought to one of the following forms$:$

\medskip
\noindent $\begin{array}{ll}\text{{\rm (Z1)}\ \ $L=L_{a,b}=a(x^3+y^3+z^3)+bxyz$ with $a,b\in\K;$\qquad}&\text{{\rm (Z5)}\ \ $L=xz^2+y^2z;$}\\
\text{{\rm (Z2)}\ \ $L=xyz+(y+z)^3;$}&\text{{\rm (Z6)}\ \ $L=y^3+z^3;$}\\
\text{{\rm (Z3)}\ \ $L=xyz+z^3;$}&\text{{\rm (Z7)}\ \ $L=yz^2;$}\\
\text{{\rm (Z4)}\ \ $L=xz^2+y^3;$}&\text{{\rm (Z8)}\ \ $L=z^3.$}
\end{array}$
\medskip

Furthermore, $L$ with different labels in the above list are non-equivalent $(=$can not be transformed into one another by a linear substitution$)$.
\end{lemma}

\begin{proof} It is a straightforward rephrasing of a well-known canonical form result for ternary cubics, see, for instance, \cite{cub1,cub2}. For the sake of completeness, we outline the idea of the proof. Consider the projective curve $C$ given by $L=0$. If $C$ is regular, $L$ can be transformed into $L_{a,b}$ with $27a^3+b^3\neq 0$ with the coefficients of the corresponding substitution written explicitly via coordinates of the inflection points of $C$. The rest is just going through various types of irregular $C$.
\end{proof}

This result is not entirely final. For instance, the $GL_3(\K)$-orbit of a generic $L\in \K[x,y,z]$ contains $L_{1,b}$ for more than one $b$ (12, actually). However, Lemma~\ref{co3-3} is sufficient for our purposes.

\begin{lemma}\label{co2-4} Let $G\in\K[x,y]$ be a homogeneous degree $4$ polynomial. Then by means of a linear substitution $(=$natural action of $GL_2(\K))$ $G$ can be turned into one of the following forms$:$

\medskip
\noindent $\begin{array}{ll}\text{{\rm (C1)}\ \ $G=0;$\ \ }&\text{{\rm (C4)}\ \ $G=x^2y^2;$}\\
\text{{\rm (C2)}\ \ $G=x^4;$}&\text{{\rm (C5)}\ \ $G=x^4+x^2y^2;$}\\
\text{{\rm (C3)}\ \ $G=x^3y;$\ \ }&\text{{\rm (C6)}\ \ $G=x^4+ax^2y^2+y^4$ with $a^2\neq 4.$}
\end{array}$
\medskip

Moreover, to which of the above six forms $G$ can be transformed is uniquely determined by $G$. As for the last option, the parameter $a$ is not determined by $G$ uniquely. However, the level of non-uniqueness is clear from the following fact.

For $a\in\K$, $a^2\neq 4$, the set of $S\in GL_2(\K)$, the substitution by which turns $x^4+ax^2y^2+y^4$ into $\lambda(x^4+bx^2y^2+y^4)$ for some $\lambda\in\K^*$ and $b\in\K$ with $b^2\neq 4$ does not depend on $a$, forms a subgroup $H$ of $GL_2(\K)$ and consists of non-zero scalar multiples of the matrices of the form
$$
\left(\begin{array}{cc}1&0\\ 0&p\end{array}\right),\ \left(\begin{array}{cc}0&1\\ p&0\end{array}\right),\
\left(\begin{array}{cc}1&p\\ -1&p\end{array}\right),\ \left(\begin{array}{cc}1&p\\ 1&-p\end{array}\right),\
\left(\begin{array}{cc}qp&1\\ p&q\end{array}\right),\ \ \text{where $p,q\in\K$, $p^4=1$ and $q^2=-1$.}
$$
\end{lemma}

\begin{proof} If $G=0$ to begin with, it stays this way after any linear sub. Thus we can assume that $G\neq 0$. Since $G$ is homogeneous of degree $4$ and $\K$ is algebraically closed, $G$ is the product of 4 non-zero homogeneous degree $1$ polynomials $G=u_1u_2u_3u_4$. Analyzing possible linear dependencies of $u_j$, we see that unless $u_j$ are pairwise linearly independent, a linear substitution turns $G$ into a unique form from (C2--C5). Indeed, all $u_j$ being proportional leads to (C2), three being proportional with one outside their one-dimensional  linear span gives (C3), two pairs of proportional $u_j$ generating distinct one-dimensional spaces corresponds to (C4), while only one pair of proportional $u_j$ leads to (C5).

This leaves the case of  $u_j$ being pairwise linearly independent. Note that for $a\in\K$ satisfying $a^2\neq 4$, this is the case with  $G=x^4+ax^2y^2+y^4$. Our next step is to see that an arbitrary $G$ with this property can be turned into a $G$ from (C6) by means of a linear substitution. We achieve this in three steps. First, making a substitution which turns $u_1$ into $x$ and $u_2$ into $y$, we make $G$ divisible by $xy$: $G=xy(px^2+qxy+ry^2)$ with $p,q,r\in\K$. Note that $pr\neq 0$ (otherwise there is a linear dependent pair of degree $1$ divisors of $G$). Thus (recall that $\K$ is algebraically independent) a scaling turns $G$ into $G=xy(x^2+qxy+y^2)$ with $q\in\K$. The substitution $x\to x+y$, $y\to x-y$ together with a scaling transforms $G$ into $x^4+ax^2y^2+y^4$. Finally, pairwise linear independence of degree $1$ factors translates into $a^2\neq 4$ and we are done with the first part of the lemma.

It is easy to see that $H$ consisting of non-zero scalar multiples of the matrices in the above display is a subgroup of $GL_2(\K)$ and that substitutions provided by matrices from $H$ preserve the class (C6) up to scalar multiples. Assume now that $a\in\K$, $a^2\neq 4$ and the linear substitution provided by
$$
S=\left(\begin{array}{cc}\alpha&\beta\\ p&q\end{array}\right)\in GL_2(\K)
$$
turns $x^4+ax^2y^2+y^4$ into $\lambda(x^4+bx^2y^2+y^4)$ for some $\lambda\in\K^*$ and $b\in\K$ with $b^2\neq 4$.
The proof will be complete if we show that $S\in H$. The condition that $x^4+ax^2y^2+y^4$ is mapped to $\lambda(x^4+bx^2y^2+y^4)$ yields the following system:
\begin{equation}\label{SSTT}
\begin{array}{r}
\alpha^4+p^4+a\alpha^2p^2=\beta^4+q^4+a\beta^2q^2\neq 0;
\\
2\alpha^3\beta+a\alpha^2pq+a\alpha\beta p^2+2p^3q=0;
\\
2\alpha\beta^3+a\alpha\beta q^2+a\beta^2pq+2pq^3=0.
\end{array}
\end{equation}
Indeed, the first equation ensures that after the sub the $x^4$ and $y^4$ coefficients of $G$ are equal and non-zero, while the remaining two equations are responsible for the absence of $x^3y$ and $xy^3$ in $G$.

If $q\alpha=0$, the above system immediately gives $\alpha=q=0$ and $\beta^4=p^4\neq 0$ and therefore $S\in H$. If $p\beta=0$, we similarly have $\beta=p=0$ and $\alpha^4=q^4\neq 0$ ensuring the membership of $S$ in $H$. Thus it remains to consider the case $pq\alpha\beta\neq 0$. Set $s=\alpha/\beta$ and $t=q/p$. The last two equations in the above display now read
$$
2s^3+as^2t+as+2t=0,\qquad 2s+ast^2+at+2t^3=0.
$$
Multiplying the first equation by $t$, the second by $s$ and subtracing (after an obvious rearrangement) yields $(s^2-t^2)(st-1)=0$.
$S$ being non-degenerate implies $st\neq 1$. Hence $s^2=t^2$. Thus $t=s$ or $t=-s$.

If $t=s$, we plug the definitions of $s$ and $t$ back into the first equation of (\ref{SSTT}). This gives $s^2=-1$. If $t=-s$, the same procedure yields $s^2=1$. Thus we have the following options for $(s,t)$: $(i,i)$, $(-i,-i)$, $(1,-1)$ and $(-1,1)$. The inclusion $S\in H$ becomes straightforward.
\end{proof}

\begin{remark} \label{groupH} For the group $H$ of Lemma~\ref{co2-4}, the group $H_0=H/\K^*I$ is finite. Moreover, it is isomorphic to $S_4$. Indeed, it is easy to check that $H_0$ has 24 elements and trivial centre. Since there is only one such group up to an isomorphism, namely $S_4$, $H_0\simeq S_4$.
\end{remark}

\begin{lemma}\label{S3} For $a,b\in\K$ satisfying $4(a+b)^2\neq1$, let
$$
F_{a,b}=x^4+a{x^2y^2}^\rcirclearrowleft+bxyxy^\rcirclearrowleft+y^4\in\K^{\rm cyc}\langle x,y\rangle.
$$
Then $F_{a,b}$ and $F_{a',b'}$ are equivalent $($=can be obtained from one another by a linear substitution$)$ if and only if $(a,b)$ and $(a',b')$ belong to the same orbit of the group action generated by two involutions $(a,b)\mapsto (-a,-b)$ and $(a,b)\mapsto\left(\frac{1-2b}{1+2a+2b},\frac{1-2a+2b}{2(1+2a+2b)}\right)$. This group has $6$ elements and is isomorphic to $S_3$.
\end{lemma}

\begin{proof} Note that the abelianization $F^{\rm ab}_{a,b}\in\K[x,y]$ of $F_{a,b}$ is given by $F^{\rm ab}_{a,b}=x^4+4(a+b)x^2y^2+y^4$. According to the assumption $4(a+b)^2\neq1$, each $F^{\rm ab}_{a,b}$ is of the form (C6) of Lemma~\ref{co2-4}. Since every linear substitution transforming $F_{a,b}$ into $F_{a',b'}$ must also transform $F^{\rm ab}_{a,b}$ into $F^{\rm ab}_{a',b'}$, the relevant substitutions can only be provided by matrices from the group $H$ of  Lemma~\ref{co2-4}. Factoring out the scalar multiples, we are left with the group $H/\K^*I$, which happens to be finite (of order 24) and whose elements are listed in  Lemma~\ref{co2-4}. Note also that the substitutions $x\to -x$, $y\to y$ and $x\to y$, $y\to x$ transform each $F_{a,b}$ to itself. After factoring these out from $H/\K^*I$, we are left with a group of order $6$, which is easily seen to be isomorphic to $S_3$ and to act essentially freely on $F_{a,b}$. Two involutions generating $S_3$ correspond to substitutions $x\to x$, $y\to iy$ and $x\to x+iy$, $y\to x-iy$, which act by $(a,b)\mapsto (-a,-b)$ and $(a,b)\mapsto\left(\frac{1-2b}{1+2a+2b},\frac{1-2a+2b}{2(1+2a+2b)}\right)$ on the parameters $(a,b)$.
\end{proof}

\section{General results on twisted potential algebras}

\begin{lemma}\label{gen3} Let $F\in \K\langle x_1,\dots,x_n\rangle$ be a twisted potential with $F_0=F_1=0$ and let $A=A_F$. Then the sequence $(\ref{comq})$ is a complex, which is exact at its three rightmost terms.
\end{lemma}

\begin{proof}Since no constants feature in the defining relations of $A$, the augmentation map $d_0:A\to \K$ is well-defined. By definition of $d_1$, $d_0\circ d_1=0$. Using the definition of $d_2$, we see that
$$
d_1\circ d_2(u_1,\dots,u_n)=\sum_{k,j=1}^n x_j(\delta_{x_j}\delta^R_{x_k} F)u_k=\sum_{k=1}^n (\delta^R_{x_k} F)u_k=0\ \ \text{in $A$,}
$$
where the second equality is due to Lemma~\ref{gen1}. Thus $d_1\circ d_2=0$. Finally, by definition of $d_3$,
$$
d_2\circ d_3(u)_j=\sum_{k=1}^n  (\delta_{x_j}\delta^R_{x_k} F)x_ku=\delta_{x_j} Fu=0\ \ \text{in $A$,}
$$
where the second equality follows from Lemma~\ref{gen1}. Thus (\ref{comq}) is indeed a complex. Its exactness at $\K$ and at the rightmost $A$ is trivial. It remains to check that (\ref{comq}) is exact at the third term from the right. In order to do this, we have to show that if $u=(u_1,\dots,u_n)\in A^n$ and $x_1u_1+{\dots}+x_nu_n=0$, then $u=d_2(v)$ for some $v\in A^n$. Let $I$ be the ideal of relations of $A$. Pick $a_1,\dots,a_n\in\K\langle x_1,\dots,x_n\rangle$ such that $u_j=a_j+I$ for all $j$. Since $x_1u_1+{\dots}+x_nu_n=0$ in $A$, we have $x_1a_1+{\dots}+x_na_n\in I$. Hence,
$$
x_1a_1+{\dots}+x_na_n=\delta^R_{x_1}Fs_1+{\dots}+\delta^R_{x_n}Fs_n+x_1p_1+{\dots}+x_np_n,\ \ \text{in $\K\langle x_1,\dots,x_n\rangle$, where $p_j\in I$}
$$
and $s_j\in\K\langle x_1,\dots,x_n\rangle$. By Lemma~\ref{gen1}, $\delta^R_{x_j}F=\sum\limits_{k=1}^n x_k\delta_{x_k}\delta^R_{x_j}F$. Plugging this into the above display, we get
$$
\sum_{k=1}^n x_k\Bigl(a_k-p_k-\sum_{j=1}^n \delta_{x_k}\delta^R_{x_j}Fs_j \Bigr)=0\ \ \text{in $\K\langle x_1,\dots,x_n\rangle$}.
$$
Hence
$$
a_k-p_k-\sum_{j=1}^n \delta_{x_k}\delta^R_{x_j}Fs_j =0\ \ \text{in $\K\langle x_1,\dots,x_n\rangle$ for $1\leq k\leq n$}.
$$
Factoring out $I$, we obtain $u=d_2(v)$ with $v_j=s_j+I\in A$.
\end{proof}

\subsection{Minimal series of graded twisted potential algebras}

Recall that a twisted potential algebra $A_F$ is called {\it exact} if the corresponding sequence (\ref{comq}) is an exact complex. For an algebra $A$ generated by $x_1,\dots,x_n$, we say that $u\in A$ is a {\it right annihilator} if $x_ju=0$ for $1\leq j\leq n$. A right annihilator $u$ is {\it non-trivial} if $u\neq 0$.

\begin{lemma}\label{commin} Let $F\in{\cal P}^*_{n,k}$ with $n\geq 2$, $k\geq 3$ and let $A=A_F$. Then the following statements are equivalent$:$
\begin{itemize}\itemsep=-2pt
\item[\rm (1)] $A$ is an exact twisted potential algebra$;$
\item[\rm (2)] $A$ has no non-trivial right annihilators and $H_A=(1-nt+nt^{k-1}-t^k)^{-1}$.
\end{itemize}
Moreover, if $A$ is exact, then $A$ is proper. Finally,
\begin{equation}\label{exko}
\text{if $k=3$, then}\qquad
\left\{\begin{array}{l}
\text{$A$ is exact $\Longrightarrow$ $A$ is Koszul,}\\
\text{$A$ proper and Koszul $\Longrightarrow$ $A$ is exact.}
\end{array}\right.
\end{equation}
\end{lemma}

\begin{proof}Assume that $A$ is exact. Denote $a_j=\dim A_j$ and set $a_j=0$ for $j=-1$. Since both defining relations of $A$ are of degree $k-1$, $a_j=n^j$ for $0\leq j<k-1$, exactness of (\ref{comq}) yields the recurrent equality $a_{m+k}-na_{m+k-1}+na_{m+1}-a_m=0$ for $m\geq-1$. Together with the initial data  $a_j=n^j$ for $0\leq j<k-1$ and $a_{-1}=0$, this determines $a_n$ for $n\geq 0$ uniquely, yielding $H_A=(1-nt+nt^{k-1}-t^k)^{-1}$. Since (\ref{comq}) is exact, $d_3$ is injective and therefore $A$ has no non-trivial right annihilators.

Now assume that  $H_A=(1-nt+nt^{k-1}-t^k)^{-1}$ and that $A$ has no non-trivial right annihilators. Then the map $d_3$ in (\ref{comq}) is injective and therefore the complex is exact at the leftmost $A$. By Lemma~\ref{gen3}, the only place, where the exactness may fail is the leftmost $A^n$. Considering the graded slices of the complex and using the exactness of the complex everywhere after the left $A^n$, we can compute the dimension of the intersection of the kernel of $d_2$ with $A_{m+1}^n$, which is $a_{m+k}-na_{m+k-1}+na_{m+1}$. On the other hand, $\dim d_3(A_m)=a_m$. Since $a_m$ are the Taylor coefficients of $(1-nt+nt^{k-1}-t^k)^{-1}$, they satisfy $a_{m+k}-na_{m+k-1}+na_{m+1}-a_m=0$, which proves that the above image and kernel have the same dimension and therefore coincide. Thus the exactness extends to the missing term. The fact that $A$ is proper when exact now follows from Lemma~\ref{gen2} (just look at $\dim A_k$). The  Koszulity statement is a consequence of (\ref{koex}).
\end{proof}

\begin{lemma}\label{commin01} Let $n\geq 2$, $k\geq 3$, $d\in\N$, $P_d=\K[t_1,\dots,t_d]$ and $F\in P_d\langle x_1,\dots,x_n\rangle$ be homogeneous of degree $k$ and such that for every $a\in\K^d$, specification $t_j=a_j$ for $1\leq j\leq d$ makes $F$ into a twisted potential. We denote by $A^a$ the corresponding twisted potential algebra. Assume also that $\K$ is uncountable and that there is $a_0\in\K^d$ such that $A^{a_0}$ is an exact twisted potential algebra. Then $A^a$ is exact for generic $a\in\K^d$ and the variety $W=\{A^a:a\in\K^d\}$ of graded algebras satisfies $H_W=(1-nt+nt^{k-1}-t^k)^{-1}$.
\end{lemma}

\begin{proof} Applying Lemma~\ref{complex} to the sequence (\ref{comq}), we see that $A^a$ is exact for generic $a$. By Lemma~\ref{commin}, $H_{A^a}=(1-nt+nt^{k-1}-t^k)^{-1}$ for generic $a$. By Lemma~\ref{miser}, $H_W=(1-nt+nt^{k-1}-t^k)^{-1}$.
\end{proof}

\subsection{Examples of exact potential algebras}

We start with an observation, which saves us the trouble of doing some computations.

\begin{lemma}\label{sy}Let $F\in{\cal P}^*_{n,k}$ with $n\geq 2$ and $k\geq 3$ and assume that monomials in $x_j$ are equipped with a  well-ordering compatible with multiplication. Then the leading monomials of the defining relations $r_j=\delta_{x_j}F$ of $A=A_F$ exhibit at least one overlap of degree $k$. Furthermore, if they have exactly one overlap, then this overlap has degree $k$ and it resolves. The latter means that the defining relations themselves form a reduced Gr\"obner basis in the ideal of relations of $A$.
\end{lemma}

\begin{proof} The equation (\ref{syz1}) provides a non-trivial linear dependence of $x_jr_m$ and $r_mx_j$ for $1\leq j,m\leq n$. Since the degree of each $r_m$ is $k-1>1$, this dependence produces a non-trivial syzigy (a syzigy outside the submodule generated by trivial ones). By Remark~\ref{syzz}, there must be at least one overlap of degree $k$, which resolves. Since the defining relations have degree $k-1$ and all other members of the Gr\"obner basis must have higher degrees, there should be at least one degree $k$ overlap of the leading monomials of the defining relations, which resolves. The result follows.
\end{proof}

\begin{example}\label{potex} Let $n$ and $k$ be integers such that $k\geq n\geq 2$, $k\geq 3$ and $(n,k)\neq (2,3)$. Consider the potential $F\in{\cal P}_{n,k}$ given by
$$
F=\sum_{\sigma\in S_{n-1}} x_n^{k-n+1}x_{\sigma(1)}\dots x_{\sigma(n-1)}^\rcirclearrowleft,
$$
where the sum is taken over all bijections from the set $\{1,\dots,n-1\}$ to itself.
Then the potential algebra $A=A_F$ is exact. Furthermore, $x_1u\neq 0$ for every non-zero $u\in A$.
\end{example}

\begin{proof} We order the generators by $x_n>x_{n-1}>{\dots}>x_1$ and equip monomials with the left-to-right degree-lexicographical ordering. The leading monomials of the defining relations of $A$ are easily seen to be $m_n=x_n^{k-n}x_{n-1}\dots x_1$, and $m_j=x_n^{k-n+1}x_{n-1}\dots x_{j+1}x_{j-1}\dots x_1$ with $1\leq j\leq n-1$ (after $x_n^{k-n+1}$  we have all other $x_k$ in descending order with $x_j$ missing). There is just one overlap of these monomials $x_n^{k-n+1}x_{n-1}\dots x_1=x_nm_n=m_1x_1$. By Lemma~\ref{sy}, it happily resolves and the defining relations themselves form a (finite) reduced Gr\"obner basis in the ideal of relations of $A$. This allows to confirm that $H_A=(1-nt+nt^{k-1}-t^k)^{-1}$. Since no leading monomial of the elements of the Gr\"obner basis starts with $x_1$, the map $u\mapsto x_1u$ from $A$ to $A$ is injective. Hence $A$ has no non-trivial right annihilators and therefore $A$ is exact according to Lemma~\ref{commin}.
\end{proof}

\begin{example}\label{potex1} Let $n$ and $k$ be integers such that $n>k\geq 3$. Order the generators by $x_n>x_{n-1}>{\dots}>x_1$ and consider the left-to-right degree-lexicographical ordering on the monomials. Consider the set $M$ of degree $k-2$ monomials in $x_1,\dots,x_{n-1}$ in which each letter $x_j$ features at most once. Let $m_1,\dots,m_{n-1}$ be the top $n-1$ monomials in $M$ enumerated in such a way that $m_{n-k+1}=x_{n-1}\dots x_{n-k+2}$ $($the biggest one$)$. Now define the potential  $F\in{\cal P}_{n,k}$  by
$$
F=x_nx_{n-1}\dots x_{n-k+1}^\rcirclearrowleft+\sum_{1\leq j\leq n-1\atop j\neq n-k+1} x_jx_nm_j^\rcirclearrowleft.
$$
Then the potential algebra $A=A_F$ is exact.  Furthermore, $x_1u\neq 0$ for every non-zero $u\in A$.
\end{example}

\begin{proof} We use the same order as above. The leading monomials of the defining relations of $A$ are easily seen to be $x_{n-1}\dots x_{n-k+1}$ and $x_nm_j$ for $1\leq j\leq n-1$. Again, there is just one overlap of these monomials
$x_nx_{n-1}\dots x_{n-k+1}=x_n(x_{n-1}\dots x_{n-k+1})=(x_nx_{n-1}\dots x_{n-k+2})x_{n-k+1}$, which happens to resolve according to Lemma~\ref{sy}. The rest of the proof is the same as for the previous example.
\end{proof}

Combining Examples~\ref{potex} and~\ref{potex1} with  Lemma~\ref{commin01}, we get  the following result.

\begin{lemma}\label{commin1} Let $n\geq 2$, $k\geq 3$ and $(n,k)\neq (2,3)$. Then $H_W=(1-nt+nt^{k-1}-t^k)^{-1}$ for the variety
$W=\{A_F:F\in{\cal P}_{n,k}\}$. Furthermore, if $\K$ is uncountable, then for a generic $F\in{\cal P}_{n,k}$, $A=A_F$ is an exact potential algebra and satisfies $H_A=(1-nt+nt^{k-1}-t^k)^{-1}$.
\end{lemma}

\subsection{Algebras $A_F$ with $F\in{\cal P}^*_{2,3}$}

Note that the case $(n,k)=(2,3)$ is indeed an odd one out. It turns out that in this case there are no exact twisted potential algebras at all and the formula for the minimal series fails to follow the pattern as well. For aesthetic reasons, we use $x$ and $y$ instead of $x_1$ and $x_2$.

\begin{proposition}\label{commin2} There are just four pairwise non-isomorphic algebras in the variety $W=\{A_F:F\in {\cal P}_{2,3}\}$. These are the algebras corresponding to the potentials $F=0$, $F=x^3$, $F={xy^2}^\rcirclearrowleft$ and $F=x^3+y^3$. Their Hilbert series are $(1-2t)^{-1}$, $\frac{1+t}{1-t-t^2}$ and $\frac{1+t}{1-t}$ for the last two algebras. All algebras in $W$ are PBW, Koszul, infinite dimensional and non-exact.
\end{proposition}

\begin{proof} An $F\in {\cal P}_{2,3}$ has the form $F=ax^3+b{xy^2}^\rcirclearrowleft+cx^2y^\rcirclearrowleft+dy^3$ with $a,b,c,d\in\K$. Then the abelianization of $F$ is $F^{\rm ab}=ax^3+3bxy^2+3cx^2y+dy^3\in\K[x,y]$. Since $\K$ is not of characteristic $3$, $F$ recovers uniquely from its abelianization. Now since a degree $3$ cubic form on two variables is a product of three linear forms (provided $\K$ is algebraically closed), we see that by a linear substitution $F^{\rm ab}$ can be turned into one of the following forms $x^3$, $xy^2$ or $x^3+y^3$ unless it was zero to begin with. This corresponds to $F$ turning into one of the four potentials listed in the statement of the lemma by means of a linear substitution.  If $F=0$, $A_F$ is the free algebra on two generators and $H_A=(1-2t)^{-1}$. If $F=x^3$, $A_F$ is defined by one relation $x^2$, which forms a one-element Gr\"obner basis. In this case $H_{A_F}=\frac{1+t}{1-t-t^2}$. If $F={xy^2}^\rcirclearrowleft$ or $F=x^3+y^3$, the defining relations of $A_F$ are $xy+yx$ and $y^2$ or $x^2$ and $y^2$ respectively. Again, they form a Gr\"obner basis, yielding  $H_{A_F}=\frac{1+t}{1-t}$. Since the last two algebras are easily seen to be non-isomorphic and the Hilbert series of the first three are pairwise distinct, the four algebras are pairwise non-isomorphic. As all four algebras have quadratic Gr\"obner basis, they are PBW and therefore Koszul. Obviously, they are infinite dimensional. If any of these algebras were exact, Lemma~\ref{commin} would imply that its Hilbert series is $(1-2t+2t^{2}-t^3)^{-1}=(1-t)^{-1}(1-t+t^2)^{-1}$, which does not match any of the above series. Thus they are all non-exact.
\end{proof}

We extend the above proposition to include twisted potential algebras.

\begin{proposition}\label{commin22} Any non-potential twisted potential algebra $A$ on two generators given by a homogeneous degree $3$ twisted potential is isomorphic to either $A_G$ or $A_{G_\alpha}$ with $\alpha\in\K\setminus\{0,1\}$, where $G=x^2y-xyx+yx^2+y^3$ and $G_\alpha=x^2y+\alpha xyx+\alpha^2yx^2$. $A_G$ is non-isomorphic to any of $A_{G_\alpha}$ and $A_{G_\alpha}$ is non-isomorphic to $A_{G_\beta}$ if $\alpha\neq \beta$. Furthermore all these algebras are non-degenerate, infinite dimensional, non-exact, PBW, Koszul and have the Hilbert series $\frac{1+t}{1-t}$.
\end{proposition}

\begin{proof} We know that $A=A_F$, where $F$ is the corresponding twisted potential. If $\delta_x F$ and $\delta_y F$ are linearly dependent, one easily sees that either $F=0$ or $F$ is the cube of a degree one homogeneous element. In either case $A$ is potential, which contradicts the assumptions. Thus $\delta_x F$ and $\delta_y F$ are linearly independent and therefore $F$ is non-degenerate. Let $M\in GL_2(\K)$ be the matrix providing the twist. By Remark~\ref{rer}, we can assume that $M$ is in Jordan normal form. Note that $M$ is not the identity matrix (otherwise $A$ is potential). Let $\alpha$ and $\beta$ be the eigenvalues of $M$ and $T(M)$ be the set of all homogeneous degree $3$ non-degenerate twisted potentials on two generators, whose twist is provided by $M$. One easily sees that this set is empty unless $1\in\{\alpha^2\beta,\alpha\beta^2\}$. Without loss of generality, we can assume that $\alpha^2\beta=1$. If $M$ is one Jordan block, we must have $\alpha=\beta$ and $\alpha^3=1$. If in this case $\alpha=1$, we again easily see that $T(M)$ is empty. Thus it remains to consider the following options for $M$:
$$
M=N_\alpha=\left(\begin{array}{ll} \alpha&1\\0&\alpha\end{array}\right)\ \text{with $\alpha^3=1\neq\alpha$\ \ and}\ \
M=M_\alpha=\left(\begin{array}{ll} \alpha&0\\0&\alpha^{-2}\end{array}\right)\ \text{with $\alpha\neq 1$.}
$$
In the case $M=M_\alpha$, $T(M)$ contains $x^2y+\alpha xyx+\alpha^2yx^2$ and consists only of its scalar multiples unless $\alpha^3=1$ or $\alpha=-1$. In the case $\alpha^3=1\neq\alpha$, $T(M)$ sits in the two-dimensional space spanned by $x^2y+\alpha xyx+\alpha^2yx^2$ and $xy^2+\alpha^2yxy+\alpha y^2x$. If $\alpha=-1$, $T(M)$ is contained in the two-dimensional space spanned by $x^2y-xyx+yx^2$ and $y^3$. In the case $M=N_\alpha$ with $\alpha^3= 1\neq\alpha$, $T(M)$ consists of scalar multiples of $\frac{x^3}{1-\alpha}+yx^2+\alpha x^2y+\alpha^2 xyx$. All this is obtained by translating (\ref{qua}) into a system of linear equations on coefficients of $F$ and solving it.

When $F=x^2y+\alpha xyx+\alpha^2yx^2=G_\alpha$, the defining relations of $A=A_F$ are $xx$ and $xy+\alpha yx$ with
$\alpha\in\K^*$, $\alpha\neq 1$. It is easy to see that different $\alpha$ correspond to non-isomorphic algebras. Indeed, a linear substitution providing an isomorphism must map $x$ to its scalar multiple ($xx$ is the only square among quadratic relations) and taking this into account, it is easy to see that $y$ also must be mapped to its own scalar multiple. Such a substitution preserves the space of defining relations and therefore $\alpha$ is an isomorphism invariant. It is easy to see that the defining relations form a Gr\"obner basis in the ideal of relations. Hence $A$ is PBW and Koszul. It is easy to see that $H_A=\frac{1+t}{1-t}$.

If $F=s(x^2y+\alpha xyx+\alpha^2yx^2)+t(xy^2+\alpha^2yxy+\alpha y^2x)$ with $s,t\in\K$, $(s,t)\neq(0,0)$ and $\alpha^3=1\neq\alpha$, we have options. If $st=0$, then by means of a scaling combined with the swap of $x$ and $y$ in the case $s=0$, we can transform $F$ into $G_\alpha$. If $st\neq 0$, a scaling reduces considerations to the case $t=s=1$. Then $F=x^2y+\alpha xyx+\alpha^2yx^2+xy^2+\alpha^2yxy+\alpha y^2x$. The substitution $x\to x+\alpha y$, $y\to y$, provides an isomorphism of $A=A_F$ and $A_{G_\alpha}$. It remains to consider $F=s(x^2y-xyx+yx^2)+ty^3$ with $s,t\in\K$. If $s=0$, $A$ is potential. Thus $s\neq 0$. If $t=0$ and $s\neq 0$, then up to a scalar multiple, $F=G_{-1}$. Thus we can assume that $st\neq0$. By a scaling, we can turn both $s$ and $t$ into $1$, which transforms $F$ into $G$. The defining relations of $A=A_F$ are now $xy-yx$ and $x^2+y^2$. This time the space of quadratic relations fails to contain a square of a degree one element and therefore the corresponding algebra is not isomorphic to any of $A_{G_\alpha}$. Again, the defining relations form a Gr\"obner basis in the ideal of relations. Hence $A$ is PBW and Koszul and $H_A=\frac{1+t}{1-t}$. This series fails to coincide with $(1-2t+2t^2-t^3)^{-1}$ and therefore none of these algebras is exact according to Lemma~\ref{commin}.
\end{proof}

Proposition~\ref{commin22} and Remark~\ref{rem0} provide a complete description of degenerate twisted potential non-potential algebras on three generators with homogeneous degree $3$ twisted potentials. Indeed, the latter are free products of algebras from Proposition~\ref{commin22} with the algebra of polynomials on one variable. This observation is recorded as follows.

\begin{lemma}\label{23} $A$ is a non-potential twisted potential algebra on three generators given by a homogeneous degree $3$ degenerate twisted potential if and only if $A$ is isomorphic to an algebra from {\rm (T22)} or {\rm (T23)} of Theorem~$\ref{main33-t}$. The algebras with different labels are non-isomorphic and the information in the table from Theorem~$\ref{main33-t}$ concerning {\rm (T22)} and {\rm (T23)} holds true.
\end{lemma}

\subsection{Lower estimate for $P_A$, $A$ being a potential algebra}

The methods we develop and apply in this section work for many varieties of twisted potential algebras. We restrict ourselves to potential algebras for the sake of clarity. The main objective of this section is to prove Theorem~\ref{main-growth}. For $n,k,m\in\N$ such that $n\geq 2$ and $m\geq k\geq 3$, denote
$$
{\cal P}_{n,k}^{(m)}=\{F\in \K^{\rm cyc}\langle x_1,\dots,x_n\rangle:F_j=0\ \text{for $j<k$ and for $j>m$}\}.
$$
Clearly, ${\cal P}_{n,k}^{(m)}$ is a vector space and ${\cal P}_{n,k}^{(k)}={\cal P}_{n,k}$. Recall that for $j\in\Z_+$ and $F\in {\cal P}_{n,k}^{(m)}$, $A_F^{(j)}$ is the quotient of $A_F$ by the ideal generated by the monomials of degree $j+1$.

\begin{lemma}\label{abc} Let $n,k\in\N$, $n\geq 2$, $m\geq k\geq 3$ and $(n,k)\neq (2,3)$ and $F\in {\cal P}_{n,k}^{(m)}$. Assume also that $x_1a\neq 0$ in $A_{F_k}$ for every non-zero $a\in A_{F_k}$. Then for each $j\in\Z_+$, $x_1b\neq 0$ in $A_F^{(j+1)}$ for every $b\in \K\langle x_1,\dots,x_n\rangle$ such that $b\neq 0$ in $A_F^{(j)}$.
\end{lemma}

\begin{proof} Assume the contrary. Then there exist $j\in\Z_+$ and $a\in \K\langle x_1,\dots,x_n\rangle$ such that $a\neq 0$ in $A_F^{(j)}$ and $x_1a=0$ in $A_F^{(j+1)}$. The latter means that
\begin{equation*}
x_1a=\sum_{j\in N} u_jr_{s(j)}v_j\,({\rm mod}\,\,J^{(j+1)}),
\end{equation*}
where $r_j=\delta_{x_j}F$, $N$ is a finite set, $s$ is a map from $N$ to $\{1,\dots n\}$, $u_j,v_j$ are non-zero homogeneous elements of $\K\langle x_1,\dots,x_n\rangle$ such that the degree of each $u_jv_j$ does not exceed $j-k+2$ and the equality $f=g\,({\rm mod}\,\,J)$ means $f-g\in J$. Let $m$ be the smallest degree of $u_jv_j$ and $N'=\{j\in N:\deg u_jv_j=m\}$. Then the smallest degree part of the above display reads
$$
x_1a_{m+k-2}=\sum_{j\in N'} u_j\rho_{s(j)}v_j\ \ \text{in $\K\langle x_1,\dots,x_n\rangle$,}
$$
where $\rho_j=\delta_{x_j}F_k$. Note that automatically $a_q=0$ for $q<m+k-2$. The condition imposed upon $A_{F_k}$ means that the ideal $K$ generated by $\rho_1,\dots,\rho_n$ satisfies $x_1b\in K\,\Longrightarrow b\in K$. Hence, by the above display,
$$
a_{m+k-2}=\sum_{p\in M} f_p\rho_{t(p)}g_p,
$$
where $M$ is a finite set $t$ is a map from $M$ to $\{1,\dots n\}$, $f_p,g_p$ are non-zero homogeneous elements of $\K\langle x_1,\dots,x_n\rangle$ such that the degree of each $f_pg_p$ is $m-1$. Now we replace $a$ by
$$
a'=a-\sum_{p\in M} f_pr_{t(p)}g_p.
$$
Note that $a=a'$ in $A_F$ and therefore $a=a'$ in $A_F^{(j)}$ and $x_1a=x_1a'$ in $A_F^{(j+1)}$. So $a'$ satisfies the same properties as $a$ with the only essential difference being that $a'_{m+k-2}=0$. Now we can repeat the process chipping off the homogeneous degree-components of $a$ from bottom up one by one until at the final step we arrive to a contradiction with $a\neq 0$ in $A_F^{(j)}$.
\end{proof}

\begin{lemma}\label{geni} Let $\K$ be uncountable, $n,k\in\N$, $n\geq 2$, $k\geq 3$ and $(n,k)\neq (2,3)$. Then for a generic $F\in{\cal P}_{n,k}$, $x_1a\neq 0$ in $A_F$ for every non-zero $a\in A_F$.
\end{lemma}

\begin{proof} Let $F_0$ be the potential provided by the appropriate (depending on whether $k\geq n$ or $k<n$) Example~\ref{potex} or Example~\ref{potex1}. Then $x_1a\neq 0$ in $A_{F_0}$ for every non-zero $a\in A_{F_0}$ and $H_{A_{F_0}}=(1-nt+nt^{k-1}-t^k)^{-1}$. Lemma~\ref{commin1} guarantees that $H_{A_F}=(1-nt+nt^{k-1}-t^k)^{-1}$ for generic $F\in{\cal P}_{n,k}$. Applying Lemma~\ref{maxra} to the map $a\mapsto x_1a$ from $A_F$ to $A_F$, we now see that $\dim x_1(A_F)_j\geq \dim x_1(A_{F_0})_j$ for all $j$ for generic $F\in{\cal P}_{n,k}$. Since $\dim x_1(A_{F_0})_j=\dim (A_{F_0})_j=\dim (A_{F})_j$ for generic $F$, the map $a\mapsto x_1a$ from $A_F$ to itself is injective for generic $F$.
\end{proof}

\begin{lemma}\label{abc1} Let $n,k\in\N$, $n\geq 2$, $m\geq k\geq 3$ and $(n,k)\neq (2,3)$, $F\in {\cal P}_{n,k}^{(m)}$ and $A=A_F$. Then $P_{A}\geq (1-t)^{-1}(1-nt+nt^{k-1}-t^k)^{-1}$. \end{lemma}

\begin{proof} First, observe that exchanging the ground field $\K$ for a field extension does not affect the series $P_A$. Thus we can without loss of generality assume that $\K$ is uncountable. For $j\in\Z_+$, let $b_j$ be Taylor coefficients of the rational function $Q(t)=(1-t)^{-1}(1-nt+nt^{k-1}-t^k)^{-1}$ (that is, $Q(t)=\sum b_jt^j$) and $a_j=\min\{\dim A_G^{(j)}:G\in {\cal P}_{n,k}^{(m)}\}$. The proof will be complete if we show that $a_j=b_j$ for all $j\in\Z_+$. Denote $P=\sum a_jt^j$. First, note that Examples~\ref{potex} and~\ref{potex1}, provide $G\in {\cal P}_{n,k}\subseteq {\cal P}_{n,k}^{(m)}$ for which $H_G=(1-nt+nt^{k-1}-t^k)^{-1}$. It immediately follows that $P_G=Q$. By definition of $P$ (minimality of $a_j$), we then have $P\leq Q$, that is, $a_j\leq b_j$ for all $j\in\Z_+$.

By Lemmas~\ref{miser1} and~\ref{geni}, for a generic $G\in {\cal P}_{n,k}^{(m)}$, we have that $P_{A_G}=P$ and $x_1a\neq 0$ in $A_{G_k}$ for every non-zero $a\in A_{G_k}$. In particular, we can pick a single $G\in {\cal P}_{n,k}^{(m)}$ such that for $B=A_G$, $P_B=P$ and $x_1a\neq 0$ in $A_{G_k}$ for every non-zero $a\in A_{G_k}$. According to Lemma~\ref{abc}, we then have that for each $j\in\Z_+$, $x_1b\neq 0$ in $B^{(j+1)}$ for every $b\in \K\langle x_1,\dots,x_n\rangle$ such that $b\neq 0$ in $B^{(j)}$.
This property allows us to pick inductively (starting with $M_0=\{1\}$) sets $M_j$ of monomials of degree $j$ such that $M_{j+1}\supseteq x_1M_j$ and $N_j=M_0\cup{\dots}\cup M_j$ is a linear basis in $B^{(j)}$ for each $j\in\Z_+$. For every $j$, let $B^+_j$ be the linear span of $N_j$ and $B^{++}_j$ be the linear span of $N_j\setminus x_1N_{j-1}$ in $\K\langle x_1,\dots,x_n\rangle$. Clearly $P_B=\sum (\dim B^+_j)t^j$ and therefore $a_j=\dim B^+_j$ for all $j\in\Z_+$. Let also $\pi^{(j)}$ be the natural projection of $\K\langle x_1,\dots,x_n\rangle$ onto the linear span of monomials of length $\leq j$ along $J^{(j)}$. As usual, let $V$ be the linear span of $x_1,\dots,x_n$, $r_j=\delta_{x_j}G$, $R$ be the linear span of $r_1,\dots,r_n$ and $I$ be the ideal generated by $r_1,\dots,r_n$ (=the ideal of relations of $B$). For the sake of brevity denote $\Phi=\K\langle x_1,\dots,x_n\rangle$. Obviously,
$I=VI+R\Phi$. Then $\pi^{(j+1)}(I)=V\pi^{(j)}(I)+\pi^{(j+1)}(R\Phi)$ for every $j\in\Z_+$. Using the definition of $B_j^+$ and the fact that each $r_j$ starts at degree $\geq k-1$, we obtain
$$
\pi^{(j+1)}(I)=V\pi^{(j)}(I)+\pi^{(j+1)}(RB^+_{j+2-k}).
$$
Since by Lemma~\ref{gen1}, $\sum[x_j,r_j]=0$ in $\Phi$, we can get rid of $r_1x_1$:
$$
V\pi^{(j)}(I)+\pi^{(j+1)}(RB^+_{j+2-k})=V\pi^{(j)}(I)+\pi^{(j+1)}(R'B^+_{j+2-k}+r_1B^{++}_{j+2-k})
$$
where $R'$ is the linear span of $r_2,\dots,r_n$. Thus
$$
\pi^{(j+1)}(I)=V\pi^{(j)}(I)+\pi^{(j+1)}(R'B^+_{j+2-k}+r_1B^{++}_{j+2-k}).
$$
Hence
$$
\begin{array}{l}
\dim \pi^{(j+1)}(I)\leq \dim V\pi^{(j)}(I)+\dim R'B^+_{j+2-k}+\dim r_1B^{++}_{j+2-k}
\\ \qquad\qquad =n\dim \pi^{(j)}(I)+(n-1)\dim B^+_{j+2-k}+\dim B^{++}_{j+2-k}.
\end{array}
$$
Plugging the equalities $\dim B^+_j=a_j$, $\dim B^{++}_j=a_j-a_{j-1}$ (assume $a_{s}=0$ for $s<0$) and $\dim \pi^{(j)}(I)=1+n+{\dots}+n^j-a_j$ into the inequality in the above display, we get
$$
1+{\dots}+n^{j+1}-a_{j+1}\leq n+{\dots}+n^{j+1}-na_j+na_{j+2-k}-a_{j+1-k}.
$$
Hence $a_{j+1}\geq na_j-na_{j+2-k}+a_{j+1-k}-1$ for $j\in\Z_+$. On the other hand, it is easy to see that the Taylor coefficients $b_j$ of $Q$ satisfy $b_{j+1}=nb_j-nb_{j+2-k}+b_{j+1-k}-1$ for $j\geq k-1$. It is also elementary to verify that $a_j=b_j$ for $0\leq j\leq k-1$. Now for $c_j=b_j-a_j$, we have $c_j=0$ for $0\leq j\leq k-1$, $c_j\geq 0$ for $j\geq k$ and $c_{j+1}\leq nc_j-nc_{j+2-k}+c_{j+1-k}$ for $j\geq k-1$. The only sequence satisfying these conditions is easily seen to be the zero sequence. Hence $a_j=b_j$ for all $j\in\Z_+$, which completes the proof.
\end{proof}

Now Theorem~\ref{main-growth} is a direct consequence of Lemma~\ref{abc1}. Indeed, every potential $F$ on $n$ variables starting in degree $\geq k$ belongs to ${\cal P}_{n,k}^{(m)}$ for $m$ large enough and Lemma~\ref{abc} kicks in providing at least cubic growth of $A_F$ in the case $(n,k)=(3,3)$ or $(n,k)=(2,4)$ and exponential growth otherwise.
We end this section with another observation concerning the growth of potential algebras. We say that $F\in \K^{\rm cyc}\langle x_1,\dots,x_n\rangle$ is {\it S-trivial} if the module of syzigies of $A_F$ presented by generators $x_1,\dots,x_n$ and relations $r_1,\dots,r_n$ with $r_j=\delta_{x_j}F$ is generated by trivial syzigies and the syzigy $\sum [x_j,\widehat{r}_j]$ provided by Lemma~\ref{gen1}.

\begin{lemma}\label{genii} Let $\K$, $n,k,m\in\N$ be such that $n\geq 2$, $m\geq k\geq 3$, $(n,m)\neq (2,3)$, $F\in {\cal P}_{n,k}^{(m)}$ and $A=A_F$. Assume also that $H_{A_{F_m}}=(1-nt+nt^{m-1}-t^m)^{-1}$. Then $P^*_{A}=(1-t)^{-1}(1-nt+nt^{m-1}-t^m)^{-1}$.
\end{lemma}

\begin{proof} First, we observe that if $G\in {\cal P}_{n,m}$ and $H_{A_G}=(1-nt+nt^{m-1}-t^m)^{-1}$, then $G$ is $S$-trivial. Indeed, otherwise an 'extra' syzigy will 'drop' the dimension of the corresponding component of the ideal of relations thus increasing the dimension of the component of the algebra compared to the minimal Hilbert series $(1-nt+nt^{m-1}-t^m)^{-1}$.
Now we equip monomials in $x_1,\dots,x_n$ with a well-ordering compatible with multiplication such that monomials of greater degree are always greater (for instance, we can use a degree-lexicographical ordering). We proceed to compare the corresponding reduced Gr\"obner basis in the ideals of relations for $A_F$ and $A_{F_m}$. Since $A_{F_m}$ is $S$-trivial (by the above observation) and $\sum [x_j,\widehat{r}_j]$ is a syzigy for $A_F$ anyway, it is easy to see that the elements of the reduced Gr\"obner basis in the ideal of relations for $A_{F_m}$ are exactly the highest degree components of the elements of the reduced Gr\"obner basis in the ideal of relations for $A_{F}$. In particular, the leading monomials are the same and the exact same overlaps resolve. Hence normal words for $A_F$ and $A_{F_m}$ are the same and therefore $P^*_{A}=P^*_{A_{F_m}}$. Since $H_{A_{F_m}}=(1-nt+nt^{m-1}-t^m)^{-1}$, we have $P^*_{A_{F_m}}=(1-t)^{-1}(1-nt+nt^{m-1}-t^m)^{-1}$ and the result follows.
\end{proof}

\section{Potential algebras $A_F$ for $F\in{\cal P}_{2,4}$}

Throughout this section we equip the monomials in $x,y$ with the left-to-right degree-lexicographical ordering assuming $x>y$. The following statement is elementary.

\begin{lemma}\label{cyab24}
The kernel of the canonical homomorphism from $\K\langle x,y\rangle$ onto $\K[x,y]$ $(=$abelianization$)$ intersects ${\cal P}_{2,4}$ by the one-dimensional space spanned by ${x^2y^2}^\rcirclearrowleft-xyxy^\rcirclearrowleft$.
\end{lemma}

We shall use the following observation on a number of occasions.

\begin{lemma}\label{smth}Let $F\in{\cal P}^*_{2,4}$ be such that $x^2y^2$ and $yx^2y$ are in $F$ with non-zero coefficients, while the monomials $x^4$, $x^3y$, $x^2yx$ and $yx^3$ do not occur in $F$. Then $A=A_F$ is exact and $H_A=(1+t)^{-1}(1-t)^{-3}$. Similarly, if $F\in{\cal P}^*_{2,4}$ contains $y^2x^2$ and $xy^2x$ with non-zero coefficients, while $y^4$, $y^3x$, $y^2xy$ and $xy^3$ do not feature in $F$, then $A=A_F$ is exact and $H_A=(1+t)^{-1}(1-t)^{-3}$.
\end{lemma}

\begin{proof} The two statements are clearly equivalent (just swap $x$ and $y$). Thus we may assume that $x^2y^2$ and $yx^2y$ are in $F$ with non-zero coefficients and $F$, while the monomials $x^4$, $x^3y$, $x^2yx$ and $yx^3$ do not occur in $F$. Then $xy^2$ is the leading monomial of $\delta_xF$, while $x^2y$ is the leading monomial of $\delta_yF$. Since these monomials exhibit just one overlap $x^2y^2=(x^2y)y=x(xy^2)$, Lemma~\ref{sy} implies that $\delta_xF$ and $\delta_yF$ form a Gr\"obner basis in the ideal of relations of $A$. It immediately follows that $H_A=(1+t)^{-1}(1-t)^{-3}$ (the corresponding normal words are $y^n(xy)^mx^k$ with $n,m,k\in\Z_+$) and that $A$ has no non-trivial right annihilators (there is no leading monomials of a Gr\"obner basis starting with $y$ and therefore the map $u\mapsto yu$ from $A$ to $A$ is injective). By Lemma~\ref{commin}, $A$ is exact.
\end{proof}

\begin{lemma}\label{nco2-4} Let $F\in{\cal P}_{2,4}$. Then by means of a linear substitution $F$ can be turned into one of the following forms$:$

\medskip
\noindent $\begin{array}{ll}\text{{\rm (H1)}\ \ $F=0;$\qquad{\rm (H2)}\ \ $F=x^4;$\qquad}&\text{{\rm (H6)}\ \ $F=x^4+y^4;$}\\
\text{{\rm (H3)}\ \ $F=x^4+\frac12xyxy^\rcirclearrowleft;\qquad$}&\text{{\rm (H7)}\ \ $F=x^3y^\rcirclearrowleft+{x^2y^2}^\rcirclearrowleft-xyxy^\rcirclearrowleft;$}\\
\text{{\rm (H4)}\ \ $F=\frac12xyxy^\rcirclearrowleft;$}&\text{{\rm (H8)}\ \ $F_a=x^4+{x^2y^2}^\rcirclearrowleft+\frac{a}{2}xyxy^\rcirclearrowleft$ with $a\in\K;$}\\
\text{{\rm (H5)}\ \ $F=x^3y^\rcirclearrowleft;$}&\text{{\rm (H9)}\ \ $F_a={x^2y^2}^\rcirclearrowleft+\frac{a}{2}xyxy^\rcirclearrowleft$ with $a\in\K;$}\\
\text{\rlap{{\rm \!\!\!(H10)}\ \ $F_{a,b}=x^4+a{x^2y^2}^\rcirclearrowleft\!\!+bxyxy^\rcirclearrowleft\!\!+y^4$ with $a,b\in\K$, $4(a+b)^2\neq1$, $(a,b)\notin\bigl\{(0,0),\bigl(1,\frac12\bigr),\bigl(-1,-\frac12\bigr)\bigr\}$.}}&
\end{array}$
\medskip

Moreover, to which of the above $10$ forms $F$ can be turned into is uniquely determined by $F$. Similarly, the parameter $a$ in {\rm (H8)} and {\rm (H9)} is uniquely determined by $F$. As for the last option, $F_{a,b}$ and $F_{a',b'}$ can be obtained from one another by a linear substitution if and only if they belong to the same orbit of the group action generated by two involutions $(a,b)\mapsto (-a,-b)$ and $(a,b)\mapsto\left(\frac{1-2b}{1+2a+2b},\frac{1-2a+2b}{2(1+2a+2b)}\right)$. This group has $6$ elements and is isomorphic to $S_3$.
\end{lemma}

\begin{proof} Let $F\in{\cal P}_{2,4}$. First, we show that $F$ can be turned into exactly one of (H1--H10) by a linear sub. Let $G\in\K[x,y]$ be the abelianization of $F$. By Lemma~\ref{co2-4} by means of a linear substitution $G$ can be turned into exactly one of the forms (C1--C6). Thus we can assume from the start that $G$ is in one of the forms (C1--C6).

Case 1: $G=0$. By Lemma~\ref{cyab24}, $F=s({x^2y^2}^\rcirclearrowleft-xyxy^\rcirclearrowleft)$ with $s\in\K$. If $s=0$, $F$ is given by (H1). If $s\neq 0$, a scaling brings $F$ to the form (H9) with $a=-2$.

Case 2: $G=x^4$. By Lemma~\ref{cyab24},  $F=x^4+s({x^2y^2}^\rcirclearrowleft-xyxy^\rcirclearrowleft)$ with $s\in\K$. If $s=0$, $F$ is given by (H2). If $s\neq 0$, a scaling brings $F$ to the form (H8) with $a=-2$.

Case 3: $G=x^3y$. By Lemma~\ref{cyab24},  $F=\frac14x^3y^\rcirclearrowleft+s({x^2y^2}^\rcirclearrowleft-xyxy^\rcirclearrowleft)$ with $s\in\K$. If $s=0$, $F$ acquires the form (H5) after scaling. If $s\neq 0$, a scaling  brings $F$ to the form (H7).

Case 4: $G=x^2y^2$. By Lemma~\ref{cyab24},  $F=\frac14{x^2y^2}^\rcirclearrowleft+s({x^2y^2}^\rcirclearrowleft-xyxy^\rcirclearrowleft)$ with $s\in\K$. If $1+4s=0$, $F$ acquires the form (H4) after scaling. Otherwise, a scaling brings $F$ to the form (H9) with $a\neq-2$.

Case 5: $G=x^4+x^2y^2$. By Lemma~\ref{cyab24}, $F=x^4+\frac14{x^2y^2}^\rcirclearrowleft+s({x^2y^2}^\rcirclearrowleft-xyxy^\rcirclearrowleft)$ with $s\in\K$. If $4s+1=0$, $F$ acquires the form (H3) after scaling. Otherwise, a scaling brings $F$ to the form (H8) with $a\neq-2$.

Case 6: $G=x^4+cx^2y^2+y^4$ with $c^2\neq 4$. By Lemma~\ref{cyab24}, $F=x^4+\frac{c}{4}{x^2y^2}^\rcirclearrowleft+y^4+s({x^2y^2}^\rcirclearrowleft-xyxy^\rcirclearrowleft)$ with $s\in\K$. That is, $F=F_{a,b}=x^4+a{x^2y^2}^\rcirclearrowleft\!\!+bxyxy^\rcirclearrowleft\!\!+y^4$ with $a=s+\frac c4$ and $b=-s$. The condition $c^2\neq 4$ translates into $4(a+b)^2\neq1$. By Lemma~\ref{S3} $F_{0,0}$, $F_{1,1/2}$ and $F_{-1,-1/2}$ are all equivalent and are non-equivalent to any other $F_{a,b}$. Thus if $(a,b)\in\bigl\{(0,0),\bigl(1,\frac12\bigr),\bigl(-1,-\frac12\bigr)\bigr\}$, $F$ is equivalent to the potential from (H5). Otherwise, $F$ is in (H10).

The fact that $F$ with different labels from the list (H1--H10) are non-equivalent (can not be obtained from one another by a linear sub) follows from the above observations, the non-equivalence of polynomials with the different labels from the list (C1--C6) as well as the trivial observation that a symmetric (not just cyclicly) element of ${\cal P}_{2,4}$ can not be equivalent to a non-symmetric one. It remains to prove the statements about equivalence within each of (H8), (H9) and (H10). The latter follows directly from Lemma~\ref{S3}. It remains to deal with (H7) and (H8). We may remove the exceptional cases $a=-2$ from consideration. One easily sees that the only subs that transform an $F_a$ with $a\neq -2$ from (H7) (respectively, (H8)) to another $F_{a'}$ from (H7) (respectively, (H8)) up to scalar multiples are scalings combined with a possible swap of $x$ and $y$ in the (H8) case. Since none of the latter has any effect on the parameter, we have $a=a'$. It follows that distinct $F_a$ from (H7) or (H8) are non-equivalent.
\end{proof}

\begin{lemma}\label{easy1}
Let $F\in{\cal P}_{2,4}$ be a potential from {\rm (Hj)} of Lemma~$\ref{nco2-4}$ for $1\leq j\leq 6$. Then the potential algebra $A=A_F$ is non-exact. Its Hilbert series is $H_A=\frac{1+t+t^2}{1-t-t^2}$ for $j\in\{4,5,6\}$, $H_A=\frac{(1+t^2)(1-t^5)}{(1-t)(1-t-t^4-t^5)}$ for $j=3$, $H_A=\frac{1+t+t^2}{1-t-t^2-t^3}$ for $j=2$ and $H_A=\frac1{1-2t}$ for $j=1$. If $j=3$, $A$ is proper, while $A$ is non-proper in all other cases.
\end{lemma}

\begin{proof} It is straightforward to verify that the defining relations themselves form a Gr\"obner basis in the ideal of relations for all $F$ under consideration except for $F$ from (H3). In the case $j=3$, we swap $x$ and $y$ to begin with. After this the reduced Gr\"obner basis in the ideal of relations comprises $yxy$, $xyx+y^3$ and $y^4$. In each case, knowing a finite Gr\"obner basis (more specifically, knowing the leading monomials of its members), it is a routine calculation to find $H_A$ in the form of a rational function to confirm the required formulae. Since none of the resulting series coincides with $(1+t)^{-1}(1-t)^{-3}$, Lemma~\ref{ms1} implies that $A$ is non-exact. By Lemma~\ref{gen2}, $A$ is proper if and only if $\dim A_4=9$. Knowing the Hilbert series, we see that this happens precisely when $A=A_F$ with $F$ given by (H3).
\end{proof}

The following lemma is a special case of Lemmas~\ref{commin} and~\ref{commin1}.

\begin{lemma}\label{ms1}
Let $F\in{\cal P}_{2,4}$ and $A=A_F$. Then $H_A\geq (1+t)^{-1}(1-t)^{-3}$. Furthermore, $A$ is exact if and only if $H_A=(1+t)^{-1}(1-t)^{-3}$ and $A$ has no non-trivial right annihilators.
\end{lemma}

\begin{lemma}\label{wei} Let $F\in{\cal P}_{2,4}$ be either from {\rm (H7--H9)} or from {\rm (H10)} of Lemma~$\ref{nco2-4}$ with $(a,b)$ such that
$ab(a^2-1)(4b^2-a^2)(4b^2-1)(4b^2-a^4)(4b^2-2a^2+1)=0$. Then $A=A_F$ is exact and satisfies $H_A=(1+t)^{-1}(1-t)^{-3}$.
\end{lemma}

\begin{proof} If $F$ is from {\rm (H7--H9)}, the result follows directly from Lemma~\ref{smth}. It remains to deal with the case of $F$ given by (H10).

{\bf Case~1:} \ $a=0$. Since $(a,b)\neq (0,0)$, we have $b\neq 0$. Then $F=x^4+y^4+bxyxy^\rcirclearrowleft$. Scaling $x$ and $y$, we can turn $F$ into $F=x^4-\frac12xyxy^\rcirclearrowleft+qy^4$ with $q\in\K^*$. The defining relations of $A$ now are $x^3=yxy$ and $xyx=qy^3$. It is now straightforward to compute the reduced Gr\"obner basis of the ideal of relations of $A$, which comprises $x^3-yxy$, $xyx-qy^3$, $xy^4-y^4x$, $x^2y^3-\frac1q yxy^2x$ and $xy^2xy-qy^3x^2$. Knowing the leading monomials of this basis it is routine to verify that $H_A=(1+t)^{-1}(1-t)^{-3}$. Since none of the leading monomials of the members of the Gr\"obner basis starts with $y$, there are no non-trivial right annihilators in $A$. By Lemma~\ref{ms1}, $A$ is exact.

\medskip

{\bf Case~2:} \ $b=0$. Since $(a,b)\neq (0,0)$, we have $a\neq 0$. Then $F=x^4+y^4+a{x^2y^2}^\rcirclearrowleft$. Scaling $x$ and $y$, we can turn $F$ into $F=x^4+{x^2y^2}^\rcirclearrowleft+qy^4$ with $q\in\K^*$. The defining relations of $A$ now are $x^3+xy^2+y^2x$ and $x^2y+yx^2+qy^3$. First, computing the Gr\"obner basis up to degree $5$, it is easy to verify that
$$
g=xy^2x+(1-q)y^2x^2-qy^4
$$
commutes with both $x$ and $y$ and therefore is central in $A$. Consider the algebra $B=A/I$, where $I$ is the ideal generated by $g$. The algebra $B$ can be presented by the generators $x,y$ and the relations $x^3+xy^2+y^2x=0$, $x^2y+yx^2+qy^3$ and $xy^2x+(1-q)y^2x^2-qy^4$. It is now straightforward to compute the reduced Gr\"obner basis of the ideal of relations of $B$, which comprises the defining relations together with $xyx^2-(1-q)xy^3-y^2xy$, $xy^4+y^4x+(2-q)y^2xy^2$, $xyxy^2+(2-q)xy^3x+y^2xyx$, $xy^3x^2+(q^2-3q+1)y^2xy^3+(1-q)y^4xy$ and $xy^3xy^2-(q^2-4q+3)y^2xy^3x-(2-q)y^4xyx$. Knowing the leading monomials of this basis it is routine to verify that $H_B=1+\sum\limits_{j=1}^\infty 2jt^j=\frac{1+t^2}{(1-t)^2}$. Since none of the leading monomials of the members of the Gr\"obner basis starts with $y$, we have $yu\neq 0$ for every non-zero $u\in B$. Since $g$ is central, we have $\dim A_n=\dim B_n+\dim gA_{n-4}$ for every $n\geq 4$. In particular, $\dim A_n\leq \dim A_{n-4}+2n$ for $n\geq 5$ and all these inequalities turn into equalities precisely when $g$ is not a zero divisor. The inequalities $\dim A_n\leq \dim A_{n-4}+2n$ together with easily verifiable $\dim A_4=12$ imply that $H_A\leq (1+t)^{-1}(1-t)^{-3}$ and the equality is only possible if $\dim A_n\leq \dim A_{n-4}+2n$ for all $n\geq 5$. By Lemma~\ref{ms1}, $H_A\geq (1+t)^{-1}(1-t)^{-3}$. Hence $H_A=(1+t)^{-1}(1-t)^{-3}$ and $g$ is not a zero divisor. Now we check that $yu\neq 0$ for every non-zero $u\in A$. Assume the contrary. Then pick a non-zero homogeneous $u\in A$ of smallest possible degree such that $yu=0$ in $A$. Since $B$ is a quotient of $A$, $yu=0$ in $B$. Hence $u=0$ in $B$. Then $u=gv$ in $A$ for some $v\in A$. The equality $yu=0$ yields $gyv=0$ and therefore $yv=0$ in $A$. Since the degree of $v$ is smaller (by $4$) than the degree of $u$, we have arrived to a contradiction. Thus $yu\neq 0$ for every non-zero $u\in A$ and therefore $A$ has no non-trivial right annihilators. By Lemma~\ref{ms1}, $A$ is exact.

\medskip

{\bf Case~3:} \ $a=1$. Since $(a,b)\neq (1,1/2)$, we have $b\neq 1/2$. Since $4(a+b)^2\neq 1$, we have $b\neq -1/2$. Denoting $b=\frac{q}2$, we have $F=x^4+{x^2y^2}^\rcirclearrowleft+\frac{q}2xyxy^\rcirclearrowleft+y^4$ with $q^2\neq 1$. The defining relations of $A$ now are $x^3+xy^2+qyxy+y^2x$ and $x^2y+qxyx+yx^2+y^3$. First, computing the Gr\"obner basis up to degree $5$, it is easy to verify that $g=xy^2x-y^4$ commutes with both $x$ and $y$ and therefore is central in $A$. Consider the algebra $B=A/I$, where $I$ is the ideal generated by $g$. The algebra $B$ can be presented by the generators $x,y$ and the relations $x^3+xy^2+qyxy+y^2x$, $x^2y+qxyx+yx^2+y^3$ and $xy^2x-y^4$. It is now straightforward to compute the reduced Gr\"obner basis of the ideal of relations of $B$, which comprises the defining relations together with $xyx^2-y^2xy$, $xyxyx+\frac1q xy^4+\frac1qy^2xy^2+\frac1qy^4x$, $xyxy^2+xy^3x+y^2xyx+qy^5$, $xy^3x^2+xy^5+y^4xy+qy^5x$, $xy^3xy+\frac1q xy^4x+\frac1qy^4x^2+\frac1qy^6$ and $xy^6-y^6x$. Knowing the leading monomials of this basis it is routine to verify that $H_B=\frac{1+t^2}{(1-t)^2}$. The rest of the proof is the same as in Case~2.

\medskip

{\bf Case~4:} \ $a=-1$ or $a=-2b$ or $b=\pm \frac12$. These cases follow from the already considered ones due to the isomorphism conditions in (H10). Indeed, one easily sees that our algebras in the case $b=\pm \frac12$ are isomorphic to those with $a=0$. The cases $a=1$, $a=-1$ and $a=-2b$ are linked in a similar way.

\medskip

{\bf Case~5:} \ $2b=a$. Since $(a,b)\neq (0,0)$, we have $a\neq 0$ and $b\neq 0$. Scaling $x$ and $y$, we can turn $F$ into $F=x^4+{x^2y^2}^\rcirclearrowleft+\frac12xyxy^\rcirclearrowleft+qy^4$ with $q=a^{-2}\in\K^*$. Since $(a,b)\neq \pm(1,1/2)$, we have $q\neq 1$. The defining relations of $A$ now are $x^3+xy^2-yxy+y^2x$ and $x^2y-xyx+yx^2+qy^3$. First, computing the Gr\"obner basis up to degree $5$, it is easy to verify that $g=xyxy-y^2x^2$ commutes with both $x$ and $y$ and therefore is central in $A$. Consider the algebra $B=A/I$, where $I$ is the ideal generated by $g$. The algebra $B$ can be presented by the generators $x,y$ and the relations $x^3+xy^2-yxy+y^2x$, $x^2y-xyx+yx^2+qy^3$ and $xyxy-y^2x^2$. It is now straightforward to compute the reduced Gr\"obner basis of the ideal of relations of $B$, which comprises the defining relations together with $xy^3-y^3x$. Knowing the leading monomials of this basis it is routine to verify that $H_B=\frac{1+t^2}{(1-t)^2}$. The rest of the proof is the same as in Case~2.

\medskip

{\bf Case~6:} \ $2b=a^2$. Since $(a,b)\neq (0,0)$, we have $a\neq 0$ and $b\neq 0$. Scaling $x$ and $y$, we can turn $F$ into $F=x^4+{x^2y^2}^\rcirclearrowleft+\frac{a}2xyxy^\rcirclearrowleft+\frac1{a^2}y^4$. Since the case $b=\pm\frac12$ is already considered, we can assume that $a^2\neq 1$. The defining relations of $A$ now are $x^3+xy^2+ayxy+y^2x$ and $x^2y+axyx+yx^2+\frac1{a^2}y^3$. Computing the Gr\"obner basis up to degree $5$, it is easy to verify that $\textstyle g=xyxy+axy^2x-\frac1a y^2x^2-\frac1a y^4$ commutes with both $x$ and $y$ and therefore is central in $A$. Consider the algebra $B=A/I$, where $I$ is the ideal generated by $g$. The algebra $B$ can be presented by the generators $x,y$ and the relations $x^3+xy^2+ayxy+y^2x$, $x^2y+axyx+yx^2+\frac1{a^2}y^3$ and $xyxy+axy^2x-\frac1a y^2x^2-\frac1a y^4$. It is now straightforward to compute the reduced Gr\"obner basis of the ideal of relations of $B$, which comprises the defining relations together with $xyx^2+\frac1{a^2}xy^3-y^2xy-\frac1a y^3x$, $xy^2x^2+\frac1a y^3xy$, $xy^3x+\frac1a y^5$, $xy^2xyx+\frac1{a^3} xy^5-\frac1{a^2}y^3xy^2-\frac1{a^2}y^5x$, $xy^2xy^2+xy^4x-\frac1{a}y^3xyx-y^6$, $xy^4xy+\frac1{a}xy^5x-\frac1{a^2}y^5x^2-\frac1{a^2}y^7$, $xy^4x^2+\frac1{a^2}xy^6-\frac1{a}y^5xy-y^6x$, $xy^5xy+\frac1{a}xy^6x+\frac1{a}y^6x^2+\frac1{a^3}y^8$, $xy^5x^2+xy^7+y^6xy+ay^7x$, $xy^6xy-y^7x$ and $xy^8-y^8x$. Knowing the leading monomials of this basis it is routine to verify that $H_B=\frac{1+t^2}{(1-t)^2}$. The rest of the proof is the same as in Case~2.

\medskip

{\bf Case~7:} \ $4b^2-2a^2+1=0$ or $2b=-a^2$. As in Case~4, these cases follow from the already considered ones due to the isomorphism conditions in (H10). Indeed, one easily sees that our algebras in the case $2b=-a^2$ are isomorphic to those with $2b=a^2$ as well as to those with $4b^2-2a^2+1=0$. It remains to notice that Cases~1--7 exhaust all possibilities.
\end{proof}

\begin{lemma}\label{nwei} Let $F\in{\cal P}_{2,4}$ be given by {\rm (H10)} of Lemma~$\ref{nco2-4}$ with parameters $\alpha,\beta$ $($we want to reserve letters $a$ and $b)$ such that
$$
\alpha\beta(\alpha^2-1)(4\beta^2-\alpha^2)(4\beta^2-1)(4\beta^2-\alpha^4)(4\beta^2-2\alpha^2+1)\neq 0.
$$
Then $A=A_F$ is exact and therefore $H_A=(1+t)^{-1}(1-t)^{-3}$.
\end{lemma}

\begin{proof} A scaling turns $F$ into
$$
\textstyle F=x^4+{x^2y^2}^\rcirclearrowleft+\frac{a}{2}xyxy^\rcirclearrowleft+by^4
$$
with $a,b\in\K$ given by $a=\frac{2\beta}{\alpha}$ and $b=\frac1{\alpha^2}$. In terms of $a$ and $b$ the assumption about $\alpha$ and $\beta$ reads as follows: $a\neq 0$, $b\neq 1$, $a^2\neq 1$, $b+a^2\neq 2$, $a^2b\neq 1$ and $b\neq a^2$.

Computing $\delta_xF$ and $\delta_yF$, we see that $A$ is given by generators $x$ and $y$ and relations
\begin{equation}\label{dr1}
x^3=-xy^2-ayxy-y^2x,\qquad x^2y=-axyx-yx^2-by^3.
\end{equation}
A direct computation allows to find all elements of the reduced Gr\"obner basis of the ideal of relations up to degree $5$. They correspond to the relations
\begin{align*}
\textstyle xyx^2&\textstyle=\frac{1-b}{1-a^2}xy^3+y^2xy-\frac{a(1-b)}{1-a^2}y^3x;
\\
\textstyle xyxyx&\textstyle=-\frac1a xy^2x^2-\frac{1-a^2b}{a(1-a^2)}xy^4-\frac1ay^2xy^2+\frac{1-b}{1-a^2}y^3xy;
\\
\textstyle xy^2xy&\textstyle=-\frac{a(1-b)(2-b-a^2)}{(1-a^2)(1-a^2b)}xy^3x-\frac{a(1-b)}{1-a^2b}yxyxy+\frac{1-a^2}{1-a^2b}yxy^2x+\frac{(1-b)(2-a^2b-a^2)}{(1-a^2)(1-a^2b)}y^3x^2+\frac{b(1-b)}{1-a^2b}y^5;
\\
\textstyle xyxy^2&\textstyle=-\frac{2-b-a^2}{1-a^2b}xy^3x+\frac{a^2(1-b)}{1-a^2b}yxyxy-\frac{a(1-a^2)}{1-a^2b}yxy^2x-y^2xyx-\frac{a(1-b)}{1-a^2b}y^3x^2-\frac{ab(1-b)}{1-a^2b}y^5.
\end{align*}

This provides us with a multiplication table in $A$ for degrees up to $5$. Given this, it is routine to verify that
$$
g=-a(1-b)xyxy+(1-a^2)xy^2x+(1-b)y^2x^2-b(1-a^2)y^4
$$
commutes with both $x$ and $y$ and therefore is central in $A$. Now we consider the algebra
$$
\text{$B=A/I$, where $I$ is the ideal in $A$, generated by $g$}
$$
as well as the degree-graded right $B$-module
$$
M=B/yB.
$$
Note that using the above Gr\"obner basis elements for $A$, one easily sees that Hilbert series of $M$ starts as $H_M=1+2t+2t^2+2t^3+2t^4+2t^5+{\dots}$ By the same token,
\begin{equation}\label{f5}
H_A=1+2t+4t^2+6t^3+9t^4+12t^5+{\dots}
\end{equation}
According to Lemma~\ref{ms1},
\begin{equation}\label{fff}
H_A\geq (1+t)^{-1}(1-t)^{-3}.
\end{equation}
By the same lemma, the proof will be complete if we show that
\begin{equation}\label{aa1}
\text{$H_A=(1+t)^{-1}(1-t)^{-3}$ and  $A$ has no non-trivial right annihilators}.
\end{equation}

We start by proving the following two statements:
\begin{align}
&H_M(t)=1+\sum\limits_{n=1}^\infty 2t^n\,\,\Longrightarrow\,\,\text{(\ref{aa1}) is satisfied},\label{mm1}
\\
&\text{if $k\in\N$ and $\dim M_j\leq 2$ for $1\leq j\leq k$, then $\dim M_j=2$ for $1\leq j\leq k$.}\label{mm2}
\end{align}

Assume that $k\in\N$ and $\dim M_j\leq 2$ for $1\leq j\leq k$. Clearly,
$\dim B_j=\dim yB_{j-1}+\dim M_j$ for $j\in\N$. It follows that $\dim B_j\leq 2j$ for $1\leq j\leq k$ and the inequalities turn into equalities if and only if $\dim M_j=2$ for $1\leq j\leq k$ and $yu\neq 0$ for every degree $<k$ homogeneous $u\in B$. Next, $\dim A_{j}=\dim gA_{j-4}+\dim B_j$ for all $j\geq 4$. Using this recurrent inequality and the initial data (\ref{f5}), we see that for $j\leq k$, $\dim A_j$ does not exceed the $j^{\rm th}$ coefficient of $(1+t)^{-1}(1-t)^{-3}$ and the inequalities turn into equalities if and only if $\dim B_j=2j$ for $1\leq j\leq k$ and $gu\neq 0$ for every degree $\leq k-4$ homogeneous $u\in A$. However, by (\ref{fff}), turn into equalities they must. In particular, we must have $\dim M_j=2$ for $1\leq j\leq k$, which proves (\ref{mm2}). In order to prove (\ref{mm1}), we apply the above argument with arbitrarily large $k$. It follows that the equality $H_M(t)=1+\sum\limits_{n=1}^\infty 2t^n$ not only yields $H_A=(1+t)^{-1}(1-t)^{-3}$, but also ensures that $yu\neq 0$ for every non-zero $u\in B$ and $gu\neq 0$ for every non-zero $u\in A$. In order to complete the proof, it suffices to show that $yu\neq 0$ for every non-zero $u\in A$. Assume the contrary. Then pick a non-zero homogeneous $u\in A$ of smallest possible degree such that $yu=0$ in $A$. Since $B$ is a quotient of $A$, $yu=0$ in $B$. Hence $u=0$ in $B$. Then $u=gv$ in $A$ for some $v\in A$. The equality $yu=0$ yields $gyv=0$ and therefore $yv=0$ in $A$. Since the degree of $v$ is smaller (by $4$) than the degree of $u$, we have arrived to a contradiction. This concludes the proof of (\ref{mm1}).

According to (\ref{mm1}), the proof will be complete if we verify that $H_M(t)=1+\sum\limits_{n=1}^\infty 2t^n$. By definition of $B$ and the above formulas for the low degree elements of the Gr\"obner basis for $A$, we see that the following relations are satisfied in $B$:
\begin{align}
x^3&=-xy^2-ayxy-y^2x,\label{rb1}
\\
x^2y&=-axyx-yx^2-by^3,\label{rb2}
\\
xyx^2&\textstyle=\frac{1-b}{1-a^2}xy^3+y^2xy-\frac{a(1-b)}{1-a^2}y^3x;\label{rb3}
\\
xyxy&\textstyle=\frac{1-a^2}{a(1-b)}xy^2x+\frac1a y^2x^2-\frac{b(1-a^2)}{a(1-b)}y^4.\label{rb4}
\end{align}
Actually, the first two and the last of these are the defining relations, while (\ref{rb3}) is the only other member of the degree $\leq 4$ of the Gr\"obner basis.

For each $k\in\Z_+$, consider the following property:
\begin{itemize}
\item[$(\Omega_k)$]$\dim M_j=2$ for $1\leq j\leq k+3$, $M_{k+3}$ is spanned by $xy^{k+2}$ and $xy^{k+1}x$ and there exist $a_k,b_k\in\K$ such that the equalities $xy^kx^2=a_kxy^{k+2}$ and $xy^kxy=b_kxy^{k+1}x$ hold in $M$.
\end{itemize}

Note that if $(\Omega_k)$ is satisfied, then $a_k$ and $b_k$ are uniquely determined. Indeed, otherwise $xy^{k+2}$ and $xy^{k+1}x$ would be linearly dependent in $M$. Note also that according to (\ref{rb1}--\ref{rb4}),
\begin{equation}\label{start}
\text{$\Omega_0$ and $\Omega_1$ are satisfied with $a_0=-1$, $b_0=-a$, $\textstyle a_1=\frac{1-b}{1-a^2}$ and $\textstyle b_1=\frac{1-a^2}{a(1-b)}$.}
\end{equation}
Note also that
\begin{equation}\label{oomm}
\begin{array}{l}
\text{if $k\in\Z_+$, $\dim M_j=2$ for $1\leq j\leq k+1$, $M_{k+2}$ is spanned by $xy^{k+1}$ and $xy^kx$ and}\\
\text{$xy^kx^2=a_kxy^{k+2}$, $xy^kxy=b_kxy^{k+1}x$ in $M$ for some $a_k,b_k\in\K$, then $(\Omega_k)$ holds.}
\end{array}
\end{equation}
Indeed, by (\ref{mm2}), $\dim M_{k+2}=2$. Since $M_{k+2}$ is spanned by $xy^{k+1}$ and $xy^kx$, $M_{k+3}$ is spanned by  $xy^{k+2}$, $xy^kxy$, $xy^{k+1}x$ and $xy^kx^2$. By the equations in (\ref{oomm}), $M_{k+3}$ is spanned by  $xy^{k+2}$, and $xy^{k+1}x$ and $\dim M_{k+3}=2$ by (\ref{mm2}). Thus $(\Omega_k)$ holds.

Reducing the overlaps $xy^kx^2y=(xy^kx^2)y=xy^k(x^2y)$ and $xy^kx^3=(xy^kx^2)x=xy^k(x^3)$ by means of (\ref{rb2}), (\ref{rb1}) and the equations from ($\Omega_k$), we obtain
\begin{equation}\label{omm1}
\begin{array}{l}
\text{if $k\in\Z_+$ and $\Omega_k$ is satisfied, then}\\ (ab_k+1)xy^{k+1}x^2+(a_k+b)xy^{k+3}=
(b_k+a)xy^{k+1}xy+(a_k+1)xy^{k+2}x=0\ \ \text{in $M.$}
\end{array}
\end{equation}
Reducing $xy^{k-1}xyxy=(xy^{k-1}xy)xy=xy^{k-1}(xyxy)$ and $xy^{k-1}xyx^2=(xy^{k-1}xy)x^2=xy^{k-1}(xyx^2)$ by means of (\ref{rb1}--\ref{rb4}) and the equations from ($\Omega_k$) and ($\Omega_{k-1}$), we get
\begin{equation}\label{omm2}
\begin{array}{l}
\text{if $k\in\N$ and both $\Omega_{k-1}$ and $\Omega_k$ are satisfied, then}
\\
\left(\frac1a+b_{k-1}+\frac{1-a^2b}{a(1-b)}b_{k-1}b_k\right)xy^{k+1}x^2+\left(bb_{k-1}-\frac{b(1-a^2)}{a(1-b)}\right)xy^{k+3}=0,
\\
\left(1+ab_{k-1}+\frac{2-b-a^2}{1-a^2}b_{k-1}b_k\right)xy^{k+1}xy+\left(b_{k-1}-\frac{a(1-b)}{1-a^2}\right)xy^{k+2}x=0\ \ \text{in $M.$}
\end{array}
\end{equation}

Assume now that $k\in\N$ and both $\Omega_{k-1}$ and $\Omega_k$ are satisfied. We consider the following three options:
\begin{itemize}\itemsep=-2pt
\item[(O1)] $\bigl(ab_k+1,\frac1a+b_{k-1}+\frac{1-a^2b}{a(1-b)}b_{k-1}b_k\bigr)=(0,0);$
\item[(O2)] $\bigl(b_k+a,1+ab_{k-1}+\frac{2-b-a^2}{1-a^2}b_{k-1}b_k\bigr)=(0,0);$
\item[(O3)] $\bigl(ab_k+1,\frac1a+b_{k-1}+\frac{1-a^2b}{a(1-b)}b_{k-1}b_k\bigr)\neq(0,0)$ and
$\bigl(b_k+a,1+ab_{k-1}+\frac{2-b-a^2}{1-a^2}b_{k-1}b_k\bigr)\neq (0,0)$,
\end{itemize}
which cover all possibilities.

First observe that according to (\ref{omm1}), (\ref{omm2}) and (\ref{oomm}) ,
\begin{equation}\label{oo3}
\text{if (O3) holds, then $(\Omega_{k+1})$ is satisfied.}
\end{equation}

Assume now that (O2) holds. By the equalities in (O2), $b_k=-a$ and $b_{k-1}=\frac{1-a^2}{a(1-b)}$. The conditions $a^2b\neq 1$ and $a^2+b\neq 2$ allow to check that $b_{k-1}\neq -a$ and $b_{k-1}\neq -\frac1a$. Then (\ref{omm1}) applied with $k-1$ instead of $k$ yields
\begin{equation}\label{kk1}
\textstyle b_k=-\frac{a_{k-1}+1}{b_{k-1}+a}\ \ \text{and}\ \ a_k=-\frac{a_{k-1}+b}{ab_{k-1}-1}.
\end{equation}
Plugging the above expressions for $b_{k-1}$ and $b_k$ into the first equation in (\ref{kk1}), we get $a_{k-1}=\frac{b(1-a^2)}{1-b}$. Plugging this together with  $b_{k-1}=\frac{1-a^2}{a(1-b)}$ into the second equation in (\ref{kk1}), we get (after cancellations to perform which we need the assumptions about $a$ and $b$) $a_k=-b$. Now plugging $b_k=-a$, $a_k=-b$, $a_{k-1}=\frac{b(1-a^2)}{1-b}$ and  $b_{k-1}=\frac{1-a^2}{a(1-b)}$ into the equalities in (\ref{omm1}) and (\ref{omm2}), we see that
$$
xy^{k+1}x^2=xy^{k+2}x=0\ \ \text{in $M$}.
$$
Then $M_{k+4}$ is spanned by $xy^{k+1}xy$ and $xy^{k+3}$. Using the above display, (\ref{rb2}) and (\ref{rb1}) we reduce the overlaps $xy^{k+1}x^3=(xy^{k+1}x^2)x=xy^{k+1}(x^3)$ and $xy^{k+1}x^2y=(xy^{k+1}x^2)y=xy^{k+1}(x^2y)$ to get
$$
\textstyle xy^{k+1}xy^2+xy^{k+3}x=xy^{k+1}xyx+\frac baxy^{k+4}=0\ \ \text{in $M$}.
$$
Then $M_{k+5}$ is spanned by $xy^{k+3}x$ and $xy^{k+4}$. Using the equalities in the above display together with (\ref{rb3}) and (\ref{rb4}), we reduce the overlaps $xy^{k+1}xyx^2=(xy^{k+1}xyx)x=xy^{k+1}(xyx^2)$ and $xy^{k+1}xyxy=(xy^{k+1}xyx)y=xy^{k+1}(xyxy)$ to get
$$
\textstyle xy^{k+3}x^2+bxy^{k+5}=xy^{k+3}xy+\frac 1axy^{k+4}x=0\ \ \text{in $M$}.
$$
By (\ref{oomm}), we see that $\Omega_{k+3}$ is satisfied  with $a_{k+3}=-b$ and $b_{k+3}=-\frac1a$. Dealing in a similar way with the overlaps $xy^{k+3}x^3=(xy^{k+3}x^2)x=xy^{k+3}(x^3)$ and $xy^{k+2}xyxy=(xy^{k+2}xyx)y=xy^{k+2}(xyxy)$,
we get
$$
\textstyle  xy^{k+4}xy=\frac{a(1-b)}{1-a^2}xy^{k+5}x\ \ \text{and}\ \ xy^{k+4}x^2=\frac{b(1-a^2)}{1-b}xy^{k+6}\ \ \text{in $M$},
$$
The last display together with (\ref{oomm}) shows that $\Omega_{k+4}$ is satisfied. Hence
\begin{equation}\label{oo2}
\text{if (O2) holds, then $(\Omega_{k+3})$ and $(\Omega_{k+4})$ are satisfied.}
\end{equation}

Finally, assume that (O1) holds. By the equalities in (O1), $b_k=-\frac1a$ and $b_{k-1}=\frac{a(1-b)}{1-a^2}$. The conditions $a^2b\neq 1$ and $a^2+b\neq 2$ yield $b_{k-1}\neq -a$ and $b_{k-1}\neq -\frac1a$. As above, this means that (\ref{kk1}) holds. Plugging the expressions for $b_{k-1}$ and $b_k$ into the first equation in (\ref{kk1}), we get $a_{k-1}=\frac{1-b}{1-a^2}$. Plugging this together with  $b_{k-1}=\frac{a(1-b)}{1-a^2}$ into the second equation in (\ref{kk1}), we obtain $a_k=-1$.
Plugging $a_{k-1}=\frac{1-b}{1-a^2}$, $b_{k-1}=\frac{a(1-b)}{1-a^2}$, $b_k=-\frac1a$ and $a_k=-1$ into the equalities from (\ref{omm1}) and (\ref{omm2}), we get
$$
xy^{k+1}xy=xy^{k+3}=0\ \ \text{in $M$}
$$
 Now $M_{k+4}$ is spanned by $xy^{k+1}x^2$ and $xy^{k+2}x$. From this and (\ref{rb1}) and (\ref{rb2}) it follows that $M_{k+5}$ is spanned by $xy^{k+2}x^2$ and $xy^{k+2}xy$. Now an elementary inductive procedure (use (\ref{mm2})) shows that $M_j$ is 2-dimensional for every $j$. That is,
\begin{equation}\label{oo1}
\text{if (O1) holds, then $M_j$ is 2-dimensional for every $j\in\N$.}
\end{equation}

Note that if $(\Omega_k)$ holds for infinitely many $k$, then $M_j$ is 2-dimensional for every $j\in\N$ as well. Applying (\ref{oo2}) and (\ref{oo3}) inductively ((\ref{start}) serves as the basis of induction) and using (\ref{oo1}), we see that no matter the case, $M_j$ is 2-dimensional for every $j\in\N$. This completes the proof.
\end{proof}

\subsection{Proof of Theorem~\ref{main24-p}}

Combining Lemmas~\ref{nco2-4}, \ref{easy1}, \ref{wei} and \ref{nwei}, we see that all statements of Theorem~\ref{main24-p} hold with isomorphism of $A_F$ and $A_G$ condition replaced by equivalence of $F$ and $G$ (with respect to the $GL_2(\K)$ action by linear substitutions).

By Lemma~\ref{gen2}, these two equivalences are the same for proper potentials. Thus all that remains is to show that algebras from (P24--P28) are pairwise non-isomorphic. Since isomorphic graded algebras have the same Hilbert series, it remains to verify that three algebras from (P24--P26) are pairwise non-isomorphic. Now (P25) is singled out by being non-monomial (it is easy to see that it is not isomorphic as a graded algebra to a monomial one), while algebras from (P24) and (P26) are monomial. Algebras in (P24) and (P26) are non-isomorphic since the first one has cubes in the space of degree $3$ relations, while the second one has no such thing.

\section{Potential algebras $A_F$ for $F\in{\cal P}_{3,3}$}

Throughout this section we equip the monomials in $x,y,z$ with the left-to-right degree-lexicographical ordering assuming $x>y>z$. The following statement is elementary.

\begin{lemma}\label{cyab33}
The kernel of the canonical homomorphism from $\K\langle x,y,z\rangle$ onto $\K[x,y,z]$ $(=$abelianization$)$ intersects ${\cal P}_{3,3}$ by the one-dimensional space spanned by $xyz^\rcirclearrowleft-xzy^\rcirclearrowleft$.
\end{lemma}

\begin{lemma}\label{nco3-3} Let $F\in{\cal P}_{3,3}$. Then by means of a linear substitution $F$ can be turned into one of the following forms$:$

\medskip
\noindent $\begin{array}{ll}\text{{\rm \!(G1)}\ \ $F=0;$\ \ }&\text{{\rm \phantom0(G9)}\ \ $F=z^3+xyz^\rcirclearrowleft;$}\\
\text{{\rm \!(G2)}\ \ $F=z^3;$}&\text{{\rm (G10)}\ \ $F=(y+z)^3+xyz^\rcirclearrowleft;$}\\
\text{{\rm \!(G3)}\ \ $F={yz^2}^\rcirclearrowleft;$}&\text{{\rm (G11)}\ \ $F={yz^2}^\rcirclearrowleft+xyz^\rcirclearrowleft-xzy^\rcirclearrowleft;$}\\
\text{{\rm \!(G4)}\ \ $F=y^3+z^3;$}&\text{{\rm (G12)}\ \ $F=y^3+z^3+xyz^\rcirclearrowleft-xzy^\rcirclearrowleft;$}\\
\text{{\rm \!(G5)}\ \ $F=xyz^\rcirclearrowleft;$}&\text{{\rm (G13)}\ \ $F=y^3+{xz^2}^\rcirclearrowleft+xyz^\rcirclearrowleft-xzy^\rcirclearrowleft;$}\\
\text{{\rm \!(G6)}\ \ $F=x^3+y^3+z^3;$}&\text{{\rm (G14)}\ \ $F={xz^2}^\rcirclearrowleft+{y^2z}^\rcirclearrowleft+xyz^\rcirclearrowleft-xzy^\rcirclearrowleft;$}\\
\text{{\rm \!(G7)}\ \ $F={xz^2}^\rcirclearrowleft+y^3;$}&\text{{\rm (G15)}\ \ $F_a=xyz^\rcirclearrowleft-axzy^\rcirclearrowleft$ with $a\in\K^*;$}\\
\text{{\rm \!(G8)}\ \ $F={xz^2}^\rcirclearrowleft+{y^2z}^\rcirclearrowleft;$\quad}&\text{{\rm (G16)}\ \ $F_a=z^3+xyz^\rcirclearrowleft+axzy^\rcirclearrowleft$ with $a\in\K^*;$}\\
\text{\rlap{{\rm \!\!\!\!(G17)}\ \ $F_a=(y+z)^3+xyz^\rcirclearrowleft+axzy^\rcirclearrowleft$ with $a\notin\{0,-1\};$}}&\\
\text{\rlap{{\rm \!\!\!\!(G18)}\ \ $F_{a,b}{=}x^3{+}y^3{+}z^3{+}axyz^\rcirclearrowleft\!{+}bxzy^\rcirclearrowleft\!$ with $a,b\in\K$, $(a{+}b)^3{+}1{\neq}0$, $(a,b){\neq}(0,0)$ and $(a^3{-}1,b^3{-}1){\neq} (0,0).$}}&
\end{array}$
\medskip

Moreover, to which of the above $18$ forms $F$ can be turned into is uniquely determined by $F$. For $F=F_a$ from {\rm (G15--G17)} $F_a$ can be obtained from $F_b$ by a linear substitution if and only if $a=b$ or $ab=1$. Finally, for $F=F_{a,b}$ from {\rm (G17)}, $F_{a,b}$ and $F_{a',b'}$ can be obtained from one another by a linear substitution if and only if they belong to the same orbit of the group action generated by two maps $(a,b)\mapsto (\theta a,\theta b)$ and $(a,b)\mapsto \bigl(\frac{1+\theta a+\theta^2b}{1+a+b},\frac{1+\theta^2a+\theta b}{1+a+b}\bigr)$. This group has $24$ elements and is isomorphic to $SL_2(\Z_3)$.
\end{lemma}

\begin{proof} Let $F\in{\cal P}_{3,3}$. First, we show that $F$ can be turned into exactly one of (G1--G18) by a linear sub. Let $G\in\K[x,y,z]$ be the abelianization of $F$. By Lemma~\ref{co3-3} by means of a linear substitution $G$ can be turned into exactly one of the forms (Z1--Z8). Thus we can assume from the start that $G$ is in one of the forms (Z1--Z8).

Case 1: $G=a(x^3+y^3+z^3)+bxyz$ with $a,b\in\K$. By Lemma~\ref{cyab33}, $F=r(x^3+y^3+z^3)+qxyz^\rcirclearrowleft+qxzy^\rcirclearrowleft$ with $p,q,r\in\K$. Note that the substitution $x\to x$, $y\to y$, $z\to\theta z$ preserves this shape of $F$ and transforms the parameters according to the rule $(p,q,r)\mapsto (\theta p,\theta q,r)$, while the sub $x\to x+y+z$, $y\to x+\theta^2y+\theta z$, $z\to x+\theta y+\theta^2 z$ also preserves the shape of $F$ and transforms the parameters according to the rule $(p,q,r)\mapsto (\theta p+\theta^2q+r,\theta^2p+\theta q+r,p+q+r)$. If $p=q=r=0$, we fall into (G1). Using the above subs and a scaling, we see that $F$ can be turned into the form (G5) if either $q=r=0$, $p\neq 0$ or $p=r=0$, $q\neq 0$ or $p^3=q^3=r^3\neq 0$, $p\neq q$ and $F$ can be turned into the form (G6) if either $p=q=0$, $r\neq 0$ or $p=q\neq 0$,  $p^3=r^3$. As shown in \cite{SKL2}, $F$ can be turned into the form (G15) precisely when either $r=0$ and $pq\neq 0$ or $(p+q)^3=-r^3\neq 0$. Now assume that none of the above assumptions on $p,q,r$ holds. Then $r\neq 0$. By a scaling, we can turn $r$ into $1$. The rest of the assumptions now read $(p+q)^3+1\neq 0$, $(p,q)\neq (0,0)$ and $(p^3-1,q^3-1)\neq (0,0)$. That is, we end up in (G18). The fact that $F$ from (G1), (G5), (G6), (G14) and (G18) with different labels are non-equivalent follows from the fact \cite{SKL2} that even the corresponding potential algebras are non-isomorphic.

Case 2: $G=z^3$. By Lemma~\ref{cyab33},  $F=z^3+s(xyz^\rcirclearrowleft-xzy^\rcirclearrowleft)$ with $s\in\K$. If $s=0$, $F$ is given by (G2). If $s\neq 0$, a scaling brings $F$ to the form (G16) with $a=-1$.

Case 3: $G=yz^2$. By Lemma~\ref{cyab33},  $F=\frac13{yz^2}^\rcirclearrowleft+s(xyz^\rcirclearrowleft-xzy^\rcirclearrowleft)$ with $s\in\K$. If $s=0$, $F$ acquires form (G3) after scaling. If $s\neq 0$, a scaling brings $F$ to the form (G11).

Case 4: $G=y^3+z^3$. By Lemma~\ref{cyab33},  $F=y^3+z^3+s(xyz^\rcirclearrowleft-xzy^\rcirclearrowleft)$ with $s\in\K$. If $s=0$, $F$ is given by  (G4). If $s\neq 0$, a scaling  brings $F$ to the form (G12).

Case 5: $G=xz^2+y^3$. By Lemma~\ref{cyab33},  $F=y^3+\frac13{xz^2}^\rcirclearrowleft+s(xyz^\rcirclearrowleft-xzy^\rcirclearrowleft)$ with $s\in\K$. If $s=0$, $F$ acquires form (G7) after scaling. If $s\neq 0$, a scaling brings $F$ to the form (G13).

Case 6: $G=xz^2+y^2z$. By Lemma~\ref{cyab33},  $F=\frac13{xz^2}^\rcirclearrowleft+\frac13{y^2z}^\rcirclearrowleft+s(xyz^\rcirclearrowleft-xzy^\rcirclearrowleft)$ with $s\in\K$. If $s=0$, $F$ acquires form (G8) after scaling. If $s\neq 0$, a scaling brings $F$ to the form (G14).

Case 7:  $G=xyz+z^3$. By Lemma~\ref{cyab33},  $F=z^3+\bigl(s+\frac{1}{3}\bigr)xyz^\rcirclearrowleft-sxzy^\rcirclearrowleft$ with $s\in\K$. If $s=0$, a scaling brings $F$ to the form (G9). If $s=-\frac13$, swapping $x$ and $y$ and a scaling brings $F$ to the form (G9) again. If $s\neq 0$ and $s\neq -\frac13$, a scaling turns $F$ into the form (G16) with $a\neq -1$.

Case 8: $G=xyz+(y+z)^3$. By Lemma~\ref{cyab33},  $F=(y+z)^3+\bigl(s+\frac{1}{3}\bigr)xyz^\rcirclearrowleft-sxzy^\rcirclearrowleft$ with $s\in\K$. Same as in the previous case, if $s=0$ or $s=-\frac13$ a scaling or the same together with swapping of $y$ and $z$ turns $F$ into the form (G10). If $s\neq 0$ and $s\neq -\frac13$, a scaling turns $F$ into the form (G17) (automatically, $a\neq -1$).

The fact that $F$ from the list (G1--G18) with different labels are non-equivalent (can not be obtained from one another by a linear sub) follows from the non-equivalence of polynomials with different labels from the list (Z1--Z8), the equivalence statements in Case~1  as well as the trivial observation that a symmetric (not just cyclicly) element of ${\cal P}_{3,3}$ can not be equivalent to a non-symmetric one.

It remains to prove the statements about equivalence within each of (G15--G18). The (G15) and (G18) cases are done in \cite{SKL2}. It remains to deal with (G16) and (G17). Let $F_a$, $F_b$ be two potentials both from (G16). Their abelianizations are $G_a=z^3+(1+a)xyz$ and $G_b=z^3+(1+b)xyz$. A linear sub turning $F_a$ to $F_b$ must transform $G_a$ into $G_b$. If $a=-1$, such a thing can obviously exist only if $b=-1$. Thus we can assume that $a\neq -1$ and $b\neq -1$. Now it is straightforward to check that that such subs are among the scalings $x\to px$, $y\to qy$ and $z\to rz$ or scalings composed with the swap of $x$ and $y$: $x\to py$, $y\to qx$ and $z\to rz$ with $p,q,r\in\K^*$, $r^3=1$. In order for an $F_a$ to be transformed to any $F_{a'}$, we need additionally $pqr=1$ in the first case and $pqra=1$ in the second. Analyzing the way how these subs act on $F_a$, we see that $F_a$ is transformed to itself if no swap is involved and to $F_{a^{-1}}$ otherwise. The situation with $F_a$ from (G17) can be analyzed in a similar way.
\end{proof}

\begin{lemma}\label{ms2}
Let $F\in{\cal P}_{3,3}$ and $A=A_F$. Then $H_A\geq (1-t)^{-3}$. Furthermore, the following statements are equivalent$:$
\begin{itemize}\itemsep=-2pt
\item[\rm (K1)]$A$ is exact$;$
\item[\rm (K2)]$H_A=(1-t)^{-3}$ and $A$ has no non-trivial right annihilators$;$
\item[\rm (K3)]$H_A=(1-t)^{-3}$ and $A$ is Koszul.
\end{itemize}
\end{lemma}

\begin{proof}The inequality  $H_A\geq (1-t)^{-3}$ and the equivalence of (K1) and (K2) follow from  Lemmas~\ref{commin} and~\ref{commin1} with $n=k=3$.  The equivalence of (K1) and (K3) follows from the already mentioned fact that the complex (\ref{comq}) coincides with the Koszul complex if the potential $F$ is proper, while the latter happens if and only if $\dim A_3=10$ (see Lemma~\ref{gen2}).
\end{proof}

\begin{lemma}\label{g18} Let $F\in{\cal P}_{3,3}$ be given by {\rm (G18)} of Lemma~$\ref{nco3-3}$. Then the potential algebra $A=A_F$ is Koszul, exact, non-PBW and satisfies $H_A=(1-t)^{-3}$.
\end{lemma}

\begin{proof} The fact that $A$, known also as a Sklyanin algebra, is Koszul and satisfies $H_A=(1-t)^{-3}$ is proved in \cite{TVB}. Different proofs are presented in \cite{SKL1} and \cite{SKL2}. In \cite{SKL1} it is shown that these algebras are non-PBW. Now, by Lemma~\ref{ms2},  Koszulity of $A$ yields its exactness.
\end{proof}

\begin{lemma}\label{smth1} Let $F\in{\cal P}^*_{3,3}$ be such that $xyz$, $yxz$ and $zxy$ are present in $F$ with non-zero coefficients, while $xxx$, $xxy$, $xxz$, $xyx$, $xyy$, $yxx$, $yxy$ and $zxx$ do not feature in $F$. Then $A=A_F$ is PBW, Koszul, exact and satisfies $H_A=(1-t)^{-3}$.
\end{lemma}

\begin{proof} By assumptions, the leading monomials of $\delta_z F$, $\delta_y F$ and $\delta_x F$  are $xy$, $xz$ and $yz$ respectively. Since the said monomials exhibit just one overlap, it must resolve by Lemma~\ref{sy}, turning the defining relations into a quadratic Gr\"obner basis and $\{xy,xz,yz\}$ into the set of leading monomials of members of the said basis. The equality $H_A=(1-t)^{-3}$ immediately follows. Since $A$ admits a quadratic Gr\"obner basis in the ideal of relations, $A$ is PBW and therefore Koszul. By Lemma~\ref{ms2}, $A$ is exact.
\end{proof}

\begin{lemma}\label{g11-17} Let $F\in{\cal P}_{3,3}$ be given by one of {\rm (G11--G17)} of Lemma~$\ref{nco3-3}$. Then the potential algebra $A=A_F$ is PBW, Koszul, exact and satisfies $H_A=(1-t)^{-3}$.
\end{lemma}

\begin{proof}
Just apply Lemma~\ref{smth1}: the potentials $F$ from each of {\rm (G11--G17)} satisfy the assumptions.
\end{proof}

\begin{lemma}\label{g1-8} Let $F\in{\cal P}_{3,3}$ be given by one of {\rm (Gj)} of Lemma~$\ref{nco3-3}$ with $1\leq j\leq 8$. Then the potential algebra $A=A_F$ is PBW, Koszul and non-exact. The Hilbert series of $A$ is given by  $H_A=(1-3t)^{-1}$ if $j=1$, $H_A=\frac{1+t}{1-2t-2t^2}$ if $j=2$, $H_A=\frac{1+t}{1-2t-t^2}$ if $j\in\{3,4\}$ and $H_A=\frac{1+t}{1-2t}$ if $5\leq j\leq 8$.
\end{lemma}

\begin{proof} An easy computation shows that the defining relations $\delta_x F$, $\delta_y F$ and $\delta_z F$ form a Gr\"obner basis in the ideal of relations of $A$. Hence $A$ is PBW and therefore Koszul. The computation of the Hilbert series is now easy and routine.
\end{proof}

\begin{lemma}\label{g9} Let $F\in{\cal P}_{3,3}$ be given by {\rm (G9)} and $A=A_F$. Then $A$ is
non-Koszul, non-PBW, non-exact, non-proper and satisfies $H_A=\frac{1+t+t^2+t^3+t^4}{1-2t+t^2-t^3-t^4}$.
\end{lemma}

\begin{proof} Since $F=z^3+xyz^\rcirclearrowleft$, the defining relations of $A$ are $yz$, $zx$ and $xy+zz$. The ideal of relations of $A$ turns out to have a finite Gr\"obner basis comprising $yz$, $zx$, $xy+zz$ and $zzz$. This allows us to find an explicit expression for the Hilbert series of $A$: $H_A=\frac{1+t+t^2+t^3+t^4}{1-2t+t^2-t^3-t^4}$. Next, one easily checks that the Koszul dual $A^!$ has the Hilbert series $H_{A^!}=1+3t+3t^2+2t^3$. Then the duality formula $H_A(t)H_{A^!}(-t)=1$ fails. Hence $A$ is non-Koszul  and therefore non-PBW. By Lemma~\ref{ms2}, $A$ is non-exact. The above formula for $H_A$ yields $\dim A_3=11$ and therefore $A$ is non-proper by Lemma~\ref{gen2}.
\end{proof}

\begin{lemma}\label{g10} Let $F\in{\cal P}_{3,3}$ be given by {\rm (G10)} and $A=A_F$. Then $A$ is proper, non-Koszul, non-PBW, non-exact and satisfies $H_A=\frac{1+2t+3t^2+3t^3+2t^4+t^5}{1-t-t^3-2t^4}$.
\end{lemma}

\begin{proof} Since $F=(y+z)^3+xyz^\rcirclearrowleft$, the defining relations of $A$ are $xy+(y+z)^2$, $zx+(y+z)^2$ and $yz$. The ideal of relations of $B$ turns out to have a finite Gr\"obner basis comprising $xy-zx$, $yy+zx+zy+zz$, $yz$, $xzx+xzy+xzz+zzx$, $zxz$ and $zzz$. This allows us to find an explicit expression for the Hilbert series of $A$: $H_A=\frac{1+2t+3t^2+3t^3+2t^4+t^5}{1-t-t^3-2t^4}$. Next, the dual algebra $A^!$ is easily seen to have the Hilbert series $H_{A^!}=1+3t+3t^2+t^3$. Clearly, the duality formula $H_A(t)H_{A^!}(-t)=1$ fails and therefore $A$ is non-Koszul. Hence $A$ is non-PBW. By Lemma~\ref{ms2}, $A$ is non-exact. The above formula for $H_A$ yields $\dim A_3=10$ and therefore $A$ is proper by Lemma~\ref{gen2}.
\end{proof}

\subsection{Proof of Theorem~\ref{main33-p}}

Combining Lemmas~\ref{nco3-3}, \ref{g18}, \ref{g11-17}, \ref{g1-8}, \ref{g9} and \ref{g10}, we see that all statements of Theorem~\ref{main33-p} hold with isomorphism of $A_F$ and $A_G$ condition replaced by equivalence of $F$ and $G$ (with respect to the $GL_3(\K)$ action by linear substitutions). By Remark~\ref{proper}, it remains to show that algebras (P10--P14) are pairwise non-isomorphic and algebras (P15--P18) are pairwise non-isomorphic. Since isomorphic graded algebras have the same Hilbert series, it remains to verify that four algebras from (P10--P13) are pairwise non-isomorphic and that the algebras from (P15) and (P16) are non-isomorphic. The latter holds because the algebra in (P15) is monomial, while the algebra in (P16) is not isomorphic to a monomial one. It remains to show that four algebras from (P10--P13) are pairwise non-isomorphic. The same argument on monomial algebras reduces the task to showing that the algebras in (P11) and (P12) are non-isomorphic and the algebras in (P10) and (P13) are non-isomorphic. The algebras in (P11) and (P12) are non-isomorphic since the (3-dimensional) space of quadratic relations for the first one is spanned by squares (of degree 1 elements) while the same space for the second algebra contains no squares at all. As for the algebras in (P10) and (P13), the second one sports just one (up to a scalar multiple) square in the space of quadratic relations, while the first one obviously has two linearly independent ones: $x^2$ and $z^2$.

\section{Twisted potential algebras $A_F$ with $F\in{\cal P}^*_{2,4}$}

We shall occasionally switch back and forth between denoting the generators $x,y$ or $x_1,x_2$ meaning $x=x_1$ and $y=x_2$. The reasons are aesthetic.

\begin{lemma}\label{d126} For $a\in\K^*$, $a\neq 1$, let $F_a=x^3y+ax^2yx+a^2xyx^2+a^3yx^3$ be the twisted potential of {\rm (T34)} of Theorem~$\ref{main24-t}$ and $A^a=A_{F_a}$. Then the twisted potential algebras $A^a$ are pairwise non-isomorphic, non-potential, non-proper and satisfy $H_{A^a}=\frac{1+t+t^2}{1-t-t^2}$.
\end{lemma}

\begin{proof} Clearly, $A^a$ is presented by generators $x,y$ and relations $x^2y+axyx+a^2yx^2$ and $x^3$. It is easy to check that the defining relations of $A^a$ form a Gr\"obner basis in the ideal of relations. Knowing the leading monomials $x^3$ and $x^2y$ of the members of a Gr\"obner basis, we easily confirm that $H_{A^a}=\frac{1+t+t^2}{1-t-t^2}$. Next, we show that algebras $A^a$ are pairwise non-isomorphic. Indeed, assume that a linear substitution facilitates an isomorphism between $A^a$ and $A^b$. As $x^3$ is the only cube (up to a scalar multiple) in the space of cubic relations for both $A^a$ and $A^b$, our sub must map $x$ to its own scalar multiple. Now it is elementary to check that that any such substitution leaves invariant each space $R_a$ spanned by $x^2y+axyx+a^2yx^2$ and $x^3$. Since $R_a$ are pairwise distinct, an isomorphism between $A^a$ and $A^b$ does exist only if $b=a$. Same type argument shows that each $A^a$ is not isomorphic to any of three algebras from (P24--P26) of the already proven Theorem~\ref{main24-p}. Since these three algebras are the only cubic potential algebras on two generators with the Hilbert series $\frac{1+t+t^2}{1-t-t^2}$, it follows that $A^a$ are non-potential.
\end{proof}

\begin{lemma}\label{firsttouch2} Each $F\in\K\langle x,y\rangle$ listed in {\rm (T24--T33)} of Theorem~$\ref{main24-t}$ is a proper twisted potential such that the Jordan normal form of the corresponding twist is one block with eigenvalue $-1$ for $F$ from {\rm (T25)}, one block with eigenvalue $1$ for $F$ from {\rm (T26--27)}, diagonalizable in all other cases with the two eigenvalues being $a,a^{-1}$ for $F$ from {\rm (T24)}, $a,-a^{-1}$ for $F$ from {\rm (T28)}, $\theta,1$ for $F$ from {\rm (T29)}, $\theta^2,1$ for $F$ from {\rm (T30)}, $\xi_8,-\xi_8$ for $F$ from {\rm (T31)}, $i\xi_8,-i\xi_8$ for $F$ from {\rm (T32)}, $-1,-1$ for $F$ from from {\rm (T33)}. Moreover $A_F$ is exact, non-potential and has the Hilbert series $(1+t)^{-1}(1-t)^{-3}$ for every $F$ from  {\rm (T24--T33)}.
\end{lemma}

\begin{proof} It is straightforward and elementary to check that each $F$ is a twisted potential with the Jordan normal form of the twist being as specified. For $F$ from (T24--T28) and (T33), a direct application of  Lemma~\ref{smth} shows that $A_F$ is exact and has the Hilbert series $(1+t)^{-1}(1-t)^{-3}$. Assume now that $F=x^3y+yx^3+axyx^2+a^2x^2yx+y^4$ with $a^3=1\neq a$. This covers (T29) and (T30). Then $A=A_F$ is presented by generators $x,y$ and relations $x^3+y^3$ and $x^2y+a^2xyx+ayx^2$. A direct computation shows that the defining relations together with $xy^3-y^3x$ and $xyxy^2-ayxyxy+a^2y^2xyx+2ay^3x^2$ form a Gr\"obner basis in the ideal of relations of $A$. This allows to compute the Hilbert series  $H_A=(1+t)^{-1}(1-t)^{-3}$ and to observe in the usual way that there are no non-trivial right annihilators in $A$. By Lemma~\ref{commin}, $A$ is exact. Next, assume that $F=x^4-ayx^3-y^2x^2+ay^3x+y^4+xy^3+x^2y^2+x^3y$ with $a^2=-1$. This covers (T31) and (T32). Again, the ideal of relations of $A=A_F$ has a finite Gr\"obner basis with the leading monomials of its members being $x^3$, $x^2y$, $xyx^2$, $xy^4$, $xy^2x^2$, $xyxy^2$ and $xy^2xy^2$ (7 members in total). We skip spelling out the exact formulas for the members of the basis this once since some of them turn out to be rather long. For instance, the last one is the sum of 9 terms. Anyway, knowing the above leading terms allows to  compute the Hilbert series  $H_A=(1+t)^{-1}(1-t)^{-3}$ and to observe in the usual way that there are no non-trivial right annihilators in $A$. By Lemma~\ref{commin}, $A_F$ is exact for $F$ from (T29--T32). By Lemma~\ref{gen2}, each $A_F$  for $F$ from (T24--T33) is proper. Since the corresponding twist (it is uniquely determined by $A_F$) is non-trivial, none of $A_F$ is potential.
\end{proof}

\begin{lemma}\label{1block2} Let $G\in{\cal P}^*_{2,4}$ be non-degenerate, $M\in GL_2(\K)$ be the unique matrix providing the twist for $G$ and assume that $A=A_G$ is non-potential. Assume also that the normal Jordan form of $M$ consists of one block. If $A$ is non-proper, then $A$ is isomorphic to an $A_F$ with $F$ from {\rm (T34)} of Theorem~$\ref{main24-t}$. If $A$ is proper, then $A$ is isomorphic to $A_F$ for $F$ from {\rm (T25--T27)} of Theorem~$\ref{main24-t}$. Moreover, algebras $A_F$ for $F$ from {\rm (T24--T27)} are pairwise non-isomorphic.
\end{lemma}

\begin{proof} By Remark~\ref{rer}, we can without loss of generality assume that
$$
M=\left(\begin{array}{cc}\alpha&1\\ 0&\alpha\end{array}\right).
$$
If $F=\sum\limits_{j,k,m,n=1}^2 a_{j,k,m,n}x_jx_kx_mx_n$, then the inclusion $F\in {\cal P}_{2,4}(M)$ is equivalent to the following system of linear equations on the coefficients of $F$:
\begin{equation}\label{b1-24}
a_{j,k,m,2}=\alpha a_{2,j,k,m}\ \ \text{and}\ \ a_{j,k,m,1}=\alpha a_{1,j,k,m}+a_{2,j,k,m}\ \ \text{for}\ \ 1\leq j,k,m\leq 2.
\end{equation}
One easily sees that (\ref{b1-24}) has only zero solution unless $\alpha^4=1$. This leaves three cases to consider: $\alpha^2=-1$, $\alpha=-1$ and $\alpha=1$.

If $\alpha^2=-1$, the space of solutions of (\ref{b1-24}) is one-dimensional, corresponding to ${\cal P}_{2,4}(M)$ being spanned by $G=yx^3+\alpha x^3y-\alpha xyx^2-x^2yx+\frac{1+\alpha}{2}x^4$. One easily sees that $A_G$ is isomorphic to the algebra from (T34) with $a=\alpha$. By Lemma~\ref{d126}, it is non-proper.

If $\alpha=-1$, solving (\ref{b1-24}), we see that
$$
\textstyle {\cal P}_{2,4}(M)=\{F_{s,t}=s\bigl(\frac12 x^4+yx^3-x^3y-xyx^2+x^2yx\bigr)+t(xyx^2-x^2yx-yx^2y-xy^2x+x^2y^2+y^2x^2):s,t\in\K\}.
$$
If $t=0$, $A_F$ with $F=F_{s,t}$ is easily seen to be isomorphic to the algebra from (T34) with $a=-1$. Thus we can assume that $t\neq 0$. A scaling turns $t$ into $1$, leaving us to deal with $F_{a,1}$ for $a\in\K$. These are twisted potentials from (T25). By Lemma~\ref{firsttouch2}, the corresponding algebras are proper. Using Remark~\ref{rer} and Lemma~\ref{gen2}, we see that only substitutions of the form $y\to py+qx$, $x\to px$ with $p\in\K^*$ and $q\in\K$ can provide an isomorphism between algebras in (T25). However none of these substitutions changes the parameter $a$. Hence algebras in (T25) are pairwise non-isomorphic.

If $\alpha=1$, solving (\ref{b1-24}), we see that
$$
{\cal P}_{2,4}(M)=\{F_{s,t,r}=s({x^2y^2}^\rcirclearrowleft-xyxy^\rcirclearrowleft-xyx^2+x^2yx)+tx^3y^\rcirclearrowleft+rx^4:s,t,r\in\K\}.
$$
If $s=0$, $F_{s,t,r}$ is a potential. Thus we can assume that $s\neq 0$. A substitution $x\to x$, $y\to y+bx$ with an appropriate $b\in\K$ kills $r$. Now a scaling turns $F$ into one of the twisted potentials from (T26) if $t\neq 0$ or to the twisted potential from (T27) if $t=0$. By Lemma~\ref{firsttouch2}, the corresponding algebras are proper. Same argument as above shows that they are pairwise non-isomorphic.

Algebras from (T25) and (T26--T27) can not be isomorphic since they are proper and the Jordan normal forms of the twists do not match.
\end{proof}

\begin{lemma}\label{2block2} Let $G\in{\cal P}^*_{2,4}$ be non-degenerate, $M\in GL_2(\K)$ be the unique matrix providing the twist for $G$ and assume that $A=A_G$ is non-potential. Assume also that $M$ is diagonalizable. If $A$ is non-proper, then $A$ is isomorphic to an $A_F$ with $F$ from {\rm (T34)} of Theorem~$\ref{main24-t}$. If $A$ is proper, then $A$ is isomorphic to $A_F$ for $F$ from {\rm (T24)} or {\rm (T28--T33)} of Theorem~$\ref{main24-t}$. Moreover, the corresponding algebras $A_F$ are pairwise non-isomorphic.
\end{lemma}

\begin{proof} By Remark~\ref{rer}, we can without loss of generality assume that
$$
M=\left(\begin{array}{cc}\alpha&0\\ 0&\beta\end{array}\right).
$$
If $F=\sum\limits_{j,k,m,n=1}^2 a_{j,k,m,n}x_jx_kx_mx_n$, then the inclusion $F\in {\cal P}_{2,4}(M)$ is equivalent to the following system of linear equations on the coefficients of $F$:
\begin{equation}\label{b2-24}
a_{j,k,m,2}=\beta a_{2,j,k,m}\ \ \text{and}\ \ a_{j,k,m,1}=\alpha a_{1,j,k,m}\ \ \text{for}\ \ 1\leq j,k,m\leq 2.
\end{equation}
One easily sees that (\ref{b2-24}) has only zero solution unless $1\in\{\alpha,\beta,\alpha^3\beta,\alpha^2\beta^2,\alpha\beta^3\}$. Moreover, the case $\alpha=\beta=1$ is excluded since $F\notin{\cal P}_{2,4}$. Now if $1\notin\{\alpha^3\beta,\alpha^2\beta^2,\alpha\beta^3\}$, then $F$ is both degenerate and a potential. Indeed, $F$ is either zero or is the fourth power of a degree $1$ element. The cases $\alpha^3\beta=1$ and $\alpha\beta^3=1$ transform to one another when we swap $x$ and $y$. Thus we have just two options to consider: $\alpha^3\beta=1$ or $\alpha^2\beta^2=1$.

First, assume that $\alpha^3\beta=1$. Then $x^3y+\alpha x^2yx+\alpha^2xyx^2+\alpha^3yx^3$ is in ${\cal P}_{2,4}(M)$. Moreover, analyzing (\ref{b2-24}), we see that ${\cal P}_{2,4}(M)$ is spanned by this one element unless $\alpha^3=1\neq \alpha$, or $\alpha^8=1\neq\alpha$. If $F$ is a scalar multiple of $x^3y+\alpha x^2yx+\alpha^2xyx^2+\alpha^3yx^3$, we fall into (T34) after scaling. The corresponding algebra is non-proper by Lemma~\ref{d126}.

Next, assume that $\alpha\beta=-1$. By (\ref{b2-24}), $x^2y^2+\alpha^2y^2x^2+\alpha xy^2x-\alpha yx^2y\in {\cal P}_{2,4}(M)$ and ${\cal P}_{2,4}(M)$ is spanned by this one element unless $\alpha^4=1$. If $F$ is a scalar multiple of $x^2y^2+\alpha^2y^2x^2+\alpha xy^2x-\alpha yx^2y$, a scaling sends $F$ to (T28). By Lemma~\ref{firsttouch2}, the corresponding algebras are proper. Using Remark~\ref{rer} and Lemma~\ref{gen2}, we see that only scalings can provide an isomorphism between algebras in (T28) (unless $\alpha=\pm i$ in which case a separate easy argument is needed; we skip it now since we study this case below in detail anyway). Now one sees that algebras in (T28) are pairwise non-isomorphic.

Next, assume that $\alpha\beta=1$. By (\ref{b2-24}), $x^2y^2{+}\alpha^2y^2x^2{+}\alpha xy^2x+\alpha yx^2y,xyxy+ayxyx\in {\cal P}_{2,4}(M)$ and ${\cal P}_{2,4}(M)$ is spanned by these two elements unless $\alpha=-1$. That is,
$$
F=s(x^2y^2+\alpha^2y^2x^2+\alpha xy^2x+\alpha yx^2y)+t(xyxy+\alpha yxyx)\ \ \text{with $s,t\in\K$}.
$$
If $s=0$, $A_F$ coincides with the potential algebra from (P26), which contradicts the assumptions. Thus $s\neq 0$. Now a scaling turns $F$ into the form (T24). By Lemma~\ref{firsttouch2}, the corresponding algebras are proper. Using Remark~\ref{rer} and Lemma~\ref{gen2}, we see that only scalings can provide an isomorphism between algebras in (T28) (unless $\alpha=\pm 1$ in which case a separate easy argument is needed; we skip it now since we study this case below in detail anyway). Now one sees that algebras in (T24) are pairwise non-isomorphic.

It remains to consider the following finite set of options for $(\alpha,\beta)$: $(1,-1)$, $(-1,-1)$, $(a,1)$ with $a^3=1\neq a$, $(a,a)$ with $a^2=-1$ and $(a,-a)$ with $a^4=-1$.

If $(\alpha,\beta)=(a,1)$ with $a^3=1\neq a$, then solving (\ref{b2-24}), we see that
$$
{\cal P}_{2,4}(M)=\{F_{s,t}=s(x^3y+yx^3+axyx^2+a^2x^2yx)+ty^4:s,t\in\K\}.
$$
If $s=0$, $F_{s,t}$ is a potential. If $t=0$, $F_{s,t}$ falls into (T34). Thus we can assume that $st\neq0$. Now a scaling transforms $F_{s,t}$ into $F_{1,1}$, which is the twisted potential from (T29) if $a=\theta$ and (T30) if $a=\theta^2$. In both cases, the corresponding algebra is proper by Lemma~\ref{firsttouch2}.

Now assume that $(\alpha,\beta)=(b,-b)$ with $b^4=-1$. Since changing $b$ by $-b$ corresponds to swapping $x$ and $y$, we can assume that $b\in\{\xi_8,i\xi_8\}$.
Solving (\ref{b2-24}), we see that
$$
{\cal P}_{2,4}(M)=\{G_{s,t}=s(x^2yx-b^3x^3y+b^2yx^3+bxyx^2)+t(y^3x-b^3xy^3+b^2yxy^2-by^2xy):s,t\in\K\}
$$
If $st=0$, the corresponding algebra is easily seen to fall into (T34): just scale or swap $x$ and $y$ and scale. By Lemma~\ref{d126}, the corresponding algebra is non-proper if $st=0$. If $st\neq 0$, a scaling turns both $s$ and $t$ into $1$. That is, if $st\neq 0$, $F_{s,t}$ is equivalent to $F_{1,1}$, which is easily seen to be proper. That is, $F_{s,t}$, if proper, is isomorphic to one of two algebras $F_{1,1}$ for $b=\xi_8$ or $F_{1,1}$ for $b=i\xi_8$. Denote $a=b^2$. That is, $a=i$ if $b=\xi_8$ and $a=-i$ if $b=i\xi_8$. In both cases $a^2=-1$. Note that the matrix
$$
N=\left(\begin{array}{cc}0&a\\ 1&0\end{array}\right)
$$
is conjugate to $M$. It is easy to see that $G=x^4-ayx^3-y^2x^2+ay^3x+y^4+xy^3+x^2y^2+x^3y$ belongs to ${\cal P}_{2,4}(N)$ and therefore $G$ is equivalent to a member of ${\cal P}_{2,4}(M)$. This produces two algebras from (T31) and (T32). By Lemma~\ref{firsttouch2}, they are proper and therefore they must be isomorphic to algebras given by $F_{1,1}$ for $b=\xi_8$ or $F_{1,1}$ for $b=i\xi_8$.

If $(\alpha,\beta)=(1,-1)$, then  solving (\ref{b2-24}), we see that
$$
{\cal P}_{2,4}(M)=\{F_{s,t}=s(x^2y^2+y^2x^2-xy^2x+yx^2y)+tx^4:s,t\in\K\}
$$
If $s=0$, $F_{s,t}$ is a potential. If $s\neq 0$, a substitution given by $x\to px$, $y\to py+qx$ with appropriate $p\in\K^*$ and $q\in\K$ turns $(s,t)$ into $(1,0)$. Then $F$ falls into (T28) with $a=1$.

Assume now that $(\alpha,\beta)=(a,a)$ with $a^2=-1$. Solving (\ref{b2-24}), we see that ${\cal P}_{2,4}(M)$ consists of
$$
G_{s,t,r}=s(x^3y{+}ayx^3{-xyx}^2{-}ax^2yx)+t(x^2y^2{+}ayx^2{-}y^2x^2{-}axy^2x)+r(y^3x{+}axy^3{-}yxy^2{-}ay^2xy)
$$
with $s,t,r\in\K$. If we perform a (non-degenerate) linear substitution $x\to \lambda x+\mu y$, $y\to \gamma x+\delta y$ with $D=\lambda\delta-\mu\gamma\in\K^*$, $G_{s,t,r}$ transforms into $G_{s',t',r'}$ with
$$
s'=D(\lambda^2s-\gamma^2r+(1+a)\lambda\gamma t),\ \ r'=D(-\mu^2s+\delta^2r-(1+a)\mu\delta t),\ \
t'=D((1-a)\lambda\mu s-(1-a)\gamma\delta r+(\lambda\delta+\mu\gamma)t).
$$
Now it is routine to show that if $at^2+2sr\neq 0$, the substitution can be chosen in such a way that $s'=r'=0$ and $t'=1$ (note that $F$ is non-degenerate and therefore $(s,t,r)\neq (0,0,0)$). Otherwise, a substitution can be chosen in such a way that $r'=t'=0$ and $r'=1$. In the first case $F$ falls to (T24) with $b=0$, while in the second case it falls into (T34).

Finally, assume that $(\alpha,\beta)=(-1,-1)$.  Solving (\ref{b2-24}), we see that ${\cal P}_{2,4}(M)$ consists of
$$
G_{s,t,r,u}=s(x^3y{-}yx^3{+}xyx^2{-}x^2yx)+t(x^2y^2{-}yx^2{+}y^2x^2{-}xy^2x)+r(y^3x{-}xy^3{+}yxy^2{-}y^2xy)+u(xyxy{-}yxyx).
$$
with $s,t,r,u\in\K$. If we perform a (non-degenerate) linear substitution $x\to \lambda x+\mu y$, $y\to \gamma x+\delta y$ with $D=\lambda\delta-\mu\gamma\in\K^*$, $G_{s,t,r,u}$ transforms into $G_{s',t',r',u'}$ with
$$
s'=D(\lambda^2s-\gamma^2r+\lambda\gamma u),\ \ r'=D(-\mu^2s+\delta^2r-\mu\delta u),\ \
t'=D^2t,\ \ u'=D(2\lambda\mu s-2\gamma\delta r+(\lambda\delta+\mu\gamma)u).
$$
It is routine to show that if $u^2+4sr\neq 0$, the substitution can be chosen in such a way that $s'=r'=0$ and $F$ transforms into the form (D8) or (D9) after an appropriate scaling. It remains to consider the case  $u^2+4sr=0$. Note that this property is invariant under linear substitutions (one easily sees that ${u'}^2+4s'r'=0$). Clearly a substitution can be chosen in such a way that $s'=0$. If it happens that $t'=0$, then $F$ falls into (T34) after swapping $x$ and $y$ and scaling. Thus we can assume that $t'\neq 0$. The equation ${u'}^2+4s'r'=0$ yields $u'=0$. Then a scaling turns $F$ into the twisted potential from (T33), which is proper by  Lemma~\ref{firsttouch2}.

As for algebras from (T24) and (T28--T33) being pairwise non-isomorphic it is a consequence of the following observations. By Lemma~\ref{gen2} the question reduces to pairwise non-equivalence of corresponding twisted potentials, which certainly holds when the twists have different Jordan normal form. As for two $F$'s from the same ${\cal P}_{2,4}(M)$, $M$ being in Jordan form, they are equivalent precisely when a linear substitution with a matrix whose transpose commutes with $M$ transforms one $F$ to the other. In the case of one Jordan block this leaves us to consider substitutions of the form $x\to px$, $y\to py+qx$ with $p\in\K^*$ and $q\in\K$. If $M$ is diagonal but not scalar, we are left only with scalings. Finally, if $M$ is scalar, we have to deal with the entire $GL_2(\K)$. Fortunately, this only concerns the cases $(\alpha,\beta)=(a,a)$ with $a^2=-1$ or $a=-1$ in which we gave the explicit formulae for how the substitutions act on ${\cal P}_{2,4}(M)$. Now the stated non-equivalence is a matter for a direct verification.
\end{proof}

\subsection{Proof of Theorem~\ref{main24-t}}

Recall that any two proper (degree-graded) twisted potential algebras with non-equivalent twisted potentials are non-isomorphic and a proper twisted potential algebra can not be isomorphic to a non-proper one. Next, if $F\in{\cal P}^*_{2,4}$ is degenerate, then it is easily seen that $F$ is either $0$ or is a fourth power of a degree $1$ element. In particular, $F$ is a potential and therefore $A_F$ does not satisfy the assumptions. Taking this into account, we see that Lemmas~\ref{d126}, \ref{firsttouch2}, \ref{1block2} and~\ref{2block2} imply Theorem~\ref{main24-t}.

\section{Twisted potential algebras $A_F$ with $F\in{\cal P}^*_{3,3}$}

This section is devoted to the proof of Theorem~\ref{main33-t}. We shall occasionally switch back and forth between denoting them $x,y,z$ or $x_1,x_2,x_3$ meaning $x=x_1$, $y=x_2$ and $z=x_3$. The reasons are aesthetic. In this section we always use the left-to-right degree-lexicographical order on monomials in $x,y,z$ assuming $x>y>z$.

\begin{lemma}\label{degen} The algebras $A$ given by {\rm (T19--T21)} of Theorem~$\ref{main33-t}$ are non-proper non-degenerate non-potential twisted potential algebras. They are pairwise non-isomorphic, PBW, Koszul and have Hilbert series $H_A=\frac{1+t}{1-2t}$ as specified in {\rm (T19--T21)}.
\end{lemma}

\begin{proof} Algebras $A_a$ from (T20) are presented by the defining relations $xx$, $xy+ayx$ and $xz+a^2zx+yy$ with $a\neq 0$ and $a\neq 1$. Algebras $B_a$ from (T19) are presented by $xx$, $zz$ and $xy+ayx$ with $a\neq 0$ and $a\neq 1$. Finally the defining relations of the algebra $C$ from (T21) are $xx+zz$, $xz-zx$ and $yy$. For all these algebras, the defining relations form a Gr\"obner basis in the ideal of relations. Hence all algebras in question are PBW and therefore Koszul. Knowing the leading monomials for elements of a Gr\"obner basis, we can easily compute the Hilbert series: $H_A=\frac{1+t}{1-2t}$ in all cases. By Lemma~\ref{gen2} all these algebras are non-proper. By Theorem~\ref{main33-p}, in order to show that none of these algebras is potential, it is enough to verify that none of them is isomorphic to any of the four algebras (P10--P13) (the only potential algebras with the Hilbert series $\frac{1+t}{1-2t}$), which is an elementary exercise.

Since each $B_a$ has two linearly independent squares in the space of quadratic relations, while none of $A_a$ or $C$ has such a thing, $B_a$ is non-isomorphic to any $A_b$ or $C$. The latter is singled out by the existence of a decomposition $V=V_1\oplus V_2$ with $1$-dimensional $V_1$ for which the space of quadratic relations lies in $V_1^2+V_2^2$. It remains to verify that $A_a$ are pairwise non-isomorphic and that $B_a$ are pairwise non-isomorphic. Assume that $A_a$ is isomorphic to $A_b$. Since $x^2$ is the only square (up to a scalar multiple) in in the space of quadratic relations for both algebras, a linear substitution providing an isomorphism must map $x$ to its scalar multiple. Without loss of generality, $x$ is mapped to $x$. For both algebras, the quotient by the ideal generated by $x$ is presented by generators $y,z$ and one relation $y^2$. Hence our substitution must map $y$ to $\alpha y+\beta x$ with $\alpha,\beta\in\K$, $\alpha\neq 0$. It easily follows that $xy+ayx$ is mapped to a scalar multiple of itself plus a scalar multiple of $xx$. Thus $xy+ayx$ must be a relation of $A_{F_b}$, which yields $a=b$. Finally, assume that $B_a$ is isomorphic to $B_b$. Since $x^2$ and $z^2$ are the only squares (up to scalar multiples) in in the space of quadratic relations for both algebras, a linear substitution providing an isomorphism must either map $x$ and $z$ to their own scalar multiples or map $x$ to a scalar multiple of $z$ and $z$ to a scalar multiple of $x$. In the second case, $B_b$ has a relation of the form $zu+auz$ for some homogeneous degree one $u$ non-proportional to $z$, which is obviously nonsense ($u$ is the image of $y$ under our substitution). Hence $x$ and $z$ are mapped to their own scalar multiples. Now $B_b$ has a relation of the form $xu+aux$ with $u=y+\alpha z$ with $\alpha\in\K$. This is only possible if $b=a$ (and $\alpha=0$).
\end{proof}

\begin{lemma}\label{firsttouch} Each $F\in\K\langle x,y,z\rangle$ listed in {\rm (T1--T18)} of Theorem~$\ref{main33-t}$ is a proper twisted potential such that the Jordan normal form of the corresponding twist is one block with eigenvalue $1$ for $F$ from {\rm (T3--T4)}, two blocks of sizes $2$ and $1$ with eigenvalues $b$ and $b^{-2}$ respectively for $F$ from {\rm (T5)}, two blocks of sizes $2$ and $1$ with eigenvalues $-1$ and $1$ respectively for $F$ from {\rm (T6)}, two blocks of sizes $2$ and $1$ with both eigenvalues $1$ for $F$ from {\rm (T7)}, diagonalizable in all other cases with the three eigenvalues being $\frac{a}{b},\frac{b}{c},\frac{c}{a}$ for $F$ from {\rm (T1)}, $\frac{a}{b},\frac{b}{a},1$ for $F$ from {\rm (T2)}, $a,a,a^{-2}$ for $F$ from {\rm (T9)}, $1,-1,-1$ for $F$ from {\rm (T10)}, $a,-a,a^{-2}$ for $F$ from {\rm (T11)}, $-1,1,1$ for $F$ from {\rm (T16--T18)}, $i,-1,1$ for $F$ from from {\rm (T12)}, $-i,-1,1$ for $F$ from from {\rm (T13)}, $\xi_9,\xi_9^4,\xi_9^7$ for $F$ from from {\rm (T14)} and $\xi_9^2,\xi_9^5,\xi_9^8$ for $F$ from from {\rm (T15)}.

Moreover $A_F$ is Koszul, exact, non-potential and has the Hilbert series $(1-t)^{-3}$ for every $F$ from  {\rm (T1--T18)} and $A_F$ is PBW for $F$ from {\rm (T1--T11)} and {\rm (T16--T17)}.
\end{lemma}

\begin{proof} It is straightforward and elementary to check that each $F$ is a twisted potential with the Jordan normal form of the twist being as specified. For $F$ from (T1--T10), the defining relations as given in Theorem~\ref{main33-t} are easily seen to form a Gr\"obner basis in the ideal of relations. Thus for such $F$, $A_F$ are PBW and therefore Koszul and we immediately get $H_A=(1-t)^{-3}$. For $F$ from (T11), we perform the substitution $x\to x$, $y\to y+ix$, $z\to z$, which turns the defining relations into $xz+azx$, $yz-azy-2aizx$ and $xy+yx-iyy$. Now they form a Gr\"obner basis in the ideal of relations, showing that $A_F$ is PBW, Koszul and has the Hilbert series $(1-t)^{-3}$. For $F$ from (T16), we perform the same substitution $x\to x$, $y\to y+ix$, $z\to z$, which turns the defining relations into $xz-zx$, $yz+zy-2izx$ and $xy+yx-iyy-izz$, which now form a Gr\"obner basis in the ideal of relations. Thus $A_F$ is PBW, Koszul and has the Hilbert series $(1-t)^{-3}$. For $F$ from (T17), we first swap $y$ and $z$ turning the defining relations into $yy+zz$, $xy-yx$ and $xz+zx+zz$. Next, we follow up with the substitution $x\to x$, $y\to y$ and $z\to z+iy$, turning the defining relations into $xy-yx$, $yz+zy-izz$ and $xz+zx+2iyx-yy$. Now they form a Gr\"obner basis in the ideal of relations, showing that $A_F$ is PBW, Koszul and has the Hilbert series $(1-t)^{-3}$. Note also that for $F$ from (T1--T11) and (T16--T17) $A_F$ has no non-trivial right annihilators as no leading monomial of an element of the above quadratic Gre\"obner bases starts with $z$.

Now we shall show that for $F$ from (T12--T15) and (T18), $A_F$ has the Hilbert series $(1-t)^{-3}$ and has no non-trivial right annihilators. For $F$ from (T12--T13), we swap $x$ and $y$ to bring the defining relations to the form $xx+yy$, $xy-yx+zz$ and $xz\pm izx$. A direct computation shows that the defining relations together with $yyz+zyy$ and $yzz+zzy$ form a Gr\"obner basis in the ideal of relations of $A_F$. This allows to confirm that the Hilbert series of $A_F$ is indeed $(1-t)^{-3}$. Since none of the leading monomials of elements of the above Gr\"obner basis starts with $z$, $A_F$ has non non-trivial right annihilators. The cases of $F$ from (T14) and (T15) are identical (just swap $\theta$ and $\theta^2$). The substitution $x\to x$, $y\to y-\alpha x$, $z\to z$ turns the defining relations of $A_F$ for $F$ from (T14) into $xx-\theta^2yx-zy$, $xy+\theta^2xz-yx-\theta zx+(1-\theta)zy$ and $yz-\theta zy$. Again, we are in a finite Gr\"obner basis situation. Namely, the defining relations together with $xzy+\theta xzz-\theta^2yxz-zxz+(1-\theta)zzy$, $xzx+xzz-\theta yxz-\theta^2zxz-zzy$, $xzzz-\theta yxzz-\theta^2 zxzz+\theta^2zzyx+(1-\theta)zzzy$ and $xzzy-\theta^2 zzyx$  form a Gr\"obner basis in the ideal of relations of $A_F$. As above this allows to conclude that $(1-t)^{-3}$ is the Hilbert series of $A_F$ and that $A_F$ has non non-trivial right annihilators. It remains to consider $A_F$ with $F$ from (T18). After swapping $y$ and $z$, the defining relations of $A_F$ for $F$ from (T18) take shape $xy-yx$, $xx+ayy+yz+zy$ and $yy+zz$. A direct computation shows that the defining relations together with $yzz-zzy$, $xzz-zzx$ and $xzy+yxz-yzx-zyx$ form a Gr\"obner basis in the ideal of relations of $A_F$. As above, $(1-t)^{-3}$ is the Hilbert series of $A_F$ and that $A_F$ has non non-trivial right annihilators. Now Lemma~\ref{commin} implies that $A_F$ for $F$ from (T1--T18) is exact and Koszul, while Lemma~\ref{gen2} says that it is proper. As each $A_F$ is proper and the corresponding twist (it is uniquely determined by $A_F$) is non-trivial, none of $A_F$ is potential.
\end{proof}

On few occasions we need to show that certain quadratic algebras are non-PBW.

\begin{lemma}\label{non-pbwb-all}
Let $F$ be the twisted potential from {\rm (T12--T15)} or {\rm (T18)} of Theorem~$\ref{main33-t}$ and $A=A_F$. Then $A$ is non-PBW.
\end{lemma}

\begin{proof} Since the PBW property is preserved when one passes to the opposite multiplication and the algebras from (T12) and (T13) as well as the algebras from (T14) and (T15) are isomorphic to each other's opposites, it is enough to deal with $F$ from (T12), (T14) and (T18). That is, $A$ is presented by the generators $x,y,z$ and three quadratic relations $r_1$, $r_2$ and $r_3$ from the following list:
\begin{itemize}\itemsep=-2pt
\item[\rm (1)] $r_1=xx+yy$, $r_2=xy-yx+zz$ and $r_3=xz+izx;$
\item[\rm (2)] $r_1=xz-zx$, $r_2=xx+yz+zy+azz$ and $r_3=yy+zz$, where $a\in\K$, $a^2+4\neq 0;$
\item[\rm (3)] $r_1=yx+\theta zy+\theta^2zx$, $r_2=xy+zy+\theta^2xz$ and $r_3=yx+yz+\theta xz.$
\end{itemize}

Assume the contrary: $A$ is PBW. By Lemma~\ref{firsttouch}, $H_A=(1-t)^{-3}$ and therefore $\dim A_1=3$, $\dim A_2=6$ and $\dim A_3=10$. By Lemma~\ref{ome0}, there exists a well-ordering $\leq$ on the $x,y,z$ monomials compatible with multiplication and satisfying $x>y>z$ (this we can acquire by permuting the variables) and a non-degenerate linear substitution $x\mapsto ux+\alpha_1 y+\beta_1z$, $y\mapsto vx+\alpha_2y+\beta_2 z$, $z\mapsto wx+\alpha_3y+\beta_3 z$ such that the leading monomials $m_1,m_2,m_3$ of the new space of defining relations satisfy
\begin{equation}\label{lemo1}
\{m_1,m_2,m_3\}{\in}\bigl\{\!\{xy,xz,yz\},\{xy,xz,zy\},\{xy,zx,zy\},
\{yx,yz,xz\},\{yx,yz,zx\},\{yx,zy,zx\}\!\bigr\}.
\end{equation}
Note that we do not assume that $m_1>m_2>m_3$ here. Since $xx$ is the biggest degree $2$ monomial,
\begin{equation}\label{noxx}
\text{$xx$ is absent in each of $r_j$ after the substitution.}
\end{equation}
Since the order satisfies $x>y>z$ and is compatible with multiplication,
\begin{equation}\label{4big2}
\text{four biggest degree $2$ monomials are either $xx,xy,yx,xz$ or $xx,xy,yx,zx$}
\end{equation}
(not necessarily in this order).

{\bf Case 1:} \ $r_j$ are given by (1). In this case (\ref{noxx}) reads $0=uw=w^2=u^2+v^2$. Since our substitution is non-degenerate $(u,v,w)\neq (0,0,0)$. By scaling $x$ (this does not effect the leading monomials), we can assume that $u=1$. Then $w=0$ and $v^2=-1$. The following table gives the coefficients in $r_j$ in front of certain monomials:
$$
\begin{matrix}
&\vrule&xx&xy&yx&xz&zx&yy\\
\noalign{\hrule}
r_1\hhhhh&\vrule&0&\alpha_1+v\alpha_2&\alpha_1+v\alpha_2&\beta_1+v\beta_2&\beta_1+v\beta_2&\alpha_1^2+\alpha_2^2\\
r_2&\vrule&0&\alpha_2-v\alpha_1&-v\alpha_1-\alpha_2&\beta_2-v\beta_1&-v\beta_1-\beta_2&\alpha^2_3\\
r_3&\vrule&0&\alpha_3&i\alpha_3&\beta_3&i\beta_3&\alpha_1\alpha_3(1+i)
\end{matrix}
$$

Using the fact that our sub is non-degenerate, we easily see that if $\alpha_1+v\alpha_2\neq 0$, then the both $3\times 3$ matrices of coefficients of $xy$, $yx$ and $xz$ and of $xy$, $yx$ and $zx$ are non-degenerate. By (\ref{4big2}), in the case $\alpha_1+v\alpha_2\neq 0$, the set of leading monomials of the relations is either $\{xy,yx,xz\}$ or $\{xy,yx,zx\}$, contradicting (\ref{lemo1}). Hence we must have $\alpha_1+v\alpha_2=0$. Using the fact that $v=\pm i$, we see that the above table takes the following form:
$$
\begin{matrix}
&\vrule&xx&xy&yx&xz&zx&yy\\
\noalign{\hrule}
r_1\hhhhh&\vrule&0&0&0&\beta_1+v\beta_2&\beta_1+v\beta_2&0\\
r_2&\vrule&0&0&0&\beta_2-v\beta_1&-v\beta_1-\beta_2&\alpha^2_3\\
r_3&\vrule&0&\alpha_3&i\alpha_3&\beta_3&i\beta_3&\alpha_1\alpha_3(1+i)
\end{matrix}
$$

Then $\alpha_3\neq 0$ (otherwise both $xy$ and $yx$ are not among the leading monomials, contradicting (\ref{lemo1})) and $\beta_1+v\beta_2\neq 0$ (otherwise both $xz$ and $zx$ are not among the leading monomials, contradicting (\ref{lemo1})). Now, one easily sees that the $yy$-column of the above matrix is not in the linear span of any of the following pair of columns: $xy$ and $xz$, $xy$ and $zx$, $yx$ and $xz$ or $yx$ and $zx$. Since $yy>yz$, $yy>zy$ and $yy>zz$, it follows that $yy$ is among the leading monomials of the relations, which contradicts (\ref{lemo1}). This contradiction completes the proof in Case~1.

{\bf Case 2}  \ $r_j$ are given by (2). In this case (\ref{noxx}) reads
$$
0=uv+\theta vw+\theta^2uw=uv+vw+\theta^2uw=uv+vw+\theta uw,
$$
which is equivalent to $uv=vw=uw=0$. Hence exactly two of $u$, $v$ and $w$ are zero and we can normalize to make the third equal $1$. Since cyclic permutations of the variables composed with appropriate scalings provide automorphisms of our algebra, we can without loss of generality assume that $u=1$ and $v=w=0$.
The following table gives the coefficients in $r_j$ in front of certain monomials:
$$
\begin{matrix}
&\vrule&xx&xy&yx&xz&zx&yy\\
\noalign{\hrule}
r_1\hhhhh&\vrule&0&0&\alpha_2+\theta^2\alpha_3&0&\beta_2+\theta^2\beta_3&\alpha_1\alpha_2+\theta\alpha_2\alpha_3+\theta^2\alpha_1\alpha_3\\
r_2&\vrule&0&\alpha_2+\theta^2\alpha_3&0&\beta_2+\theta^2\beta_3&0&\alpha_1\alpha_2+\alpha_2\alpha_3+\theta^2\alpha_1\alpha_3\\
r_3&\vrule&0&\theta\alpha_3&\alpha_2&\theta\beta_3&\beta_2&\alpha_1\alpha_2+\alpha_2\alpha_3+\theta\alpha_1\alpha_3
\end{matrix}
$$

Using the fact that our sub is non-degenerate, we easily see that if $\alpha_2+\theta^2\alpha_3\neq 0$, then the both $3\times 3$ matrices of coefficients of $xy$, $yx$ and $xz$ and of $xy$, $yx$ and $zx$ are non-degenerate. By (\ref{4big2}), in the case $\alpha_2+\theta^2\alpha_3\neq 0$, the set of leading monomials of the relations is either $\{xy,yx,xz\}$ or $\{xy,yx,zx\}$, contradicting (\ref{lemo1}). Hence we must have $\alpha_2+\theta^2\alpha_3=0$. The above table takes the following form.
$$
\begin{matrix}
&\vrule&xx&xy&yx&xz&zx&yy\\
\noalign{\hrule}
r_1\hhhhh&\vrule&0&0&0&0&\beta_2+\theta^2\beta_3&-\theta\alpha_3^2\\
r_2&\vrule&0&0&0&\beta_2+\theta^2\beta_3&0&-\theta^2\alpha_3^2\\
r_3&\vrule&0&\theta\alpha_3&\alpha_2&\theta\beta_3&\beta_2&*
\end{matrix}
$$

Now unless $\alpha_3(\beta_2+\theta^2\beta_3)=0$, all six $3\times 3$ matrices of coefficients of $xy$, $xz$ and $yy$; $yx$, $xz$ and $yy$; $xy$, $zx$ and $yy$; $yx$, $zx$ and $yy$; $xy$, $xz$ and $zx$; $yx$, $xz$ and $yy$ are non-degenerate. The latter means that either $xz$ and $zx$ or $yy$ are among the leading monomials of the defining relations, which contradicts (\ref{lemo1}). Thus we must have $\alpha_3(\beta_2+\theta^2\beta_3)=0$ and $\alpha_2+\theta^2\alpha_3=0$, which contradicts the fact that our substitution is non-degenerate. This contradiction completes the proof in Case~2.

{\bf Case 3:} \ $r_j$ are given by (3).

Since this class of algebras is closed (up to an isomorphism) with respect to passing to the opposite multiplication and the two options in (\ref{4big2}) reduce to one another via passing to the opposite multiplication, for the rest of the proof we can assume that
\begin{equation}\label{4big}
\text{four biggest degree $2$ monomials are $xx,xy,yx,xz$.}
\end{equation}

In the current case (\ref{noxx}) reads $0=u^2+2vw+aw^2=v^2+w^2$. Since $(u,v,w)\neq (0,0,0)$, we have $v\neq 0$, which allows to normalize: $v=1$. Then $w\in\{i,-i\}$ and $a=2w-u^2$. Since $a^2+4\neq 0$ and $w^2=-1$, we have $u\neq 0$. It is easy to see that we can split our substitution into two consecutive substitutions: first, $x\mapsto ux$, $y\mapsto x+y$, $z\mapsto wx+z$ and, second, $x\mapsto x+\alpha_1 y+\beta_1z$, $y\mapsto \alpha_2y+\beta_2 z$, $z\mapsto \alpha_3y+\beta_3 z$ ($\alpha_j$ and $\beta_j$ are not the same as before).
After the first substitution, the defining relations are spanned by $r_1=xz-zx$, $r_2=u^2(xz+zx)+w(yz+zy)+yy-(1+wu^2)zz$ and $r_3=u^2(xy+yx)+(u^2-w)yy+(yz+zy)+wzz$.
The following table gives the coefficients in $r_j$ in front of certain monomials after the second substitution:

$$
\begin{matrix}
&\vrule&xx&xy&yx&xz&zx\\
\noalign{\hrule}
r_1\hhhhh&\vrule&0&\alpha_3&-\alpha_3&\beta_3&-\beta_3\\
r_2&\vrule&0&u^2\alpha_3&u^2\alpha_3&u^2\beta_3&u^2\beta_3\\
r_3&\vrule&0&u^2\alpha_2&u^2\alpha_2&u^2\beta_2&u^2\beta_2
\end{matrix}
$$

Using the fact that our sub is non-degenerate, we easily see that if $\alpha_3\neq 0$, then the $3\times 3$ matrix of coefficients of $xy$, $yx$ and $xz$ is non-degenerate. By (\ref{4big}), both $xy$ and $yx$ are among the leading monomials, contradicting (\ref{lemo1}). Hence we must have $\alpha_3=0$. The above table with the extra $yy$-column takes the following form:

$$
\begin{matrix}
&\vrule&xx&xy&yx&xz&zx&yy\\
\noalign{\hrule}
r_1\hhhhh&\vrule&0&0&0&\beta_3&-\beta_3&0\\
r_2&\vrule&0&0&0&u^2\beta_3&u^2\beta_3&\alpha_2^2\\
r_3&\vrule&0&u^2\alpha_2&u^2\alpha_2&u^2\beta_2&u^2\beta_2&(u^2-w)\alpha_2^2
\end{matrix}
$$

Same way as in Case 2, it follows that unless $\alpha_2\beta_3=0$, either $xz$ and $zx$ or $yy$ feature among the leading monomials. Thus $\alpha_2\beta_3=\alpha_3=0$, which contradicts the fact that our substitution is non-degenerate. This contradiction completes the proof in the final Case~3.
\end{proof}

Now we deal with the possibilities for the Jordan normal form of the twist for $F\in{\cal P}^*_{3,3}$ one by one.

\begin{lemma}\label{1block} Let $G\in{\cal P}^*_{3,3}$ be non-degenerate, $M\in GL_3(\K)$ be the unique matrix providing the twist for $G$ and assume that $A=A_G$ is non-potential. Assume also that the normal Jordan form of $M$ consists of one block. If $A$ is non-proper, then $A$ is isomorphic to $A_F$ with $F$ from {\rm (T20)} of Theorem~$\ref{main33-t}$ with $a^3=1\neq a$. If $A$ is proper, then $A$ is isomorphic to $A_F$ for $F$ from {\rm (T3)} or {\rm (T4)} of Theorem~$\ref{main33-t}$. Moreover, algebras $A_F$ for $F$ from {\rm (T3)} and {\rm (T4)} are pairwise non-isomorphic.
\end{lemma}

\begin{proof} By Remark~\ref{rer}, we can without loss of generality assume that
$$
M=\left(\begin{array}{ccc}\alpha&1&0\\ 0&\alpha&1\\ 0&0&\alpha\end{array}\right)\ \ \ \text{with $\alpha\in\K^*$.}
$$
If $G=\sum\limits_{j,k,m=1}^3 a_{j,k,m}x_jx_kx_m$, then the inclusion $G\in {\cal P}_{3,3}(M)$ is equivalent to the following system of linear equations on the coefficients of $G$:
\begin{equation}\label{b1-33}
a_{j,k,3}=\alpha a_{3,j,k},\ \
a_{j,k,2}=\alpha a_{2,j,k} +a_{3,j,k}\ \ \text{and}\ \ a_{j,k,1}=\alpha a_{1,j,k}+a_{2,j,k}\ \ \text{for}\ \ 1\leq j,k\leq 3.
\end{equation}
One easily sees that (\ref{b1-33}) has only zero solution unless $\alpha^3=1$. This leaves two cases to consider: $\alpha^3=1\neq\alpha$ and $\alpha=1$.

If $\alpha^3=1\neq \alpha$, solving (\ref{b1-33}), we see that $G$ belongs to ${\cal P}_{3,3}(M)$ precisely when
$$
\textstyle G=sx^2z+\alpha^2szx^2+\alpha sxzx-\alpha^2sxy^2-sy^2x-\alpha syxy+tx^2y+\alpha(\alpha t-s)yx^2+\alpha(t-s)xyx+\frac{\alpha t-s}{\alpha^2-1}x^3
$$
with $s,t\in\K$. Since $G$ is non-degenerate, $s\neq 0$. By scaling, we can turn $s$ into $1$. Now the space of quadratic relations of $A=A_G$ is spanned by $xx$, $xy+\alpha^2yx$ and $xz+\alpha zx-\alpha^2yy+pyx$, where $p=(\alpha-\alpha^2)t-\alpha$. Now the substitution $x\to x$, $y\to y$ and $z\to vz+uy$ with appropriate $u\in\K$ and $v\in\K^*$ turns the defining relations of $A$ into $xx$, $xy+\alpha^2yx$ and $xz+\alpha zx+yy$. Thus $A$ is isomorphic to $A_F$ with $F$ from {\rm (T20)} of Theorem~$\ref{main33-t}$ with $a=\alpha^2$. By Lemma~\ref{degen}, it is non-proper.

It remains to consider the case $\alpha=1$. Solving (\ref{b1-33}), we see that
$$
\textstyle {\cal P}_{3,3}(M)=\{G_{s,t,r}=sxyz^\rcirclearrowleft{-}sxzy^\rcirclearrowleft{+}tx^2z^\rcirclearrowleft{-}sxzx
{+}\frac{s{-}t}{2}{xy^2}^\rcirclearrowleft{-}syxy{+}tx^2y{+}\frac{t{-}s}2xyx{+}rx^3:s,t,r\in\K\}
$$
Clearly, $G_{s,t,r}$ is non-degenerate precisely when $(s,t)\neq(0,0)$. By Remark~\ref{rer}, two such twisted potentials are equivalent if and only if they are obtained from one another by a linear substitution with the matrix, whose transpose commutes with $M$. That is, we have to look only at substitutions $x\to ux$, $y\to u(vx+y)$, $z\to u(wx+vy+z)$ with $u\in\K^*$, $v,w\in\K$. A direct computation shows that this sub transforms $G_{s,t,r}$ to $G_{s',t',r'}$ with $s'=u^3s$, $t'=u^3t$ and $r'=u^3(r+(3t-s)(w+\frac12v-\frac12v^2))$. If $s=0$, $G_{s,t,r}$ is cyclicly invariant and therefore $A$ is potential. Since this contradicts the assumptions, $s\neq 0$. Now the above observation shows that $G_{s,t,r}$ is equivalent to precisely one of the following: $G_{1,t,0}$ for $t\neq \frac13$ or $G_{1,1/3,r}$ for $r\in\K$. Swapping of $x$ and $z$ brings the latter to the forms (T3) or (T4) of Theorem~$\ref{main33-t}$. By Lemma~\ref{firsttouch}, these algebras are proper. Since we already know that these twisted potentials are pairwise non-equivalent, Lemma~\ref{gen2} implies that the corresponding algebras are pairwise non-isomorphic.
\end{proof}

\begin{lemma}\label{2block} Let $G\in{\cal P}^*_{3,3}$ be non-degenerate, $M\in GL_3(\K)$ be the unique matrix providing the twist for $G$ and assume that $A=A_G$ is non-potential and the normal Jordan form of $M$ consists of two blocks. If $A$ is non-proper, then $A$ is isomorphic to $A_F$ with $F$ from {\rm (T20)} of Theorem~$\ref{main33-t}$. If $A$ is proper, then $A$ is isomorphic to $A_F$ for $F$ from {\rm (T5--T8)} of Theorem~$\ref{main33-t}$. Moreover, algebras $A_F$ for $F$ with different labels from {\rm (T5--T8)} are non-isomorphic and the isomorphism conditions of Theorem~$\ref{main33-t}$ concerning each of {\rm (T5--T8)} are satisfied.
\end{lemma}

\begin{proof} By Remark~\ref{rer}, we can without loss of generality assume that
$$
M=\left(\begin{array}{ccc}\alpha&1&0\\ 0&\alpha&0\\ 0&0&\beta\end{array}\right)\ \ \text{with $\alpha,\beta\in\K^*$.}
$$
If $G=\sum\limits_{j,k,m=1}^3 a_{j,k,m}x_jx_kx_m$, then the inclusion $F\in {\cal P}_{3,3}(M)$ is equivalent to the following system of linear equations on the coefficients of $G$:
\begin{equation}\label{b2-33}
a_{j,k,3}=\beta a_{3,j,k},\ \
a_{j,k,2}=\alpha a_{2,j,k}\ \ \text{and}\ \ a_{j,k,1}=\alpha a_{1,j,k}+a_{2,j,k}\ \ \text{for}\ \ 1\leq j,k\leq 3.
\end{equation}
One easily sees that (\ref{b2-33}) has only zero solution if $1\notin\{\alpha,\beta,\alpha^2\beta,\alpha\beta^2\}$. Furthermore, ${\cal P}_{3,3}(M)$ contains no non-degenerate elements unless $\alpha^2\beta=1$. Indeed, if $\alpha^2\beta\neq 1$, $y$ does not feature at all in members of ${\cal P}_{3,3}(M)$.
Thus for the rest of the proof, we can assume that $\alpha^2\beta=1$. That is, $\beta=\alpha^{-2}$. By (\ref{b2-33}),
$$
F_{s,t}=s(xyz{+}\alpha yzx{+}\alpha^2zxy{-}\alpha xzy{-}yxz{-}\alpha^2zyx{-}xzx)+t(xxz{+}\alpha^2zxx{+}\alpha xzx)\in {\cal P}_{3,3}(M)\ \ \text{for $s,t\in\K$}.
$$
Furthermore, there are no other elements in ${\cal P}_{3,3}(M)$ unless $\alpha^3=1$ or $\alpha^2=1$. If $s=0$, $F_{s,t}$ is degenerate. Thus, we can assume that $s\neq 0$. By scaling, we can make $s=1$. By Remark~\ref{rer}, two  twisted potentials $F_{1,t}$ and $F_{1,t'}$ are equivalent precisely when they are obtained from one another by a linear substitution with the matrix, whose transpose commutes with $M$. That is, in the case $\alpha\neq\beta$ (equivalently, $\alpha^3\neq 1$), we have to look only at substitutions $x\to ux$, $y\to u(vx+y)$, $z\to wz$ with $u,w\in\K^*$, $v\in\K$. A direct computation shows that this sub transforms $F_{1,t}$ to $F_{1,t'}$ if and only if $t=t'$. That is, in the case $\alpha^3\neq 1$, $F_{1,t}$ are pairwise non-equivalent. Swapping of $x$ and $z$ turns $F_{1,a}$ into
$$
G_{a,b}=zyx+byxz+b^2xzy-bzxy-yzx-b^2xyz+(ab-1)zxz+azzx+ab^2xzz\ \ \text{with $a,b\in\K$, $b^3\neq 1$},
$$
where $b=\alpha$ and $a=t$, which is precisely the twisted potential from (T5) with $b^3\neq 1$. By Lemma~\ref{firsttouch}, these algebras are proper. Since we already know that the corresponding twisted potentials are non-equivalent, Lemma~\ref{gen2} implies that the algebras themselves are pairwise non-isomorphic.

It remains to consider the cases $\alpha=-1$, $\alpha=1$ and $\alpha^3=1\neq\alpha$. If $\alpha=-1$, then $\beta=1$. By (\ref{b2-33}),
$$
{\cal P}_{3,3}(M)=\{F_{s,t,r}=s(xyz{-}yzx{+}zxy{+}xzy{-}yxz{-}zyx{-}xzx)+t(xxz{+}zxx{+}xzx)+rzzz:s,t,r\in\K\}.
$$
Since $F_{s,t,r}$ is degenerate for $s=0$, we can assume that $s\neq 0$. If $r=0$, we are back to the previous considerations (with $\alpha=-1$). Thus we can assume that $r\neq 0$. By scaling, we can make $s=r=1$, which leaves us with $F_{1,t,1}$. Same argument as above shows that $F_{1,t,1}$ are pairwise non-equivalent. Swapping of $x$ and $y$ turns $F_{1,a,1}$ into
$$
G_a=yxz-xzy+zyx+yzx-xyz-zxy+(a-1)yzy+ayyz+azyy+zzz\ \ \text{with $a\in\K$,}
$$
which is precisely the twisted potential from (T6). By Lemma~\ref{firsttouch}, the corresponding algebras are proper. Since we already know that these twisted potentials are pairwise non-equivalent, Lemma~\ref{gen2} implies that the corresponding algebras are pairwise non-isomorphic.

Next, consider the case $\alpha^3=1\neq\alpha$. Then $\beta=\alpha$. Solving (\ref{b2-33}), we see that ${\cal P}_{3,3}(M)$ consists of
$$
\begin{array}{c}F_{s,t,p,q}=s(xyz+\alpha yzx+\alpha^2zxy-\alpha xzy-yxz-\alpha^2 zyx-xzx)+p(xxz+\alpha^2 zxx+\alpha xzx)\\  \textstyle
+t(xzz+\alpha^2zxz+\alpha zzx)+q(xxy+\alpha^2yxx+\alpha xyx+\frac{\alpha^2}{1-\alpha}xxx)\ \ \ \text{with $s,t,p,q\in\K$.}\end{array}
$$
The only linear substitutions with the matrix, whose transpose commutes with $M$ have the form $x\to ux$, $y\to vx+uy+wz$, $z\to cz+dx$ with $u,c\in\K^*$ and $v,w,d\in\K$. A direct computation shows that this substitution transforms $F_{s,t,p,q}$ to $F_{s',t',p',q'}$ with $s'=su^2c$, $t'=tuc^2+s(1-\alpha)ucw$, $q'=qu^3+s(\alpha^2-\alpha)u^2d$ and $p'=pu^2c+qu^2w-t\alpha udc+s(\alpha^2-\alpha)udw$. If $s=0$ to begin with, a substitution of the above form allows to kill $p$ (make $p'=0$) unless $q=t=0$. In the latter case $F_{s,t,p,q}=F_{0,0,p,0}$ is degenerate. This leaves $F_{0,t,0,q}$. If $tq=0$, then again $F$ is degenerate. Thus $tq\neq 0$. A sub of the above form then allows to turn $t$ and $q$ into $1$. For $G=F_{0,1,0,1}$, the space of defining relations is spanned by $xx$, $xz+\alpha^2zx$ and $xy+\alpha yx+zz$. Swapping $y$ and $z$ now provides an isomorphism of $A$ and an algebra from (T20) with $a=\alpha^2$. It remains to consider the case $s\neq 0$. A substitution of the above form now can be chosen in such a way that $s'=1$ and $q'=t'=0$. Thus it remains to consider the case $G=F_{1,0,a,0}$ with $a\in\K$. It is easy to see that the above substitutions can not transform $F_{1,0,a,0}$ into $F_{1,0,a',0}$ with $a\neq a'$: $F_{1,0,a,0}$ are pairwise non-equivalent. Swapping $x$ and $z$ turns $F_{1,0,a,0}$ into
$$
G=zyx+\alpha yxz+\alpha^2xzy-\alpha zxy-yzx-\alpha^2 xyz+(\alpha a-1)zxz+azzx+\alpha^2 axzz,
$$
which are exactly the twisted potentials from (T5) with $b^3=1\neq b$.

It remains to consider the case $\alpha=\beta=1$. By (\ref{b2-33}),
$$
\textstyle
{\cal P}_{3,3}(M)=\{F_{s,t,p,q,r}=s(xyz^\rcirclearrowleft{-}xzy^\rcirclearrowleft{-}xzx)+pxxz^\rcirclearrowleft+txzz^\rcirclearrowleft+qxxx+rzzz:s,t,p,q,r\in\K\}.
$$
As above, the only linear substitutions with the matrix, whose transpose commutes with $M$ have the form $x\to ux$, $y\to vx+uy+wz$, $z\to cz+dx$ with $u,c\in\K^*$ and $v,w,d\in\K$. A direct computation shows that this substitution transforms $F_{s,t,p,q,r}$ to $F_{s',t',p',q',r'}$ with $s'=su^2c$,
$t'=tuc^2+rc^2d$, $q'=qu^3+rd^3+(3p-s)u^2d+3tud^2$, $p'=pu^2c+2tudc+rcd^2$ and $r'=rc^3$. If $s=0$, $F_{s,t,p,q,r}$ is cyclicly invariant and therefore the corresponding algebra is potential. Thus we can assume $s\neq 0$. If $r\neq 0$, we can find a substitution of the above shape such that $t'=0$ and $s'=r'=1$. Thus we have to consider $F_{1,0,p,q,1}$. Analyzing the action of the above substitutions on these, wee see that $F_{1,0,p,q,1}$ and $F_{1,0,p',q',1}$ are equivalent if and only if $p'=p$ and $q'=\pm q$. After swapping $x$ and $y$, we arrive to twisted potentials
$$
G_{a,b}=xzy^\rcirclearrowleft-xyz^\rcirclearrowleft-yzy+ayyz^\rcirclearrowleft+by^3+z^3
$$
with $a,b\in\K$, which are precisely the twisted potentials from (T7). By Lemma~\ref{firsttouch}, the corresponding twisted potential algebras are proper. Since we already know when their twisted potentials are equivalent, Lemma~\ref{gen2} implies the isomorphism condition for (T7). It remains to consider the case $s\neq 0$ and $r=0$. The case $t=0$ yields algebras from (T5) (with $b=1$, $a\neq0$) (after a scaling and swapping $x$ with $z$). Thus we can assume that $t\neq 0$. Now we can easily find a substitution of the above form for which $s'=t'=1$ and $r'=p'=0$.  Thus we have to consider $F_{1,1,0,q,0}$. Analyzing the action of the above substitutions on these, wee see that $F_{1,1,0,q,0}$ and $F_{1,1,0,q',0}$ are pairwise non-equivalent. After swapping $x$ and $y$, we arrive to twisted potentials
$$
G_{a}=xzy^\rcirclearrowleft-xyz^\rcirclearrowleft-yzy+yzz^\rcirclearrowleft+ay^3
$$
with $a\in\K$. If $a=0$, we are back to (T5). Thus we can assume that $a\neq 0$. Now we have precisely the twisted potentials from (T8). By Lemma~\ref{firsttouch}, the corresponding algebras are proper. Since we already know that their twisted potentials are pairwise non-equivalent, Lemma~\ref{gen2} implies these algebras are pairwise non-isomorphic.
\end{proof}

\begin{lemma}\label{3block-det1} Let $G\in{\cal P}^*_{3,3}$ be non-degenerate, $M\in GL_3(\K)$ be the unique matrix providing the twist for $G$ and assume that $A=A_G$ is non-potential. Assume also that $M$ is diagonalizable and has determinant $1$. If $A$ is non-proper, then $A$ is isomorphic to $A_F$ with $F$ from {\rm (T20)} of Theorem~$\ref{main33-t}$. If $A$ is proper, then $A$ is isomorphic to $A_F$ for $F$ from {\rm (T1--T2)} or {\rm (T9--T10)} of Theorem~$\ref{main33-t}$ with different labels corresponding to non-isomorphic algebras. Furthermore, the relevant isomorphism statements of Theorem~$\ref{main33-t}$ hold.
\end{lemma}

\begin{proof} By Remark~\ref{rer}, we can without loss of generality assume that
$$
M=\left(\begin{array}{ccc}\alpha&0&0\\ 0&\beta&0\\ 0&0&\gamma\end{array}\right)
$$
with $\alpha,\beta,\gamma\in\K^*$. If $G=\sum\limits_{j,k,m=1}^3 a_{j,k,m}x_jx_kx_m$, then the inclusion $F\in {\cal P}_{3,3}(M)$ is equivalent to the following system of linear equations on the coefficients of $G$:
\begin{equation}\label{b3-33}
a_{j,k,3}=\gamma a_{3,j,k},\ \
a_{j,k,2}=\beta a_{2,j,k}\ \ \text{and}\ \ a_{j,k,1}=\alpha a_{1,j,k}\ \ \text{for}\ \ 1\leq j,k\leq 3.
\end{equation}
Since $M$ has determinant $1$, we have $\alpha\beta\gamma=1$. Since we are not interested in potentials, $(\alpha,\beta)\neq (1,1)$.

Analyzing (\ref{b3-33}), we see that
$$
F_{s,t}=s(xyz+\alpha yzx+\alpha\beta zxy)+t(yxz+\beta xzy+\alpha\beta zyx)\in {\cal P}_{3,3}(M)\ \ \text{for $s,t\in\K$}.
$$
Furthermore, there are no other elements in ${\cal P}_{3,3}(M)$ unless either $1$ is among the eigenvalues or at least two of the eigenvalues are equal.

If $s=t=0$, then $F_{s,t}$ is degenerate. If $st=0$ and $(s,t)\neq (0,0)$, then the corresponding twisted potential algebra is easily seen to be isomorphic to the algebra from (P12) and therefore is potential. Thus we can assume that $st\neq 0$. By scaling, we can make $s=1$. Then $G=F_{1,t}$ acquires the form (T1) with $a=t$, $b=\frac t\alpha$ and $c=\beta t$. Since $(\alpha,\beta)\neq(1,1)$, we have the condition $(a-b,a-c)\neq (0,0)$ of (T1). By Lemma~\ref{firsttouch}, the algebras from (T1) are proper. By Lemma~\ref{gen2}, two algebras from (T1) are isomorphic precisely when their twisted potentials are equivalent. Using Remark~\ref{rer}, we see that if the eigenvalues of $M$ are pairwise distinct, then the only substitutions transforming a corresponding $F$ from (T1) to another $F$ from (T1) are scalings composed with permutations of the variables. The isomorphism condition in (T1) is now easily verified. The case when some of the eigenvalues coincide leads to a bigger group of eligible substitutions, however the result in terms of isomorphic members of (T1) is easily seen to be the same.

It remains to consider two options for the triple $(\alpha,\beta,\gamma)$ of the eigenvalues of $M$ to which all the remaining options are reduced by a permutation of the variables: $(\alpha,\alpha^{-1},1)$ and $(\alpha^{-2},\alpha,\alpha)$ with $\alpha\in\K^*$, $\alpha\neq 1$.

Consider the case when the eigenvalues of $M$ are $(\alpha,\alpha^{-1},1)$. Solving (\ref{b3-33}), we see that
$$
F_{s,t,r}=s(xyz+\alpha yzx+zxy)+t(yxz+\alpha^{-1}xzy+zyx)+rz^3\in{\cal P}_{3,3}(M)\ \ \text{for $s,t,r\in\K$}.
$$
Furthermore, there are no other elements in ${\cal P}_{3,3}(M)$ unless $\alpha=-1$ (the case $\alpha=1$ is already off the table). If $r=0$, we fall back into the previous case. Thus we can assume $r\neq 0$. If $s=t=0$, $F_{s,t,r}$ is degenerate (and potential to boot). Thus $(s,t)\neq (0,0)$. If $st=0$, a scaling (if $t=0$) or a scaling composed with the swap of $x$ and $y$ turns $F_{s,t,r}$ into $xyz+\alpha yzx+zxy+z^3$. Now the corresponding twisted potential algebra is easily seen to be isomorphic to the potential algebra from (P14). Hence  $str\neq 0$ and by a scaling we can turn  $s$ and $r$ into $1$, leaving us with $F_{1,t,1}$, which is a scalar multiple of the twisted potential from (T2) with $a=\frac t\alpha$ and $b=t$. By Lemma~\ref{firsttouch}, the algebras from (T2) are proper. By Lemma~\ref{gen2}, algebras from (T2) are isomorphic if and only if their twisted potentials are equivalent. By Remark~\ref{rer}, this happens precisely when they can be transformed into one another by a linear substitution with the matrix whose transpose commutes with $M$. Now it is easy to verify the isomorphism condition from (T2).

Next, consider the case when the eigenvalues of $M$ are $(\alpha^{-2},\alpha,\alpha)$. Solving (\ref{b3-33}), we see that
$$
\begin{array}{c}
F_{s,t,p,q}=s(xyz+\alpha^{-2}yzx+\alpha^{-1}zxy)+t(xzy+\alpha^{-2}zyx+\alpha^{-1}yxz)\\
+p(xyy+\alpha^{-2}yyx+\alpha^{-1}yxy)+q(xzz+\alpha^{-2}zzx+\alpha^{-1}zxz)\in{\cal P}_{3,3}(M)\ \ \text{for $s,t,p,q\in\K$}.
\end{array}
$$
Furthermore, there are no other elements in ${\cal P}_{3,3}(M)$ unless $\alpha=-1$ or $\alpha^3=1\neq \alpha$ (again, the case $\alpha=1$ is off). Applying Lemma~\ref{1-dim} to $syz+tzy+pyy+qzz$, we see that by a linear substitution, which leaves both $x$ and the linear span of $y,z$ invariant, $(s,t,p,q)$ can be transformed into exactly one of the following forms: $(0,0,0,0)$, $(0,0,0,1)$, $(1,t,0,0)$ with $t\in\K$ or $(1,-1,0,1)$. All cases except for the last one either give a degenerate twisted potential or one that has been already dealt with earlier in this proof. This leaves us with $(s,t,p,q)=(1,-1,0,1)$:
$$
G=xyz+\alpha^{-2}yzx+\alpha^{-1}zxy-xzy-\alpha^{-2}zyx-\alpha^{-1}yxz+xzz+\alpha^{-2}zzx+\alpha^{-1}zxz,
$$
which is the twisted potential from (T9) with $a=\alpha$. By Lemma~\ref{firsttouch}, these twisted potentials are proper. Using Lemma~\ref{gen2}, as above on a number of occasions, we see that the corresponding twisted potential algebras are pairwise non-isomorphic.

At this points it remains to deal with three specific triples of eigenvalues of $M$: $(-1,-1,1)$ and $(\alpha,\alpha,\alpha)$ with $\alpha^3=1\neq \alpha$. We start with the case when the eigenvalues of $M$ are $(-1,-1,1)$. Solving (\ref{b3-33}), we see that ${\cal P}_{3,3}(M)$ is the space of
$$
F_{s,t,p,q,r}=s(xyz-yzx+zxy)+t(yxz-xzy+zyx)+p(xxz-xzx+zxx)+q(yyz-yzy+zyy)+rzzz
$$
with $s,t,p,q,r\in\K$. Applying Lemma~\ref{1-dim} to $sxy+tyx+pxx+qyy$, we see that by a linear substitution, which leaves both $z$ and the linear span of $x,y$ invariant, $(s,t,p,q)$ can be transformed into exactly one of the following forms: $(0,0,0,0)$, $(0,0,0,1)$, $(1,t,0,0)$ with $t\in\K$ or $(1,-1,0,1)$. If either $r=0$ or any of the first three of the last four cases occurs, then either our $F$ is degenerate or it is a twisted potential that has been already dealt with earlier in this proof (up to a possible permutation of variables). An additional scaling allows to turn $r$ into $1$ leaving us with
$$
G=xyz-yzx+zxy-yxz+xzy-zyx+yyz-yzy+zyy+zzz,
$$
which is the twisted potential from (T10). By Lemmas~\ref{firsttouch}, the corresponding twisted potential algebra is proper.

This leaves us with the final case when the eigenvalues of $M$ are $(\alpha,\alpha,\alpha)$ with $\alpha^3=1\neq \alpha$. For the sake of convenience, we use the following notation: $\cial{uvw}=uvw+\alpha vwu+\alpha^2wuv$. Note that $\cial{u^3}=0$ since $1+\alpha+\alpha^2=0$.  Solving (\ref{b3-33}), we see that
$$
F_u=u_1\cial{x^2y}+u_2\cial{x^2z}+u_3\cial{y^2x}+u_4\cial{y^2z}+u_5\cial{z^2x}+u_6\cial{z^2y}+u_7\cial{xyz}+u_8\cial{xzy}
$$
for $u=(u_1,\dots,u_8)\in\K^8$ comprise ${\cal P}_{3,3}(M)$. Since $M$ is central in $GL_3(\K)$, every linear substitution preserves this general form of a twisted potential, changing the coefficients however. First, we shall verify that there always is a linear substitution, which kills $u_1$ and $u_2$ (=turns both of them into $0$). If $u_5=u_6=0$, then swapping $x$ and $z$ achieves the objective. Thus we can assume that $(u_5,u_6)\neq(0,0)$. If $u_5=0$, the substitution $x\to x$, $y\to y+x$, $z\to z$ makes both $u_5$ and $u_6$ non-zero. If $u_6=0$, the substitution $x\to x+y$, $y\to y$, $z\to z$ makes both $u_5$ and $u_6$ non-zero. Thus we can assume $u_5u_6\neq 0$. Now the substitution $x\to x$, $y\to y$, $z\to z+\frac{u_2}{u_5}x+\frac{u_4}{u_6}y$ is easily seen to kill both $u_2$ and $u_4$, while leaving $u_5$ and $u_6$ unchanged. Thus we can assume that $u_2=u_4=0$ and $u_5u_6\neq 0$. If $u_1=0$, the job is already done. Thus we can assume $u_1\neq 0$. If $u_3=0$, then swapping $x$ and $y$ we turn both $u_1$ and $u_2$ into zero. Thus we can assume that $u_3\neq 0$. Performing a scaling, we can turn both $u_1$ and $u_3$ into $1$, while the conditions $u_2=u_4=0$ and $u_5u_6\neq 0$ remain unaffected. Thus we have $u_1=u_3=1$, $u_2=u_4=0$ and $u_5u_6\neq 0$. Using the fact that $\K$ is algebraically closed, we can find $s,t\in\K$ such that $w_1=1-\alpha^2s+(u_8+\alpha^2u_7)t+u_6t^2=0$ and $w_2=-u_6st-u_5t+(u_7+\alpha^2u_8)=0$. Indeed, the first equation amounts to expressing $s$ in terms of $t$. Plugging this into the second equation yields a genuinely cubic equation on $t$: the $t^3$-coefficient is $-\alpha^2u_6\neq 0$. Now the substitution $x\to x$, $y+sx$, $z\to z+tx$ transforms $(u_1,u_2)$ into $(w_1,w_2)$ thus killing both $u_1$ and $u_2$. That is, no matter the case, a linear substitution kills both $u_1$ and $u_2$. By Lemma~\ref{1-dim} applied to $f=u_3yy+u_5zz+u_7\alpha yz+u_8\alpha zy$, there is a linear substitution on the variables $y,z$ turning $f$ into (exactly) one of the following four forms: $0$, $zz$, $yz-azy$ with $a\in\K$ or $yz-zy+zz$. The same substitution augmented by $x\to x$ transforms $F_u$ into one of the following forms:
$$
\begin{array}{l}
G_1=p\cial{y^2z}+q\cial{z^2y};
\\
G_2=\cial{z^2x}+p\cial{y^2z}+q\cial{z^2y};
\\
G_3=\cial{yzx}-a\cial{zyx}+p\cial{y^2z}+q\cial{z^2y};
\\
G_4=\cial{yzx}-\cial{zyx}+\cial{z^2x}+p\cial{y^2z}+q\cial{z^2y}
\end{array}\qquad \text{with $a,p,q\in\K$.}
$$
We can disregard $G_1$, since it is degenerate. The twisted potential $G_2$ is degenerate if $p=0$. Thus we can assume that $p\neq 0$. The substitution $x\to x-qy$, $y\to y$, $z\to z$ kills $q$ in $G_2$. A scaling turns $p$ into $1$ yielding the twisted potential $\cial{z^2x}+\cial{y^2z}$, which falls into (T20) after a permutation of variables.
As for $G_3$, if $a\neq \alpha$, a substitution $x\to x+sy$, $y\to y$, $z\to z$ with an appropriate $s\in\K$ kills $p$, while if $a\neq \alpha^2$, a substitution $x\to x+sz$, $y\to y$, $z\to z$ with an appropriate $s\in\K$ kills $q$. In any case, we can assume that $pq=0$, which lands us (up to a permutation of variables) into cases already considered above in this very proof. Finally, a substitution $x\to x+sy+tz$, $y\to y$, $z\to z$ with an appropriate $s,t\in\K$, applied to $G_4$, kills both $p$ and $q$ and again we arrive to a situation already dealt with earlier. Annoyingly, the last case required quite a bit of work while producing no extra twisted potentials.
\end{proof}

\begin{lemma}\label{3block-detn1} Let $G\in{\cal P}^*_{3,3}$ be non-degenerate, $M\in GL_3(\K)$ be the unique matrix providing the twist for $G$ and assume that $A=A_G$ is non-potential. Assume also that $M$ is diagonalizable and has determinant different from $1$. If $A$ is non-proper, then $A$ is isomorphic to $A_F$ with $F$ from {\rm (T19--T21)} of Theorem~$\ref{main33-t}$. If $A$ is proper, then $A$ is isomorphic to $A_F$ for $F$ from {\rm (T11--T18)} of Theorem~$\ref{main33-t}$ with different labels corresponding to non-isomorphic algebras. Furthermore, the relevant isomorphism statements of Theorem~$\ref{main33-t}$ hold.
\end{lemma}

\begin{proof} Applying Remark~\ref{rer} in the same way as in the last proof, we can assume that $M$ is diagonal with $\alpha,\beta,\gamma\in\K^*$ on the main diagonal. Since the determinant of $M$ is different from $1$, we have $\alpha\beta\gamma\neq 1$. Analyzing (\ref{b3-33}), we see that ${\cal P}_{3,3}(M)$ contains only degenerate twisted potentials unless the eigenvalues of $M$ in some order are $(\alpha,\alpha^{-2},1)$ with $\alpha\neq 1$, or $(\alpha,-\alpha,\alpha^{-2})$, or $(\alpha,\alpha^{-2},\alpha^4)$ with $\alpha^3\neq 1$ (everywhere $\alpha\in\K^*$).

First, assume that the eigenvalues of $M$ are $(\alpha,\alpha^{-2},1)$ with $\alpha\neq 1$. Solving (\ref{b3-33}), we see that
$$
F_{s,t}=s(xxy+\alpha xyx+\alpha^2 yxx)+tz^3\in{\cal P}_{3,3}(M)\ \ \text{for $s,t\in\K$}.
$$
Furthermore, there are no other elements in ${\cal P}_{3,3}(M)$ unless $\alpha^4=1$ or $\alpha^3=1$.
If $st=0$, then $F_{s,t}$ is degenerate. Thus we can assume that $st\neq 0$. By scaling, we can make $s=t=1$. That is, $G$ is equivalent to $F_{1,1}$, which falls into (T19) and is non-proper according to Lemma~\ref{degen}.

Assume now that the eigenvalues of $M$ are $(\alpha,-\alpha,\alpha^{-2})$. According to (\ref{b3-33}),
$$
F_{s,t}=s(xxz+\alpha xzx+\alpha^2 zxx)+t(yyz-\alpha yzy+\alpha^2zyy)\in{\cal P}_{3,3}(M)\ \ \text{for $s,t\in\K$}.
$$
Furthermore, there are no other elements in ${\cal P}_{3,3}(M)$ unless $\alpha^6=1$. If $st=0$, then $F_{s,t}$ is degenerate and we can assume that $st\neq 0$. By scaling, we can make $s=t=1$. That is, $G$ is equivalent to $F_{1,1}$, which falls into (T11) with $a=\alpha$. By Lemma~\ref{firsttouch}, the corresponding algebras are proper. Since the isomorphic proper twisted potential algebras must have conjugate twists, the isomorphism of two algebras from (T11) corresponding to parameters $a$ and $a'$ is only possible if $a'=a$ or $a'=-a$. In the latter case the swap of $x$ and $y$ provides a required isomorphism.

Next, assume that the eigenvalues of $M$ are $(\alpha,\alpha^{-2},\alpha^4)$ with $\alpha^3\neq 1$. Solving (\ref{b3-33}), we see that
$$
F_{s,t}=s(xxy+\alpha xyx+\alpha^2 yxx)+t(\alpha^4yyz+\alpha^2 yzy+zyy)\in {\cal P}_{3,3}(M)\ \ \text{for $s,t\in\K$}.
$$
Furthermore, there are no other elements in ${\cal P}_{3,3}(M)$ unless $\alpha^6=1$ or $\alpha^4=1$ or $\alpha^9=1$. If $st=0$, then $F_{s,t}$ is degenerate and we can assume that $st\neq 0$. By scaling, we can make $s=t=1$. That is, $G$ is equivalent to $F_{1,1}$, which falls into (T20) after a permutation of variables. By Lemma~\ref{degen}, the corresponding algebra is non-proper.

It remains to deal with few specific triples of eigenvalues of $M$: $(-1,1,1)$, $(\alpha,\alpha,-\alpha)$ with $\alpha^3=1\neq \alpha$, $(\alpha,\alpha,1)$ with $\alpha^3=1\neq\alpha$, $(\alpha,-1,1)$ with $\alpha^2=-1$ and $(\alpha,\alpha^4,\alpha^7)$ with $\alpha^9=1\neq\alpha^3$. This are the triples for which there are more solutions than in the generic cases considered above.

First, assume that the eigenvalues of $M$ are $(\alpha,\alpha,1)$ with $\alpha^3=1\neq\alpha$. By (\ref{b3-33}),
$$
{\cal P}_{3,3}(M)=\{F_{s,t,r}=s(xxy+\alpha xyx+\alpha^2 yxx)+t(yyx+\alpha yxy+\alpha^2 xyy)+rz^3:s,t,r\in\K\}.
$$
One easily sees that a linear substitution leaving both $z$ and the space spanned by $x$ and $y$ invariant can be chosen to kill $t$. This places the twisted potential within the framework of the very first case considered in this proof.

Next, assume that the eigenvalues of $M$ are $(\alpha,\alpha,-\alpha)$ with $\alpha^3=1\neq\alpha$. Solving (\ref{b3-33}), we see that
$$
F_{s,t,p,q}=s(xxy+\alpha xyx+\alpha^2 yxx)+t(yyx+\alpha yxy+\alpha^2 xyy)+p(zzx-\alpha zxz+\alpha^2 xzz)+q(zzy-\alpha zyz+\alpha^2 yzz)
$$
with $s,t,p,q\in\K$ comprise the space ${\cal P}_{3,3}(M)$. One can easily verify that a linear substitution leaving both $z$ and the space spanned by $x$ and $y$ invariant can be chosen to kill either $t$ and $p$ or $s$ and $p$. In both cases we fall into situations already dealt with in this proof.

Now assume that the eigenvalues of $M$ are $(\alpha,-1,1)$ with $\alpha^2=-1$. According to (\ref{b3-33}),
$$
{\cal P}_{3,3}(M)=\{F_{s,t,r}=s(xxy+\alpha xyx-yxx)+t(yyz-yzy+zyy)+rz^3:s,t,r\in\K\}.
$$
If $rt=0$, we are back to the already considered cases. If $s=0$, our twisted potential is degenerate. Thus we can assume $str\neq 0$. By a scaling we can make $s=t=r=1$. That is, $G$ is equivalent to $F_{1,1,1}$. By swapping $x$ and $z$, we see that $G$ is equivalent to the twisted potential
$$
F=zzy+\alpha zyz-yzz+yyx-yxy+xyy+x^3\ \ \text{with $\alpha=\pm i$},
$$
which are the two twisted potentials from (T12) and (T13). By Lemma~\ref{firsttouch}, they are proper.

Next, assume that the eigenvalues of $M$ are $(\alpha,\alpha^4,\alpha^7)$ with $\alpha^9=1\neq\alpha^3$. By (\ref{b3-33}),
$$
{\cal P}_{3,3}(M)=\{F_{s,t,r}=s(xxz{+}\alpha xzx{+}\alpha^2 zxx)+t(yyx{+}\alpha^4yxy{+}\alpha^8xyy)+r(zzy{+}\alpha^7zyz{+}\alpha^5yzz):s,t,r\in\K\}.
$$
If $str=0$, we are back to the already considered cases and we know that the corresponding twisted potential algebra is non-proper. If $str\neq 0$, by a scaling we can make $s=t=r=1$. Thus $G$ in this case is equivalent to $F_{1,1,1}$. An easy computation shows that this time the corresponding twisted potential algebra is proper. Note that the assumption $\alpha^9=1\neq\alpha^3$ is the same as $\alpha\in\{\xi_9,\xi_9^2,\xi_9^4,\xi_9^5,\xi_9^7,\xi_9^8\}$. Since cyclic permutations of $x$, $y$ and $z$ provide equivalence of $F_{1,1,1}$ for $\alpha\in\{\xi_9,\xi_9^4,\xi_9^7\}$ as well as for $\alpha\in\{\xi_9^2,\xi_9^5,\xi_9^8\}$, we have just two twisted potentials to deal with in this case: $F_{1,1,1}$ for $\alpha=\xi_9$ and $F_{1,1,1}$ for $\alpha=\xi^2_9$. By Lemma~\ref{firsttouch}, the algebras in (T14) and (T15) are proper and their respective twists have eigenvalues $\xi_9,\xi_9^4,\xi_9^7$ and $\xi_9^2,\xi_9^5,\xi_9^8$. Thus they are isomorphic to $F_{1,1,1}$ for $\alpha=\xi_9$ and $\alpha=\xi_9^2$ respectively.

It remains to deal with the final case when the eigenvalues of $M$ are $(-1,1,1)$. By (\ref{b3-33}),
$$
{\cal P}_{3,3}(M)=\{G_{w}=w_1y^3\!+w_2{y^2z}^\rcirclearrowleft\!\!+w_3{yz^2}^\rcirclearrowleft\!\!+w_4z^3\!+w_5(x^2y{-}xyx{+}yx^2)+w_6(x^2z{-}xzx{+}zx^2):w\in\K^6\}.
$$
If $w_5=w_6=0$, $G_w$ is degenerate. Thus we can assume that $(w_5,w_6)\neq (0,0)$. Now it is easy to see that a substitution leaving both $x$ and the linear span of $y,z$ intact preserves the form of $G_w$ and turns $(w_5,w_6)$ into $(0,1)$. Now we have only to consider
$$
F_{u}=u_1y^3+u_2{y^2z}^\rcirclearrowleft+u_3{yz^2}^\rcirclearrowleft+u_4z^3+x^2z-xzx+zx^2\ \ \ \text{with $u=(u_1,\dots,u_4)\in\K^4$.}
$$
The only substitutions which preserve this general shape of a twisted potential are given by $x\to sx$, $z\to s^{-2}z$, $y\to py+qz$ with $s,p\in\K^*$, $q\in\K$. This substitution transforms $F_u$ into $F_{u'}$ with $u'_1=p^3u_1$, $u'_2=p^2s^{-2}u_2+p^2qu_1$, $u'_3=ps^{-4}u_3+2pqs^{-2}u_2+pq^2u_1$ and $u'_4=s^{-6}u_4+3s^{-4}qu_3+3s^{-2}q^2u_2+q^3u_1$. Now it is easy to see that a general $F_u$ can be transformed into one of the following forms $F_{1,0,1,a}$ with $a\in\K$, $F_{1,0,0,1}$, $F_{1,0,0,0}$, $F_{0,1,0,1}$, $F_{0,1,0,0}$, $F_{0,0,1,0}$, $F_{0,0,0,1}$ and $F_{0,0,0,0}$ among which there are no equivalent ones except for $F_{1,0,1,a}$ being equivalent to $F_{1,0,1,-a}$ for $a\in\K$. Among these $F_{0,0,0,1}$ and $F_{0,0,0,0}$ are degenerate, while  $F_{0,1,0,0}$, $F_{0,0,1,0}$ and $F_{1,0,0,0}$ fall into the cases already considered in this proof. This leaves us to deal with $F_{1,0,1,a}$ with $a\in\K$, $F_{1,0,0,1}$ and $F_{0,1,0,1}$. First, $F_{1,0,0,1}=y^3+z^3+x^2z-xzx+zx^2$ is non-proper and features as (T21). Next,
$$
G=F_{0,1,0,1}={y^2z}^\rcirclearrowleft+z^3+x^2z-xzx+zx^2
$$
features as (T16) and is proper by Lemma~\ref{firsttouch}. Since $F_{1,0,1,a}$ and $F_{1,0,1,-a}$ are equivalent, the case of $G=F_{1,0,1,a}$ with $a^2+4=0$ reduces to $G=F_{1,0,1,2i}=y^3+{yz^2}^\rcirclearrowleft+2iz^3+x^2z-xzx+zx^2$. The substitution $x\to z$, $z\to ix$, $y\to x+y$ followed by an appropriate scaling turns the latter into the twisted potential $y^3+{xy^2}^\rcirclearrowleft+z^2x-zxz+xz^2$ of (T17), which is proper by Lemma~\ref{firsttouch}. This leaves only $G=F_{1,0,1,a}$ with $a^2+4\neq 0$:
$$
G=F_{1,0,1,a}=y^3+{yz^2}^\rcirclearrowleft+az^3+x^2z-xzx+zx^2\ \ \text{with $a\in\K$, $a^2+4\neq 0$}.
$$
The latter are twisted potentials from (T18). By Lemmas~\ref{firsttouch} they are proper. Knowing which of them are equivalent justifies the isomorphism condition in (T18).

Finally, the absence of isomorphism for algebras with different labels follows from the fact that proper twisted potential algebras with non-conjugate twist can not be isomorphic.
\end{proof}

\subsection{Proof of Theorem~$\ref{main33-t}$}

Theorem~\ref{main33-t} is just an amalgamation of Lemmas~\ref{23}, \ref{degen}, \ref{firsttouch}, \ref{non-pbwb-all}, \ref{1block}, \ref{2block}, \ref{3block-det1} and \ref{3block-detn1}.

\section{Concluding remarks}

\begin{remark} \label{re1} Note that according to Theorem~\ref{main33-p}, there is only one (up to an isomorphism) proper quadratic potential algebra on three generators, which fails to be exact. Namely, it is the algebra given by (P9). Furthermore, there are exactly two non-Koszul (up to an isomorphism) quadratic potential algebras on three generators:  (P9) and (P14). By Theorem~\ref{main24-p}, there is only one (up to an isomorphism) proper cubic potential algebra on two generators, which fails to be exact: it features with the label (P23). However, we do not expect this pattern to extend to higher degrees or higher numbers of generators.
\end{remark}

\begin{remark} \label{rere1} By Theorems~\ref{main33-p} and~\ref{main24-p}, both sets $\{H_{A_F}:F\in {\cal P}_{3,3}\}$ and $\{H_{A_F}:F\in {\cal P}_{2,4}\}$ are finite. Indeed, the first set has $7$ elements, while the second has $5$ elements. By Proposition~\ref{commin2}, $\{H_{A_F}:F\in {\cal P}_{2,3}\}$ is a $3$-element set. This leads to the following question (we expect an affirmative answer).
\end{remark}

\begin{question}\label{quest2} Let $n\geq 2$ and $k\geq 3$. Is it true that the set $\{H_{A_F}:F\in {\cal P}_{n,k}\}$ is finite?
\end{question}

\begin{remark} \label{rere2} By Theorems~\ref{main33-p} and~\ref{main24-p}, $H_{A_F}$ is rational for every $F\in {\cal P}_{3,3}$ as well as for every $F\in {\cal P}_{2,4}$.
Proposition~\ref{commin2}, the same holds for $F\in {\cal P}_{2,3}$. This prompts the following question (again, we believe the answer to be affirmative).
\end{remark}

\begin{question}\label{quest3} Is it true that the Hilbert series of every degree-graded potential algebra is rational?
\end{question}

The above question resonates with the following issue. It was believed at some point that graded finitely presented algebras must have rational Hilbert series. This conjecture was disproved by Shearer \cite{she}, who produced an example of a quadratic algebra with non-rational Hilbert series. However his algebra as well as any of the later examples fail to be potential or Koszul. Note that the question whether the Hilbert series of a Koszul algebra must be rational is a long-standing open problem, see, for instance, \cite{popo}.

\begin{remark} \label{re000} By Theorems~\ref{main33-t} and~\ref{main24-t}, every non-potential proper twisted potential algebra $A_F$ with $F\in{\cal P}^*_{3,3}\cup{\cal P}^*_{2,4}$ is exact. Furthermore, every non-potential twisted potential algebra $A_F$ with $F\in{\cal P}^*_{3,3}$ is Koszul.
\end{remark}


\bigskip 

{\bf Acknowledgements.} \ We are grateful to IHES and MPIM for hospitality, support, and excellent research atmosphere.
This work was partially funded by the ERC grant 320974, EPSRC grant EP/M008460/1 and the ESC
Grant N9038.

\vfill\break

\small\rm

%
%
%
%
%
%
%


\begin{thebibliography}{99}

\itemsep=-2pt

\bibitem{AS}M.~Artin and W.~Schelter, \it Graded algebras of global dimension $3$, \rm
Adv. in Math. \bf66\rm\ (1987), 171--216

\bibitem{TVB}M.~Artin, J.~Tate and M.~Van~den~Bergh, \it Some algebras associated to automorphisms of elliptic curves, \rm The Grothendieck Festschrift {\bf I}, 33--85, Progr. Math. {\bf 86}, Birkh\"auser, Boston 1990

\bibitem{yy}R.~Bocklandt, \it Graded Calabi Yau algebras of dimension $3$, \rm J. Pure and Applied Algebra \bf 212\rm (2008), 14--32

\bibitem{BW}R.~Bocklandt, T.~Schedler and M.~Wemyss, \it Superpotentials and higher order derivations. \rm J. Pure Appl. Algebra {\bf 214} (2010), no. 9, 1501--1522

\bibitem{WEM}W.~Donovan and M.~Wemyss, \it Noncommutative deformations and flops, \rm Duke Math. J. {\bf 165} (2016), 1397--1474

\bibitem{dr}V.~Drinfeld, \it On quadratic quasi-commutational relations in
quasi-classical limit, \rm Selecta Math. Sovietica {\bf 11} (1992), 317--326

\bibitem{DV1}M.~Dubois-Violette, \it Multilinear forms and graded algebras,\rm J. Algebra \bf317\rm\ (2007), 198–-225

\bibitem{DV2}M.~Dubois-Violette, \it Graded algebras and multilinear forms. C. R. Math. Acad. Sci. Paris\bf\ 341\rm\ (2005), 719-–724

\bibitem{xx}V.~Ginzburg, \it Calabi Yau algebras, \rm ArXiv:math/0612139v3, 2007

\bibitem{cub1}G.~Gurevich, \it Foundations of the theory of algebraic invariants, \rm Noordhoff, 1964

\bibitem{SKL1}N.~Iyudu and S.~Shkarin,  \it Three dimensional Sklyanin algebras and Gr\"obner bases, \rm J.Algebra,  {\bf 470} (2017), p.378-419,

\bibitem{SKL2}N.~Iyudu and S.~Shkarin, \it Sklyanin Algebras and qubic root of unity, \rm Max-Planck-Institute f\"ur Mathematic preprint series, {\bf 49} (2017)

\bibitem{AN}N.~Iyudu and A.~Smoktunowitz, \it Golod--Shafarevich type theorems and potential algebras, IMRN, 2018 and
   \rm IHES preprint M/16/19 (2016)

\bibitem{kon}M.~Kontsevich, \it Formal $($non$)$commutative symplectic geometry, \rm The Gelfand Math. Seminars (Paris 1992) 97--121, Progr. Math. \bf120\rm, Birkh\"auser, Basel, 1994

\bibitem{cub2}K.~Kraft, \it Geometrische Methoden in Invarianttheorie, \rm Friedr. Vieweg\&Sohn, Brauunschweig, 1985

\bibitem{popo}A.~Polishchuk and L.~Positselski, \it Quadratic
algebras, \rm University Lecture Series \bf37\rm\ American
Mathematical Society, Providence, RI, 2005

\bibitem{she}J.~Shearer, \it A graded algebra with non-rational Hilbert series, \rm  J. Algebra {\bf 62} (1980), 228--231

\bibitem{toda}Y.~Toda, \it Noncommutative width and Gopakumar--Vafa invariants, \rm Manuscripta Mathematica {\bf 148} (2015), 521–-533

\bibitem{Zelm}E.~Zelmanov, \it Some open problems in the theory of infinite dimensional algebras, \rm J. Korean Math. Soc. {\bf 44}\ (2007),  1185-–1195

\bibitem{Uf}V.~Ufnarovskij, \it Combinatorial and Asymptotic Methods in Algebra, \rm Encyclopaedia of Mathematical Sciences  {\bf 57}, Editors: A.~Kostrikin and I.~Shafarevich, Berlin, Heidelberg, New York: Springer-Verlag (1995), 1--196

\end{thebibliography}
\end{document}